%% file: main_updated.tex
\documentclass[11pt]{article}

\usepackage{dashrule,arydshln,empheq,amsfonts,graphicx,amsmath,amssymb,amsthm,geometry,bbm,hyperref,nicefrac,enumerate,comment,fontawesome,mathtools,cases,pict2e,picture,mathrsfs,mathabx,setspace,mathrsfs,cancel,enumitem}
\usepackage{stmaryrd}
\usepackage{float}            
\usepackage{algorithm}
\usepackage{algpseudocode}
\usepackage{booktabs}
\usepackage{siunitx}
\usepackage{mleftright}
\usepackage{accents}
\usepackage{hyperref}

\usepackage{subcaption}

\usepackage{color, colortbl}
\definecolor{Gray}{gray}{0.9}

\usepackage{pgfplots}
\usepgfplotslibrary{groupplots,dateplot}
\usetikzlibrary{patterns,shapes.arrows, calc}
\pgfplotsset{compat=newest}

\usepackage{stmaryrd}
\SetSymbolFont{stmry}{bold}{U}{stmry}{m}{n}

\newcommand{\valueFunctionPlayeri}{V^i}
\newcommand{\valueFunctionPlayerj}{V^j}

\newcommand{\Tr}{\mathrm{Tr}}
\newcommand{\gradient}{\mathrm{D}_{\x}}
\newcommand{\hessian}{\gradient^2}
\newcommand{\sumAllPlayersj}{\sum_{j\in \SetOfPlayerIndices}}

\newcommand{\ThetaOne}{\Theta_1^{j}}
\newcommand{\ThetaTwo}{\Theta_2^{j}}
\newcommand{\ThetaThree}{\Theta_3^{i,j}}
\newcommand{\ThetaThreeTilde}{\tilde{\Theta}_3^{i,j}}
\newcommand{\ThetaFour}{\Theta_4^{i,j}}
\newcommand{\ThetaFive}{\Theta_5^{i,j}}
\newcommand{\ThetaSix}{\Theta_6^{i,j}}
\newcommand{\ThetaSeven}{\Theta_7^{i,j}}
\newcommand{\ThetaEight}{\Theta_8^{i,j}}

\newcommand{\ThetaFourFiveSeven}{\Theta_{4,5,7}^{i,j}}
\newcommand{\ThetaSixEight}{\Theta_{6,8}^{i,j}}

\newcommand{\dataText}{\mathrm{data}}
\newcommand{\FP}{\scalebox{0.55}{$\mathrm{FP}$}}
\newcommand{\Bsup}{\scalebox{0.55}{$\mathrm{B}$}}
\newcommand{\Asup}{\scalebox{0.55}{$\mathrm{A}$}}
\newcommand{\inttT}{\int_{t}^{T}}

\newcommand{\intot}{\int_{0}^{t}}
\newcommand{\ds}{\, \mathrm{d}s}
\newcommand{\dWs}{\, \mathrm{d}W_s}
\newcommand{\dt}{\, \mathrm{d}t}

\newcommand{\diag}{\mathrm{diag}}
\newcommand{\LQGA}{\mathbf{A}}
\newcommand{\LQGB}{\mathbf{B}}
\newcommand{\LQGC}{\mathbf{C}}
\newcommand{\LQGAi}{A^i}
\newcommand{\LQGBi}{B^i}
\newcommand{\LQGBhat}{\hat{B}}
\newcommand{\LQGBhati}{\LQGBhat^i}

\newcommand{\LQGCi}{C^i}
\newcommand{\LQGSigma}{\mathbf{\Sigma}}
\newcommand{\LQGSigmaOnePlayer}{\Sigma}
\newcommand{\LQGSigmai}{\Sigma^i}
\newcommand{\LQGRxi}{R_{\StateVariable}^i}
\newcommand{\LQGrx}{r_{\StateVariable}}
\newcommand{\LQGRui}{R^i_{\Control}}

\newcommand{\LQGGi}{G^i}
\newcommand{\LQGgx}{g_{\StateVariable}}
\newcommand{\diffusionMatrix}{\CC}
\newcommand{\attractionMatrix}{F}
\newcommand{\controlMatrixPlayeri}{B^i}
\newcommand{\controlMatrixPlayerj}{B^j}
\newcommand{\costMatrix}{K}
\newcommand{\stateCostPlayeri}{\costMatrix^{i}_{\x}}
\newcommand{\controlCostPlayerii}{\costMatrix^{i,i}_{\Control}}
\newcommand{\controlCostPlayerij}{\costMatrix^{i,j}_{\Control}}
\newcommand{\controlCostPlayerjj}{\costMatrix^{j,j}_{\Control}}
\newcommand{\controlstateCostPlayerii}{\costMatrix^{i,i}_{\Control\x}}
\newcommand{\controlstateCostPlayerij}{\costMatrix^{i,j}_{\Control\x}}
\newcommand{\controlstateCostPlayerjj}{\costMatrix^{j,j}_{\Control\x}}
\newcommand{\Phii}{\tilde{\Phi}^{i}}
\newcommand{\Phij}{\tilde{\Phi}^{j}}
\newcommand{\PPshort}{\tilde{P}}
\newcommand{\transpose}{\top}
\newcommand{\bbR}{\mathbb{R}}
\newcommand{\spd}{\mathbb{S}_+}
\newcommand{\StateVariable}{\mathbf{x}}
\newcommand{\AllStates}{\mathbf{X}^{\AllControls}}

\newcommand{\AllStatesInitialValue}{\mathbf{X}_0}
\newcommand{\OneStateInitialValue}{X_0}

\newcommand{\FinalTime}{T}
\newcommand{\TimeStep}{h}
\newcommand{\Control}{a}
\newcommand{\ControlPlayeri}{\Control^i}

\newcommand{\OptimalControlPlayeri}{\Control^{i,*}}

\newcommand{\AllControls}{\mathbf{u}}

\newcommand{\SpaceOfAdmissibleControls}{\mathcal{U}}
\newcommand{\ControlMappingSelf}{b_+}
\newcommand{\ControlMappingOther}{b_-}
\newcommand{\CostFunctionPlayeri}{J_{\ControlPlayeri}}
\newcommand{\StateDimension}{d}
\newcommand{\NumPlayers}{N}
\newcommand{\ControlDimension}{\ell}
\newcommand{\Mdata}{M_{\dataText}}
\newcommand{\MdataYzero}{M_{\dataText, y_0}}
\newcommand{\MdataZ}{M_{\dataText, Z}}
\newcommand{\SetOfPlayerIndices}{\mathcal{I}}
\newcommand{\LFP}{L_\zeta^{\FP}}
\newcommand{\MFP}{M^{\FP}}

\DeclareMathOperator*{\argmin}{arg\,min}

\newcommand{\tnorm}[1]{\vert\!\vert\!\vert #1 \vert\!\vert\!\vert}
\newcommand{\btnorm}[1]{\big|\!\big|\!\big| #1 \big|\!\big|\!\big|}

\newcommand{\cL}{\mathcal{L}}
\newcommand{\bcL}{\boldsymbol{\mathcal{L}}}

\newcommand{\X}{\boldsymbol{X}}
\newcommand{\Z}{\boldsymbol{Z}}
\newcommand{\x}{\boldsymbol{x}}
\newcommand{\z}{\boldsymbol{z}}
\newcommand{\bal}{\boldsymbol{\alpha}}
\newcommand{\bbeta}{\boldsymbol{\beta}}

\newcommand{\bmu}{\boldsymbol{\mu}}
\newcommand{\bzeta}{\boldsymbol{\zeta}}
\newcommand{\bphi}{\boldsymbol{\phi}}
\newcommand{\bPhi}{\boldsymbol{\Phi}}

\newcommand{\ba}{\boldsymbol{a}}
\newcommand{\bb}{\boldsymbol{b}}
\newcommand{\bc}{\boldsymbol{c}}
\newcommand{\p}{\boldsymbol{p}}

\newcommand{\f}{\boldsymbol{f}}
\newcommand{\g}{\boldsymbol{g}}
\newcommand{\Y}{\boldsymbol{Y}}
\newcommand{\bH}{\boldsymbol{H}}

\newcommand{\cD}{\mathcal{D}}

\newcommand{\bV}{\boldsymbol{V}}

\newcommand{\Xal}{\boldsymbol{X}^{\boldsymbol{\alpha}}}
\newcommand{\I}{\mathcal{I}}
\newcommand{\R}{\mathbb{R}}
\newcommand{\N}{\mathbb{N}}
\newcommand{\E}{\mathbb{E}}

\renewcommand{\P}{\mathbb{P}}

\newcommand{\A}{\mathbb{A}}

\newcommand{\F}{\mathcal{F}}

\newcommand{\e}{\text{e}}
\renewcommand{\d}{\text{d}}
\newcommand{\D}{\mathcal{D}}
\newcommand{\B}{\mathcal{B}}

\newcommand{\CC}{\mathbf{C}}
\newcommand{\V}{\mathcal{V}}
\newcommand{\W}{\boldsymbol{W}}
\newcommand{\bSigma}{\boldsymbol{\Sigma}}
\usepackage[textsize=tiny]{todonotes}   

\usepackage{array}
\newcolumntype{P}[1]{>{\centering\arraybackslash}p{#1}}
\newcommand{\triple}{{\vert\kern-0.25ex\vert\kern-0.25ex\vert}}
\usepackage{mathtools}
\newcommand{\defeq}{\vcentcolon=}

\usepackage{geometry}
\usepackage{xcolor}
\newcommand{\kristoffer}[1]{\textcolor{green!50!black}{#1}}
\newcommand{\per}[1]{\textcolor{magenta!50!black}{#1}}
\newcommand{\adam}[1]{\textcolor{blue}{#1}}

\newcommand{\bekir}[1]{\textcolor{orange!80!black}{#1}}
\newcommand{\assumpt}[1]{\textcolor{magenta}{#1}}
\newcommand{\conseq}[1]{\textcolor{cyan!80!black}{#1}}

\newtheorem{lemma}{Lemma}[section]
\newtheorem{remark}{Remark}[section]

\newtheorem{assumption}{Assumption}

\newtheorem{theorem}{Theorem}[section]

\newtheorem{definition}[lemma]{Definition}
\newtheorem{proposition}{Proposition}[section]
\newtheorem{example}{Example}[section]

\theoremstyle{definition}

\newcommand{\footremember}[2]{%
    \footnote{#2}
    \newcounter{#1}
    \setcounter{#1}{\value{footnote}}%
}
\newcommand{\footrecall}[1]{%
    \footnotemark[\value{#1}]%
} 

\begin{document}
\title{Convergence of fictitious play for fully coupled FBSDEs in finite-player stochastic differential games}
\author{%
  Adam Andersson%
   \footremember{affB}{Saab AB, Gothenburg, Sweden} \footremember{affA}{Department of Mathematical Sciences, Chalmers University of Technology \& University of Gothenburg}
     \and
  Kristoffer Andersson%
    \footremember{affK}{Department of Economics, University of Verona. Email: \texttt{kristofferherbert.andersson@univr.it}}%
  \and
  Per Ljung
  \footrecall{affB}
  \footremember{affC}{Department of Electrical Engineering, Chalmers University of Technology}
  }

\maketitle

\begin{abstract}
In this article we investigate the theoretical convergence properties of the fictitious-play approximation procedure applied to coupled FBSDE systems for finite-player non-zero-sum stochastic differential games. Under one set of assumptions, the convergence is shown to be geometric. Under an additional structural assumption, the geometric convergence rate further improves to a super-exponential rate in a special class of games. To the best of our knowledge, this provides the first convergence analysis of fictitious play for fully coupled FBSDEs. A numerical experiment with a linear-quadratic interbank borrowing and lending problem confirms the geometric convergence.

\end{abstract}

\section{Introduction}

Stochastic differential games (SDG) provide a mathematical framework for modeling interacting agents or players whose actions or controls influence the evolution of a common stochastic system. In such problems, each player optimizes an individual objective while accounting for the strategies of the other players, leading to the famous Nash equilibrium solution concept. Non-zero-sum stochastic differential games arise naturally in multi-agent systems under uncertainty, where each agent optimizes an individual objective, with applications ranging from financial markets and energy systems to economics and engineering. Representative examples include interbank lending and systemic risk models~\cite{fouque2015mean}, 
applications in economics such as dynamic oligopoly and resource management problems in environmental and engineering systems~\cite{dockner2000differential}, 
multi-agent optimal execution problems with price impact~\cite{schied2017state}, 
pursuit-evasion~\cite{berkovitz1986differential, isaacs1965differential,ramachandran2012stochastic,yavin1986stochastic}
and missile guidance and control~\cite{anderson1981comparison,faruqi2017differential,pontani2008optimal,weintraub2020introduction,yavin1990application}.

Similarly to stochastic control, stochastic differential games are either solved with the stochastic maximum principle~\cite{bui2019approximation,elliott1975stochastic,ramachandran2012stochastic}, or by dynamic programming~\cite{bui2019approximation,carmona2018probabilistic,faruqi2017differential,friedman1972stochastic,kushner2002numerical,kushner2007numerical,pontani2008optimal,ramachandran2012stochastic,schied2017state, yavin1990application,yavin1986stochastic}. In this work, we focus on the latter approach. Approaching SDG via dynamic programming leads to a coupled system of Hamilton--Jacobi--Bellman (HJB) equations, the Nash PDE system, and to a coupled system of forward-backward stochastic differential equations (FBSDEs), the Nash FBSDE system. In regular settings, these two representations are equivalent, but the Nash FBSDE system is a more general representation for regimes with lower regularity. 

Only games with highly specific structure, such as the important and well-studied class of linear--quadratic games, admit analytical or semi-analytical solutions. Therefore, for many applications, numerical approximations are necessary. However, accurate, scalable, and practical numerical approximation of SDG remains an open problem. Finite difference methods have been studied for two-player zero-sum games~\cite{falcone2006numerical} and for mean-field games~\cite{achdou2013mean,achdou2010mean}. Grid-based finite difference and finite element methods for PDEs, however, suffer from the curse of dimensionality (COD) and are therefore limited to low-dimensional settings. To the best of our knowledge, direct approximation of the Nash PDE system for non-zero-sum stochastic differential games has not been studied in the literature, and would in any case be of limited interest because of the aforementioned scaling problem. A Markov chain approximation framework for two-player zero-sum games was introduced in~\cite{kushner2002numerical}, and extended to finite-player non-zero-sum stochastic differential games in~\cite{kushner2007numerical}. It was later extended to more general dynamics, including regime switching and jump diffusions~\cite{bui2019approximation}. However, this approach is also grid-based and thus suffers from the COD.


Han~\cite{hu2019deep}, Han and Hu~\cite{han2020deep}, and Han, Hu, and Long~\cite{han2022convergence} proposed so-called deep fictitious play algorithms as a general numerical scheme for the Nash FBSDE system. It is based on the iterative fictitious-play procedure introduced in~\cite{brown1951iterative}. It is a Picard-type iteration in which, at each step, every player solves a single-player stochastic control problem with the other players' control policies fixed. Approximating the resulting single-player fictitious-play FBSDEs using the deep BSDE method~\cite{han2018solving} leads to deep fictitious-play. In~\cite{han2022convergence}, Han, Hu, and Long proved the geometric convergence of fictitious play when the forward process of the FBSDE is decoupled from the solution components of the backward equation. This decoupled setting is relevant because it is possible, by means of a transformation, to remove the coupling in the forward process, after an appropriate modification of the backward equation. However, when the original deep BSDE method is applied directly to coupled FBSDEs from stochastic control, as shown empirically in~\cite{andersson2023convergence, andersson2026deep}, it fails to converge except for rather obscure parameter choices. Based on the authors' experience in the work related to~\cite{andersson2023convergence, andersson2026deep}, the aforementioned decoupling does not appear to remedy the convergence problem. For this reason, it is desirable to work with fully coupled FBSDEs and instead use one of the robust deep FBSDE methods in~\cite{andersson2023convergence, andersson2026deep}, which are practically applicable to a larger class of problems.

In the present work, we extend the analysis in~\cite{han2022convergence} by showing geometric convergence of fictitious play in the fully coupled setting. Under an alternative set of assumptions, we show that the convergence is super-exponential for a special class of games. We consider both stopped games on bounded domains under additional smoothness assumptions and unrestricted games under either a gradient bound assumption or a typical smallness condition on the coefficients or the terminal time. The former case is closely related to BSDEs on domains, for which deep learning-based approximation methods with error guarantees have recently been developed~\cite{andersson2025multi}. The geometric rate is verified experimentally for a linear--quadratic interbank borrowing and lending game with 2 and 20 players.

The paper is organized as follows. Section~\ref{sec:problem_formulation} introduces the game-theoretic setting and presents the Nash PDE and Nash FBSDE formulations. To make the manuscript more self-contained, accessible, and useful to a wider range of readers, this section is intentionally comprehensive and concludes with examples of games. Section~\ref{sec:FP} presents the fictitious-play approximation, the detailed setting, and convergence analysis including our main result. In Section~\ref{sec:NE}, we present our numerical experiment.

\section{General setting and preliminaries on SDG}\label{sec:problem_formulation}

To make the manuscript more self-contained, accessible, and useful to a broader range of readers, we introduce stochastic differential games from both the PDE and FBSDE perspectives. We present the assumptions, give a high-level overview of the synthesis through dynamic programming and end with concrete examples which demonstrate our abstract setting.

\subsection{Spaces of stochastic processes and preliminaries}\label{sec:notation_spaces}
Throughout this paper, let \(T\in(0,\infty)\), \(n\in\N\), and let $(\W_t)_{t\in[0,T]}$ be an $n$-dimensional standard Brownian motion on a filtered probability space \((\Omega,\F,(\F_t)_{t\in[0,T]},\P)\). For $\ell\geq1$, denote by $\|\cdot\|_{\R^\ell}$ and $\langle\cdot,\cdot\rangle_{\R^\ell}$ the norm and scalar product on $\R^\ell$, and by $\tnorm{\cdot}_{\R^{\ell \times n}}$ the Frobenius norm on $\R^{\ell\times n}$. Furthermore, for $p\in[1,\infty)$, let 
$\mathbb L^{p,\ell}$ be the space of random variables \(U\colon\Omega\to\R^\ell\) such that
\begin{align*}
\|U\|_{\mathbb L^{p,\ell}} \defeq \Big(\E\Big[\|U\|_{\mathbb R^\ell}^p\Big]\Big)^\frac1p<\infty.
\end{align*}
For $t \in (0,T]$, $p \in [1, \infty)$, and $\beta \in [0,\infty)$, let $\mathbb S^{p,\ell}_{\beta,t}$, $\mathbb H^{p,\ell}_{\beta,t}$, and $\mathbb H_{\beta,t}^{p,\ell,n}$ denote the Banach spaces of predictable processes $\mathcal Y,\mathcal Z\colon\Omega\times[0,t]\to\R^\ell$, and $\mathcal W\colon\Omega\times[0,t]\to\R^{\ell\times n}$, that satisfy
\begin{align*}
    \|\mathcal{Y}\|_{\mathbb S^{p,\ell}_{\beta,t}}
    &\defeq\bigg(\E\bigg[\sup_{s\in[0,t]}\e^{\beta s}\|\mathcal{Y}_s\|_{\mathbb R^{\ell}}^p\bigg]\bigg)^{\frac1p}<\infty,\\
    \|\mathcal{Z}\|_{\mathbb H^{p,\ell}_{\beta,t}}
    &\defeq
    \bigg(\E\bigg[\bigg(\int_0^t\e^{\beta s}\|\mathcal{Z}_s\|_{\mathbb R^{\ell}}^2\,\d s\bigg)^{\frac{p}2}\bigg]\bigg)^\frac1p
    <\infty,\\
  \|\mathcal W\|_{\mathbb H^{p,\ell,n}_{\beta,t}}
  &\defeq
  \bigg(\E\bigg[\bigg(\int_0^t\e^{\beta s}\tnorm{\mathcal W_s}_{\R^{\ell\times n}}^2\,\d s\bigg)^{\frac{p}2}\bigg]\bigg)^\frac1p
<\infty.
\end{align*} 
In the case $\beta=0$, we suppress $\beta$ from the notation. It follows immediately from the definitions that, for $\mathcal Y \in \mathbb S^{2,\ell}_{t}$, $\mathcal Z \in \mathbb H^{2,\ell}_{t}$, we have the norm equivalences
\begin{equation}\label{eq:HHH}
\begin{split}
  \|\mathcal Y\|^2_{\mathbb S^{2,\ell}_{t}} 
  &\le 
  \|\mathcal Y\|^2_{\mathbb S^{2,\ell}_{\beta,t}} 
  \le 
  e^{\beta t}\|\mathcal Y\|^2_{\mathbb S^{2,\ell}_{t}}, \\
  \|\mathcal Z\|^2_{\mathbb H^{2,\ell}_{t}} 
  &\le 
  \|\mathcal Z\|^2_{\mathbb H^{2,\ell}_{\beta,t}} 
  \le 
  e^{\beta t}\|\mathcal Z\|^2_{\mathbb H^{2,\ell}_{t}}.
\end{split}
\end{equation}
Moreover, if $\mathcal Y \in \mathbb S^{2,1}_t$ and $\mathcal Z \in \mathbb H^{2,n}_t$, then $\mathcal Y \in \mathbb H^{2,1}_t$ and $\mathcal Y \mathcal Z \in \mathbb H^{1,n}_t$, with
\begin{align} \label{eq:HS2}
  \|\mathcal Y\|_{\mathbb{H}^{2,1}_{\beta,t}}^2
  &\leq t \|\mathcal Y\|_{\mathbb{S}^{2,1}_{\beta,t}}^2,\\
 \label{eq:HS3}
  \|\mathcal Y \mathcal Z\|_{\mathbb{H}^{1,n}_{\beta,T}}
  &\leq  
  \|\mathcal Y\|_{\mathbb{S}^{2,1}_{\beta,T}}
  \|\mathcal Z\|_{\mathbb{H}^{2,n}_{\beta,T}}.
\end{align}
For $p=2$, the spaces $\mathbb H^{2,\ell}_{\beta, t}$, $\beta \geq 0$, are Hilbert spaces with equivalent scalar products 
\begin{align*}
    \big\langle \mathcal{Z}^1,\mathcal{Z}^2\big\rangle_{\mathbb H^{2,\ell}_{\beta,t}}\defeq
    \E\bigg[\int_0^t\e^{\beta s}
    \big\langle \mathcal{Z}_s^1,\mathcal{Z}_s^2\big\rangle_{\mathbb R^{\ell}}\,\d s\bigg]
    ,\quad \mathcal{Z}^1, \mathcal{Z}^2\in\mathbb{H}_{\beta, t}^{2,\ell}.
\end{align*}
Moreover, for $\mathcal{Z}^1, \mathcal{Z}^2 \in \mathbb H^{2,\ell}_T$, a straightforward application of the Cauchy--Schwarz inequality yields
\begin{align}
    \bigg\| \int_{\cdot}^T \big\langle \mathcal Z^1_s, \mathcal Z^2_s \big\rangle_{\mathbb R^\ell} \ds \bigg\|_{\mathbb S^{1,1}_T} \leq \big\| \mathcal Z^1 \big\|_{\mathbb H^{2,\ell}_T} \big\| \mathcal Z^2 \big\|_{\mathbb H^{2,\ell}_T}. \label{eq:integral_in_S11_to_Hnorms}
\end{align}
The Burkholder--Davis--Gundy inequality (see~\cite[Theorem~3.3.28]{KarS91}) states that for $p\in\{1,2\}$, there exist constants $k_p,K_p\in(0,\infty)$ such that, for all $\phi\in \mathbb H^{p,\ell,n}_{t}$, we have
\begin{align}\label{eq:BDG1}
k_p\|\phi\|_{\mathbb{H}_{t}^{p,\ell,n}}
\le
\bigg\| \int_0^\cdot \phi_s\,\d\W_s\bigg\|_{\mathbb{S}_{t}^{p,\ell}}
\le
K_p\|\phi\|_{\mathbb{H}_{t}^{p,\ell,n}}.
\end{align}
Moreover, if $\phi\in \mathbb H^{1,\ell,n}_{t}$, then
\begin{align}\label{eq:BDG2}
\E\!\left[\int_0^t \phi_s\,\d\W_s\right]=0.
\end{align}

We next introduce notation and basic properties of stopped (or killed) stochastic processes. Let $\sigma_1,\sigma_2,\sigma_3,\sigma_4$ be $(\mathcal F_t)$-stopping times. For a predictable process $f\colon\Omega\times[0,T]\to\R^\ell$, define
\begin{align*}
   f^{[\sigma_1,\sigma_2]} \defeq f\,\mathbf 1_{[\sigma_1,\sigma_2]}, \qquad
   f^{[\sigma_1]} \defeq f^{[0,\sigma_1]}.
\end{align*}
For $\Phi\in \mathbb H^{2,\ell}_{\beta, T}$, $\Psi\in \mathbb S^{2,\ell}_{\beta, T}$, if $\sigma_1 < \sigma_2 \leq \sigma_3 < \sigma_4$, it holds that
\begin{align}\label{eq:H2_DS}
\Big\|
  \Phi^{[\sigma_1,\sigma_2]}
  +
  \Phi^{[\sigma_3,\sigma_4]}
\Big\|_{\mathbb H^{2,\ell}_{\beta,T}}^2
&=
\Big\|
  \Phi^{[\sigma_1,\sigma_2]}
\Big\|_{\mathbb H^{2,\ell}_{\beta,T}}^2
+
\Big\|
  \Phi^{[\sigma_3,\sigma_4]}
\Big\|_{\mathbb H^{2,\ell}_{\beta,T}}^2,\\
\Big\|\Psi^{[\sigma_1,\sigma_2]}+\Psi^{[\sigma_3,\sigma_4]}\Big\|_{\mathbb S^{2,\ell}_{\beta,T}}^2
&
\le \Big\|\Psi^{[\sigma_1,\sigma_2]}\Big\|_{\mathbb S^{2,\ell}_{\beta,T}}^2
+\Big\|\Psi^{[\sigma_3,\sigma_4]}\Big\|_{\mathbb S^{2,\ell}_{\beta,T}}^2. \label{eq:S_norm_split}
\end{align}

\subsection{The stochastic differential game}
\label{sec:setting2}

Throughout this paper, let $d, n, N \in \N$. Furthermore, let $\D\subseteq\R^n$ be either $\D=\R^n$ or an open, bounded spatial domain with boundary $\partial \D$, and define the space-time domain $Q_T^0=[0,T)\times \D$, its closure $Q_T=[0,T]\times \overline\D$, and its boundary $\partial Q_T=[0,T)\times \partial \D \cup \{T\}\times \overline \D$. We consider an $N$-player game and denote by \(\mathcal{I}\defeq\{1,2,\ldots,N\}\) the set of all players. The strategy of each player $i \in \I$ is given by a Markov control policy $\alpha^i \colon Q_T \to A^i\subset \R^{d}$, where \(A^i\subset \R^d\) is the compact control set for player~$i$. The policy is not given a priori, but is to be optimized. For $i\in\I$, let $\A^i$ denote the set of all $\alpha^i$ for which there exists $C=C(\alpha^i) \geq0$ such that, for all $(t_1,\x_1),(t_2,\x_2)\in Q_T$, we have
\begin{align*}
    \big\|\alpha^i(t_1,\x_1)-\alpha^i(t_2,\x_2)\big\|_{\R^{d}}^2
    &\leq C
    \big(
      |t_1-t_2|
      +
      \|\x_1-\x_2\|_{\R^n}^2
    \big).
\end{align*}
We denote by $A \defeq A^1 \times \cdots \times A^N \subset \R^{d\times N}$ the set of all possible joint actions. Similarly, the set $\A$ consists of all joint Markov policies \(\bal = (\alpha^1,\dots,\alpha^N) \colon Q_T \to A\), where $\alpha^i \in \A^i$ for each $i \in \I$. Let \(\bar{\bb}\colon Q_T\times A\to\R^n\) and \(\bSigma\colon Q_T\to\R^{n\times n}\) be the drift and diffusion coefficients of the common state process. We assume that there exist $L_{\bar{\bb}},L_{\bSigma}\geq0$ such that for all $(t_1,\x_1,\ba_1),(t_2,\x_2,\ba_2)\in Q_T\times A$ with $\ba_j=(a_j^1,\dots,a_j^N)$, $j=1,2$, we have
\begin{align}\label{eq:bhat}
    \big\|
      \bar \bb(t_1,\x_1,\ba_1)-\bar \bb(t_2,\x_2,\ba_2)
    \big\|_{\R^{n}}^2
    &\leq
    L_{\bar{\bb}}
    \big(
      |t_1-t_2|
      +
      \|\x_1-\x_2\|_{\R^n}^2
      +
      \btnorm{\ba_1-\ba_2}_{\R^{d\times N}}^2
    \big),\\
    \btnorm{\bSigma(t_1,\x_1)- \bSigma(t_2,\x_2)}_{\R^{n\times n}}^2
    &\leq
    L_{\bSigma}
    \big(
      |t_1-t_2|
      +
      \|\x_1-\x_2\|_{\R^n}^2
    \big). \label{eq:lipschitz_sigma}
\end{align}
We assume that $0\in A^i$ for all $i\in\I$ so that the trivial control $0\in\mathbb A$ is admissible. This implies that $\bar \bb$ satisfies a linear growth condition in $\x$, uniformly over $\ba\in A$, which we do not state explicitly. In Section~\ref{sec:FBSDE_Nash_system}, after a change of variables, we introduce the related coefficient $\bb$. We also assume uniform parabolicity, i.e., there exist constants $0 < c < C$ such that, for all $(t,\x)\in Q_T$ and $\boldsymbol{\xi}\in\R^n$, we have
\begin{align}\label{eq:UP}
 c\|\boldsymbol{\xi}\|_{\R^n}^2
 \leq
\boldsymbol{\xi}^\top\bSigma(t,\x)
\bSigma^\top(t,\x)\boldsymbol{\xi}
    \leq C\|\boldsymbol{\xi}\|_{\R^n}^2.
\end{align}
Uniform parabolicity implies that $\bSigma(t,\x)$ is invertible for every $(t,\x)\in Q_T$, with inverse uniformly bounded on $Q_T$. Together with~\eqref{eq:lipschitz_sigma}, this also implies that
\begin{align}\label{bPhi}
    \bPhi(t,\x) \defeq \big(\bSigma^\top(t,\x)\big)^{-1}
\end{align}
is uniformly bounded and Lipschitz continuous on $Q_T$. Consequently, there exist constants $C_{\bPhi}, L_{\bPhi} \geq 0$ such that, for all $(t_1,\x_1),(t_2,\x_2)\in Q_T$, we have
\begin{align}\label{eq:Lip_Phi}
    \tnorm{\bPhi(t_1,\x_1)}_{\R^{n\times n}}\leq C_{\bPhi},
    \quad
    \textrm{and}
    \quad
    \tnorm{
      \bPhi(t_1,\x_1)-\bPhi(t_2,\x_2)
    }^2_{\R^{n\times n}}
    \leq
    L_{\bPhi}
    \big(
      |t_1-t_2|+\|\x_1-\x_2\|_{\R^n}^2
    \big).
\end{align}
The state dynamics are determined by the predictable, square-integrable and up to modification unique stochastic processes $\X^{t,\x,\bal}\colon[t,T]\times\Omega\to\overline\D$, which for all $s\in[t,T]$, $\P$-almost surely satisfy
\begin{align}\label{eq:state_alpha_process}
    \X_s^{t,\x,\bal} 
    = 
    \x
    +
    \int_t^{s\wedge\tau}
    \bar{\bb}\big(r,\X_r^{t,\x,\bal},\bal\big(r,\X_r^{t,\x,\bal}\big)\big)\,\d r 
    + 
    \int_t^{s\wedge\tau} \bSigma\big(r,\X_r^{t,\x,\bal}\big)\,\d \W_r.
\end{align}
Here, $\tau=\tau(t,\x,\bal)\colon\Omega \to[t,T]$ denotes the exit time of $Q_T^0$, i.e., the first time at which $\X^{t,\x,\bal}$ reaches either $\partial \cD$ or the terminal boundary $\{T\}\times\overline\D$. To simplify notation,  we suppress the $(t,\x,\bal)$-dependence of $\tau$. Under our assumptions, these processes are well defined by standard SDE theory.

For \(i \in \I\), let \(\bar{f}^i \colon Q_T \times A \to \R\) and \(g^i \colon \partial Q_T \to \R\) denote player \(i\)'s running and terminal cost functions, respectively. The objective of player \(i\) is to minimize its individual cost functional, which depends both on its own control and on the actions of the remaining $N - 1$ players. These cost functionals $J^i\colon Q_T\times \A\to \R$, $i\in\I$, are given by
\begin{equation}\label{eq:cost_alpha}
\begin{split}
&J^i(t, \x, \bal) \defeq \E \bigg[ \int_t^\tau \bar{f}^i\big(s, \X_s^{t,\x,\bal}, \bal\big(s, \X_s^{t,\x,\bal}\big)\big) \, \d s + g^i\big(\tau,\X_\tau^{t,\x,\bal}\big) \bigg].
\end{split}
\end{equation}
Under our assumptions, the processes $\X^{t,\x,\bal}$ have finite moments of all orders. Thus, polynomial bounds on $\bar f^{1},\dots,\bar f^{N}$ and $g^{1},\dots,g^{N}$ are sufficient to ensure that $J^{1},\dots,J^{N}$ are well defined. However, much stronger assumptions are required for the solution theory; see, e.g., Examples~\ref{ex:1}--\ref{ex:LQ} below. For the convergence analysis in Section~\ref{sec:conv_FP}, we impose the following Lipschitz continuity. For each $i \in \I$, there exist $L_{\bar{f}^i}, L_{g^i} \geq 0$ such that for all $(t_1, \x_1, \boldsymbol{a}_1), (t_2, \x_2, \boldsymbol{a}_2) \in Q_T \times A$, 
\begin{align*}
    \big|
        \bar{f}^i(t_1, \x_1, \boldsymbol{a}_1) 
        - 
        \bar{f}^i(t_2, \x_2, \boldsymbol{a}_2)
    \big|^2
    &\leq 
    L_{\bar{f}^i}
    \big(
        |t_1 - t_2|
        +
        \| \x_1 - \x_2 \|^2_{\R^n}
        +
        \tnorm{\ba_1 - \ba_2 }^2_{\R^{d\times N}}
    \big),
\end{align*}
and, for all $(t_1, \x_1), (t_2, \x_2) \in \partial Q_T$,
\begin{align*}
    \big|
        g^i(t_1, \x_1) 
        - 
        g^i(t_2, \x_2)
    \big|^2
    &\leq 
    L_{g^i}
    \big(
        |t_1 - t_2|
        +
        \| \x_1 - \x_2 \|^2_{\R^n}
    \big).
\end{align*}
For $\ba\in A$ and $i\in\I$, let $A^{-i}$ denote the set of actions
\begin{align*}
    \ba^{-i} \defeq \big(a^1,\dots,a^{i-1},a^{i+1},\dots,a^N\big),
\end{align*}
and use the analogous notation $\bal^{-i}\in \A^{-i}$ for joint policies. For $i \in \I$, define $[\cdot,\cdot]_i\colon A^i\times A^{-i}\to A$ by
\begin{align*}
    [a, \ba^{-i}]_i=(a^1,\dots,a^{i-1},a,a^{i+1},a^N), \quad a \in A^i,\ \ba^{-i} \in A^{-i}.
\end{align*}
From an optimization perspective, it is not sufficient for a player \(i \in \I\) to find an optimal strategy \(\beta^{i}\) against an arbitrary joint Markov policy \(\bal^{-i}\) of the other \(N-1\) players, because those players are optimizing their strategies as well. For a solution to be truly optimal, each player \(i \in \I\) must have a strategy \(\beta^{i}=\alpha^{i,*}\) that is optimal given the strategies \(\bal^{-i,*}=(\alpha^{1,*},\dots,\alpha^{i-1,*},\alpha^{i+1,*},\alpha^{N,*})\) of the others. Otherwise, some player would have an incentive to change strategy, which contradicts optimality. Moreover, the cost functional of player \(i\) depends on the strategies of the other players directly, and indirectly through the controlled state process. Hence, a change in the strategy of any other player may cause \(\alpha^{i,*}\) to lose its optimality. Thus, optimal strategies are generally not stable under unilateral changes unless they are jointly optimized. This leads to the notion of a Nash equilibrium: no player can benefit from unilaterally deviating from their chosen strategy. Formally, an \(N\)-player closed-loop Nash equilibrium is a tuple of strategies \(\bal^* = (\alpha^{1,*}, \alpha^{2,*}, \ldots, \alpha^{N,*}) \in \A\) such that for every player \(i \in \I\), every Markov policy \(\alpha^i \in \A^i\), and every $(t,\x)\in Q_T$, we have 
\[
J^i\big(t,\x,\bal^{*}\big) 
\leq 
J^i\big(t,\x,\big[\alpha^i,\bal^{-i,*}\big]_i\big).
\]


\subsection{Synthesis through dynamic programming}
\label{sec:SDP}

The synthesis of the game through dynamic programming is based on the players' value functions $V^i\colon Q_T\times \A^{-i}\to \R$, $i\in\I$, defined by
\[
V^i\big(t,\x,\bal^{-i}\big)
\defeq
\inf_{\alpha^i\in\A^i}
J^{i}\big(t,\x,\big[\alpha^i,\bal^{-i}\big]_i\big).
\]
Consider player $i \in \I$ and fix an arbitrary joint policy $\bal^{-i}\in\A^{-i}$ for the remaining $N - 1$ players. From the perspective of player $i$, these players are treated as fixed and therefore become part of the dynamics. Thus, player $i$ solves a single-player stochastic control problem and the associated Hamilton--Jacobi--Bellman (HJB) equation. More precisely, we seek $\mathcal V^{i}\colon Q_T\times \A^{-i}\to\R$, which for all $(t,\x,\bal^{-i})\in Q_T^0\times \A^{-i}$ satisfies
\begin{equation}\label{eq:HJB}
    \frac{\partial \V^i}{\partial t} \big(t,\x,\bal^{-i}\big)
    + \big(\cL_t \V^i\big)\big(t,\x,\bal^{-i}\big)
    + H^i\big(t,\x,\mathrm{D}_{\x}\V^i\big(t,\x,\bal^{-i}\big),\bal^{-i}(t,\x)\big)
    = 0,
\end{equation}
and for $(t,\x)\in\partial Q_T$ satisfies the boundary condition
\begin{equation}\label{eq:HJB_terminal}
        \V^i\big(t, \x,\bal^{-i}\big) = g^i(t,\x).
\end{equation}
The Hamiltonian $H^i\colon Q_T\times\R^n\times A^{-i}\to \R$ is given by 
\begin{align*}
    H^i\big(t,\x,p,\ba^{-i}\big):=
\inf_{a^i\in A^{i}}\Big(\big\langle\bar{\bb}\big(t,\x,\big[a^i,\ba^{-i}\big]_i\big),p\big\rangle + \bar{f}^i\big(t,\x,\big[a^i,\ba^{-i}\big]_i\big)\Big),
\end{align*}
and the second-order operators $\cL_t$, $t\in[0,T]$, which act on $\phi\in C^2(Q_T,\R)$, by
\begin{align*}
  (\cL_t\phi)(t,\x):=\frac{1}{2} \mathrm{Tr}\big(\bSigma(t,\x)\, \mathrm{D}_{\x}^2 \phi(t,\x)\, \bSigma^\top(t,\x)\big).
\end{align*}
Here, \(\mathrm{D}_{\x} \V^i\) and \(\mathrm{D}_{\x}^2 \V^i\) denote the gradient and Hessian of \(\V^i\) with respect to \(\x\), respectively, and \(\mathrm{Tr}\) denotes the trace operator. Verification theorems show that, in settings where the HJB equation~\eqref{eq:HJB} admits a sufficiently regular solution, this solution coincides with the value function, i.e., that $V^i(t,\x,\bal^{-i})=\V^i(t,\x,\bal^{-i})$ and that the optimal policy $\beta^{i}\colon Q_T\times \A^{-i}\to A^i$ is given by 
\begin{align*}
    \beta^{i}\big(t,\x,\bal^{-i}\big)
    =
    \kappa^i\big(t,\x,\mathrm{D}_{\x}\V^i\big(t,\x,\bal^{-i}\big),\bal^{-i}(t,\x)\big).
\end{align*}
For fixed $\bal^{-i}$, the policy $(t,\x) \mapsto \beta^{i}(t,\x,\bal^{-i})$ in $\A^i$ is given by $\kappa^i\colon Q_T\times \R^n\times A^{-i}\to A^i$, where
\begin{align*}
    \kappa^i\big(t,\x,p,\ba^{-i}\big)\in\argmin_{a^i\in A^i}
\Big(\big\langle\bar{\bb}\big(t,\x,\big[a^i,\ba^{-i}\big]_i\big),p\big\rangle + \bar{f}^i\big(t,\x,\big[a^i,\ba^{-i}\big]_i\big)\Big).
\end{align*}
We assume that $\kappa^1, \dots, \kappa^N$ exist and are unique, in the sense that, for every $(t,\x,p,\ba^{-i})\in Q_T\times \R^n\times A^{-i}$, each minimization problem defining $\kappa^1, \dots, \kappa^N$ admits a unique solution; see Example~\ref{ex:1} below. For the sake of exposition, throughout the present section we also assume the existence of a sufficiently regular $\V$ satisfying an appropriate verification theorem, and therefore write $V$ instead of~$\V$. This somewhat imprecise assumption is replaced in Section~\ref{sec:FP} by concrete and stronger assumptions.

We next introduce the vector-valued mappings $\bV \colon Q_T \times \A \to \R^N$, $\boldsymbol{\kappa} \colon Q_T \times \R^{n \times N} \times A \to A$ and $\bbeta \colon Q_T \times \A \to A$. Writing $\p=(p^1,\dots,p^N)\in\R^{n\times N}$, these are defined by
\begin{align*}
\bV(t,\x,\bal)
&\defeq
\big(
  V^1(t,\x,\bal^{-1}),
  \dots,
  V^N(t,\x,\bal^{-N})
\big),\\
\mathbf{\boldsymbol{\kappa}}(t,\x,\p,\ba)
&\defeq
\big(
  \kappa^1\big(t,\x,p^1,\ba^{-1}\big),
  \dots,
  \kappa^N\big(t,\x,p^N,\ba^{-N}\big)
\big),\\
\bbeta(t,\x,\bal)
&\defeq
\mathbf{\boldsymbol{\kappa}}\big(t,\x,\mathrm{D}_{\x} \bV(t,\x,\bal),\bal(t,\x)\big)
=
\big(\beta^1\big(t,\x,\bal^{-1}\big),\dots,\beta^N\big(t,\x,\bal^{-N}\big)\big).
\end{align*}
For $\boldsymbol{\kappa}$, we assume the following Lipschitz condition. There exist $L_{\kappa}$ and $L_{\kappa}^a \in [0,1)$ such that
\begin{align*}
\begin{split}
        &\tnorm{
        \boldsymbol{\kappa}(t_1, \x_1, \boldsymbol{p}_1, \boldsymbol{a}_1)
        -
        \boldsymbol{\kappa}(t_2, \x_2, \boldsymbol{p}_2, \boldsymbol{a}_2)}^2_{\R^{d \times N}}
    \\
    &\qquad\leq
    L_{\kappa}
    \big(
    |t_1-t_2|
    +
    \|\x_1-\x_2\|_{\R^n}^2
    +
    \tnorm{\boldsymbol{p}_1-\boldsymbol{p}_2}_{\R^{n \times N}}^2
    \big)
    +
    L_{\kappa}^a
    \tnorm{\ba_1-\ba_2}_{\R^{d\times N}}^2.
\end{split}
\end{align*}
Note that this implies Lipschitz continuity for $\kappa^i$ of the form
\begin{align}
\begin{split}
    &\|
        \kappa^i(t_1, \x_1, p_1, \boldsymbol{a}^{-i}_1)
        -
        \kappa^i(t_2, \x_2, p_2, \boldsymbol{a}^{-i}_2)
    \|^2_{\R^{d}}
    \\
    &\qquad\leq
    L_{\kappa}
    \big(
    |t_1-t_2|
    +
    \|\x_1-\x_2\|_{\R^n}^2
    +
    \|p_1-p_2\|_{\R^{n}}^2
    \big)
    +
    L_{\kappa}^a
    \tnorm{\ba^{-i}_1-\ba^{-i}_2}_{\R^{d\times (N-1)}}^2.
\end{split}\label{eq:kappai_lipschitz}
\end{align}
We remark that, for fixed $\bal\in\A$, the joint policy $\bbeta(\cdot,\cdot,\bal) \in \A$ is in general not optimal for the game, since each player $i\in\I$ optimizes against the fixed policies $\bal^{-i}$. Instead, the game is solved by finding a fixed point $\bal^*$ which, for all $(t,\x) \in Q_T$, satisfies
\begin{align*}
    \bal^*(t,\x) = \bbeta(t,\x, \bal^*),
\end{align*}
possibly in some subspace of $\A$. For such a policy $\bal^*$, we write with slight abuse of notation 
\begin{align*}
    V^i(t,\x) \defeq V^i(t,\x,\bal^{*,-i}), \qquad
    \bV(t,\x) \defeq \bV(t,\x,\bal^*),
\end{align*}
both defined on $Q_T$. By the definition of $\bbeta$, it follows that if $\p = \mathrm{D}_{\x} \bV(t,\x) \in \R^{n\times N}$, then $\bal^*$ satisfies
\begin{align}
    \bal^*(t,\x) = \boldsymbol{\kappa} \big(t, \x, \p, \bal^*(t,\x)\big), \label{eq:fp2}
\end{align}
for all $(t,\x) \in Q_T$. We next extend this relation to arbitrary $\p \in \R^{n\times N}$. To this end, we assume that the generalized Isaacs conditions hold; see, e.g., \cite{carmona2018probabilistic, friedman1972stochastic}. That is, for every $\p \in \R^{n\times N}$, there exists a fixed point $\bal^*_{\p} \in \A$ satisfying~\eqref{eq:fp2}, and the associated map
\begin{align}
    \bc(t, \x, \p) \defeq \bal^*_{\p}(t, \x) \label{eq:c_def}
\end{align}
is a well-defined function $Q_T \times \R^{n \times N} \to A$. See Example~\ref{ex:1} for one instance in which this holds. In addition, we assume that this fixed point is unique and that there exists $L_{\bc}\geq0$ such that for all $(t_1, \x_1, \p_1), (t_2, \x_2, \p_2) \in Q_T \times \R^{n \times N}$, we have
\begin{align}\label{eq:lip7}
    \tnorm{\bc(t_1,\x_1,\p_1)-\bc(t_2,\x_2,\p_2)}_{\R^{d\times N}}^2
    \leq
    L_{\bc}
    \big(
      |t_1-t_2|
      +
      \|\x_1-\x_2\|_{\R^n}^2
      +
      \tnorm{\p_1 - \p_2}_{\R^{n \times N}}^2
    \big).
\end{align}

The Hamiltonians associated with the Nash equilibrium are given by
\begin{align*}
    H^i(t,\x,\p)
=
\big\langle 
\bar\bb\big(t,\x,\bc(t,\x,\p)\big),p^i
\big\rangle
+
\bar{f}^i\big(t,\x,\bc(t,\x,\p)\big)
, \quad i\in\I.
\end{align*}
This allows us to write the resulting system of HJB equations explicitly. We seek $\bV\colon Q_T\to \R ^N$ such that, for all $(t,\x)\in Q_T^0$, we have
\begin{align}\label{eq:Nash_PDE_system1}
    \frac{\partial \bV}{\partial t}
    (t,\x)
    +
    (\bcL_t \bV)(t,\x)
    +
    \bH\big(t,\x,\mathrm{D}_{\x}\bV(t,\x)\big)=0,
\end{align}
and for all $(t,\x)\in\partial Q_T$ the boundary condition
\begin{align}\label{eq:Nash_PDE_system2}
    \bV(t,\x)=\g(t,\x).
\end{align}
Here, for $\bphi=(\phi^1,\dots,\phi^N)\in C^2(Q_T,\R^N)$, $(t,\x)\in Q_T$ and $\p \in \R^{n\times N}$, we have
\begin{align*}
    (\bcL_t \bphi)(t,\x)  
    \defeq
    \left[
    \begin{array}{c}
         (\cL_t\phi^1)(t,\x)\\
         \vdots\\
         (\cL_t\phi^N)(t,\x)
    \end{array}
    \right],\quad
    \bH
    (t,\x,\p)
    \defeq
    \left[ 
    \begin{array}{c}
         H^1(t,\x,\p)\\ 
         \vdots\\
         H^N(t,\x,\p)
    \end{array}
    \right],
\end{align*}
and $\g=(g^1,\dots,g^N)\colon \partial Q_T\to \R^N$. Since $g^i$ is Lipschitz, it follows that $\boldsymbol{g}$ is Lipschitz with constant
\begin{align*}
    L_{\boldsymbol{g}} \defeq \sum_{i \in \I} L_{g^i}.
\end{align*}
The system~\eqref{eq:Nash_PDE_system1}--\eqref{eq:Nash_PDE_system2} is called the \emph{Nash PDE system}. If a solution exists, then each player solves its corresponding HJB equation and is therefore optimal, provided the relevant verification theorem applies. This implies that the resulting solution leads to a Nash equilibrium $\bal^*$. For the exposition in this section, we assume existence and uniqueness of a sufficiently regular solution to the Nash PDE system. In Section~\ref{sec:FP} we replace this by concrete assumptions ensuring these properties.


\subsection{FBSDE formulation of the Nash system}
\label{sec:FBSDE_Nash_system}

We next introduce the FBSDE formulation of the Nash system. The optimal game dynamics are given by the family of predictable, square-integrable and up to modification unique stochastic processes $\X^{t,\x}\colon[t,T]\times\Omega\to \R^n$, $(t,\x)\in Q_T$, which satisfy, for all $s\in[t,T]$, $\P$-almost surely
\begin{align*}
    \X_s^{t,\x}
    =
    \x
    +
    \int_t^{s\wedge \tau}
    \bar \bb\big(r,\X_r^{t,\x},\bc
    \big(r,\X_r^{t,\x},\mathrm{D}_{\x}\bV\big(r,\X_r^{t,\x}\big)\big)\big)
    \,\d r
    +
    \int_t^{s\wedge \tau}
    \bSigma\big(r,\X_r^{t,\x}\big)
    \,\d \W_r.
\end{align*}
For $(t,\x) \in Q_T$, let $\Y^{t,\x}\colon[t,T]\times\Omega\to \R^N$, $\Z^{t,\x}\colon[t,T]\times\Omega\to \R^{n\times N}$ be the family of stochastic processes which, for $s\in[t,T]$, $\P$-almost surely are given by
\begin{equation}\label{eq:YZ1}
\begin{split}
\Y_s^{t,\x}
&=
\bV\big(s,\X_s^{t,\x}\big),\\
\Z_s^{t,\x}
&=
\bSigma^\top\big(s,\X_s^{t,\x}\big)\mathrm{D}_{\x}\bV\big(s,\X_s^{t,\x}\big).
\end{split}
\end{equation}
Assuming $\bV$ is sufficiently regular, It\^{o}'s formula and the HJB equations give $\P$-almost surely
\begin{align}\label{eq:unif_par}
\Y_s^{t,\x}
=
\g\big(\tau,\X_\tau^{t,\x}\big)
+
\int_{s\wedge \tau}^\tau
\bar\f
\big(r,\X_r^{t,\x},\bc\big(r,\X_r^{t,\x},\mathrm{D}_{\x}\bV\big(r,\X_r^{t,\x}\big)\big)\big)
\, \d r
-
\int_{s\wedge \tau}^\tau
\big(\Z_r^{t,\x}\big)^\top
\d \W_r,
\end{align}
where $\bar\f \defeq (\bar f^1,\dots,\bar f^N)\colon[0,T]\times\R^n\times A\to \R^N$. Since $\bar{f}^i$ is Lipschitz, $\bar \f$ is Lipschitz with constant
\begin{align*}
    L_{\bar \f} \defeq \sum_{i \in \I} L_{\bar{f}^i}.
\end{align*}
By uniform parabolicity, the map $\p \mapsto \bSigma^\top(t,\x)\p$ is injective, with inverse $\z\mapsto \bPhi(t,\x) \z$. This allows us to define $\bb\colon Q_T\times \R^{n \times N}\to\R^n$ and $\f\colon Q_T\times \R^{n \times N}\to \R^N$ by
\begin{align*}
    \bb(t,\x,\z)
    &\defeq
    \bar \bb\big(t,\x,\bc(t,\x,\bPhi(t,\x)\z)\big),\\
    \f(t,\x,\z)
    &\defeq
    \bar \f\big(t,\x,\bc(t,\x,\bPhi(t,\x)\z)\big),
\end{align*}
and write the equations for $\X^{t,\x}$, $\Y^{t,\x}$ and $\Z^{t,\x}$ as a coupled FBSDE, i.e., for all $(t,\x)\in Q_T$, $s\in[t,T]$, it holds $\P$-almost surely
\begin{align}\label{eq:FBSDE_full}
\begin{split}
    \X_s^{t,\x}&=\x + \int_t^{s\wedge \tau}\bb\big(r,\X_r^{t,\x},\Z_r^{t,\x}\big)\,\d r + \int_t^{s\wedge \tau}\bSigma\big(r,\X_r^{t,\x}\big)\,\d\W_r,\\
    \Y_s^{t,\x} &=\g\big(\tau,\X_\tau^{t,\x}\big) +\int_{s\wedge \tau}^\tau\f\big(r,\X_r^{t,\x},\Z_r^{t,\x}\big)\,\d r - \int_{s\wedge \tau}^\tau\big(\Z_r^{t,\x}\big)^\top\,\d\W_r.
    \end{split}
\end{align}
We refer to~\eqref{eq:FBSDE_full} as the \emph{Nash FBSDE system}. Our minimum assumptions on $\f$ and $\g$ are that, for all $(t,\x)\in Q_T$, there exists a unique solution
\begin{align*}
    \big(\X^{t,\x},\Y^{t,\x},\Z^{t,\x}\big)\in \mathbb S^{2,n}_{t,T}\times \mathbb S^{2,N}_{t,T} \times \mathbb H^{2,n,N}_{t,T},
\end{align*}
where $\mathbb S^{2,n}_{t,T}$, $\mathbb S^{2,N}_{t,T}$ and $\mathbb H^{2,n,N}_{t,T}$ are defined analogously to $\mathbb S^{2,n}_{T}, \mathbb S^{2,N}_{T}$ and $\mathbb H^{2,n,N}_{T}$ for processes defined on the interval $[t,T]$. In addition, for every $p\in[2,\infty)$, we assume that $\X^{t,\x}\in \mathbb S^{p,n}_{t,T}$. In Section~\ref{sec:FP}, we impose concrete assumptions that guarantee these properties.

\subsection{Additional notation for non-equilibrium play}\label{sec:non_equilibrium_play}

For later use, we introduce the control problem for player $i\in\I$ when the policies of the other players are fixed. For this purpose, define $\tilde\bb^i\colon Q_T\times \R^n\times A^{-i}\to\R^n$ and $\tilde\ell^i\colon Q_T\times \R^n\times A^{-i}\to\R$ by
\begin{equation}\label{eq:bifi}
\begin{split}
\tilde\bb^i\big(t,\x,z,\boldsymbol a^{-i}\big)
&\defeq
\bar \bb\big(t,\x,\big[\kappa^i(t,\x,\bPhi(t,\x)z,\boldsymbol a^{-i}),\boldsymbol a^{-i}\big]_i\big),\\
\tilde\ell^i\big(t,\x,z,\boldsymbol a^{-i}\big)
&\defeq
\bar f^i\big(t,\x,\big[\kappa^i(t,\x,\bPhi(t,\x)z,\boldsymbol a^{-i}),\boldsymbol a^{-i}\big]_i\big).
\end{split}
\end{equation}
Given any admissible control \(\boldsymbol\beta^{-i}\in\mathbb A^{-i}\) for the other \(N-1\) players, player \(i\in\I\) solves the FBSDE
\begin{align}\label{eq:FBSDE_single_player}
\begin{split}
    \X_t^{i,\boldsymbol\beta^{-i}}
    &=
    \x_0
    +
    \int_0^{t\wedge \tau}
      \tilde\bb^i\big(
        s,\X_s^{i,\boldsymbol\beta^{-i}},
        Z_s^{i,\boldsymbol\beta^{-i}},
        \boldsymbol\beta^{-i}_s
      \big)
      \,\d s
    +
    \int_0^{t\wedge \tau}
      \bSigma\big(s,\X_s^{i,\boldsymbol\beta^{-i}}\big)\,\d\W_s,
    \\
    Y_t^{i,\boldsymbol\beta^{-i}}
    &=
    g^i\big(\tau,\X_\tau^{i,\boldsymbol\beta^{-i}}\big)
    +
    \int_{t\wedge \tau}^{\tau}
      \tilde\ell^i\big(
        s,\X_s^{i,\boldsymbol\beta^{-i}},
        Z_s^{i,\boldsymbol\beta^{-i}},
        \boldsymbol\beta^{-i}_s
      \big)
      \,\d s
    -
    \int_{t\wedge \tau}^{\tau}
      \big(Z_s^{i,\boldsymbol\beta^{-i}}\big)^\top\,\d\W_s.
\end{split}
\end{align}
Since the state variable is shared among all players, we use boldface notation for $\X^{i,\boldsymbol\beta^{-i}}$. In contrast, we use non-bold notation for $Y^{i,\boldsymbol\beta^{-i}}$ and $Z^{i,\boldsymbol\beta^{-i}}$, since these are player-specific value and control processes. However, the processes $\X^{1,\boldsymbol\beta^{-1}},\dots,\X^{N,\boldsymbol\beta^{-N}}$ are generally different, since each player optimizes its own control given $\boldsymbol\beta^{-i}$.

\subsection{Examples of games}
\label{sec:examples}

This section presents a sequence of examples with progressively stronger assumptions. We begin with a setting based only on smoothness, growth, and boundedness conditions, which yields rather restrictive results. We then introduce additional structure on the domain, dynamics, and costs, which leads to a linear--quadratic class for which the Nash system can be reduced to a coupled Riccati system.

\begin{example}\label{ex:unbounded_domain}
Let $\D=\R^n$, so that $\partial Q_T=\{T\}\times\R^n$ and $\P(\tau=T)=1$. Assume $A$ to be compact and the functions $\bar \bb$, $\bSigma$, $\bar{f}^i$, and $g^i$, $i\in\I$, to be bounded. Then, under some additional assumptions, \cite[Proposition 2.13]{carmona2018probabilistic} guarantees existence and uniqueness of a solution to the Nash system~\eqref{eq:Nash_PDE_system1}--\eqref{eq:Nash_PDE_system2}. The boundedness is a very strong condition one has to pay for not imposing more structure.
\end{example}


\begin{example}\label{ex:bounded_domain}
We consider a bounded domain $\D$ with boundary $\partial \D$ of class $C^2$, and assume that $A$ is compact. In addition to our standing assumptions, assume that $\bSigma$ is bounded and continuously differentiable in time and space with bounded derivatives, and that $\g$ is twice differentiable in the sense required in~~\cite{friedman1972stochastic}. Then~\cite[Theorem~1]{friedman1972stochastic} guarantees the existence of a (nonclassical) solution to the Nash system~\eqref{eq:Nash_PDE_system1}--\eqref{eq:Nash_PDE_system2}. These conditions are considerably less restrictive than in the unbounded case $\D=\R^n$. For further refined results, see~\cite{bensoussan2002smooth}.
\end{example}

\begin{example}\label{ex:1}
Here, we consider a subclass of games satisfying the setting of Example~\ref{ex:bounded_domain} and thus have a unique solution to the Nash system. The drift coefficient has the form
\begin{align*}
    \bar \bb(t,\x,\ba)
    =
    \bb_1(t,\x)
    +
    \sum_{i\in\I}
    b_2^i(t,\x)a^i,
\end{align*}
where $\bb_1\colon Q_T\to\R^n$ and $b_2^i\colon Q_T\to \R^{n\times d}$, $i\in\I$, satisfy for all $(t_1,\x_1),(t_2,\x_2)\in Q_T$ the bounds
\begin{align*}
    \big\|\bb_1(t_1,\x_1)-\bb_1(t_2,\x_2)\big\|_{\R^{n}}
    &\leq C
    \big(
      |t_1-t_2|^{\frac12}
      +
      \|\x_1-\x_2\|_{\R^n}
    \big),\\
    \tnorm{b_2^i(t_1,\x_1)-b_2^i(t_2,\x_2)}_{\R^{n\times d}}
    &\leq C
    \big(
      |t_1-t_2|^{\frac12}
      +
      \|\x_1-\x_2\|_{\R^n}
    \big).
\end{align*}
In addition, $b_2^i$ is bounded on $Q_T$ for each $i \in \I$. Under these assumptions $\bar\bb$ satisfies~\eqref{eq:bhat}. For each $i \in \I$, we consider the running cost
\begin{align*}
    \bar{f}^i(t,\x,\ba)
    =
    f_1^i(t,\x)
    +
    \frac12
    \sum_{j\in\I}
    \big\langle 
        K_a^{i,j} a^j, a^j
    \big\rangle
    +
    \sum_{j\in\I}
    \big\langle 
        K_{a\x}^{i,j} a^j, \x
    \big\rangle.
\end{align*}
Here, $f_1^i\colon Q_T \to \R$ determines the state-based cost, whereas $K_a^{i,j}\in\R^{d\times d}$, $K_{a\x}^{i,j}\in\R^{n\times d}$, $i,j\in\I$, penalize control and control-state relation, respectively. We assume $f_1^i$ to have polynomial growth and $K_a^{i,i}$ to be positive definite and invertible. With these data, the related Hamiltonian for player $i$ reads
\begin{align*}
H^i(t,\x,p,\ba^{-i})
&=
\bigg\langle 
  \bb_1(t,\x)
  +
  \sum_{j\in\I\setminus\{i\}}
  b_2^j(t,\x)a^j
  ,p
\bigg\rangle
+
f_1^i(t,\x)
+
\frac12
\sum_{j\in\I\setminus\{i\}}
\big\langle 
    K_a^{i,j} a^j, a^j
\big\rangle\\
&\quad+
\sum_{j\in\I\setminus\{i\}}
\big\langle 
    K_{a\x}^{i,j} a^j,\x 
\big\rangle
+
\inf_{a^i\in A^{i}}
\bigg(
  \big\langle 
    b_2^i(t,\x)a^i,p
  \big\rangle 
  + 
  \frac12
  \big\langle 
    K_a^{i,i} a^i, a^i
  \big\rangle
  +
  \big\langle
    K_{a\x}^{i,i} a^i, \x
  \big\rangle
  \bigg).
\end{align*}
The minimization problem is quadratic and when the minimum is attained in the interior of $A^i$, its solution is explicit. This yields the optimum
\begin{align*}
    \kappa^i
    \big(
      t,\x,p,\ba^{-i}
    \big)
    =
    \Phi^i(t,\x,p)
    +
    \mathcal{E}^i_A(t,\x,p),
\end{align*}
where $\Phi^i\colon Q_T\times \R^n\to\R^d$, $i\in\I$ is given by
\begin{align*}
\Phi^i(t,\x,p)
\defeq
-\big(
  K_{a}^{i,i}
\big)^{-1}
\Big(
  b_2^i(t,\x)^\top p
  +
    \big(
      K_{a\x}^{i,i}
    \big)^\transpose
    \x
\Big),
\end{align*}
and the function $\mathcal{E}^i_A$ corrects $\kappa^i$ when the minimum is attained on the boundary of $A^i$. Since $\kappa^i$ and $\boldsymbol{\kappa}$ do not depend on $\ba$, the existence of the function $\bc$ through the fixed-point problem~\eqref{eq:fp2} is trivial. For the same reason, the Nash system is coupled only through $\bar\bb$ and $\bar{f}^i$ and not through $\boldsymbol{\kappa}$. This leads to $V^i(t,\x,\ba^{-i})=V^i(t,\x)$ for all $i\in\I$, $\ba^{-i}\in A^{-i}$ and $(t,\x)\in Q_T^0$. Inserting the gradient, we obtain the resulting Nash system: find $\bV=(V^1,\dots,V^N)\colon Q_T\to\R^N$ that, for all $(t,\x)\in Q_T^0$ and $i\in\I$, satisfies 
\begin{align*}
  -\frac{\partial V^i}
  {\partial t}
  &=
  \frac{1}{2} \mathrm{Tr}\big(\bSigma\, \mathrm{D}_{\x}^2 V^i\, \bSigma^\top\big) 
  +
  \langle 
    \bb_1,\mathrm{D}_{\x}V^i
  \rangle
  +
  \sum_{j\in\I}
  \big\langle 
    \Phi^{j}(\mathrm{D}_{\x}
    V^j)
  ,
  \big(b_2^j\big)^\top
  \mathrm{D}_{\x}V^i
\big\rangle
+
f_1^i\\
&\quad
+
\frac12
\sum_{j\in\I}
\big\langle K_{a}^{i,j} 
  \Phi^{j}(\mathrm{D}_{\x}V^j),
  \Phi^{j}(\mathrm{D}_{\x}V^j)
\big\rangle
+
\sum_{j\in\I}
\big\langle K_{a\x}^{i,j} 
  \Phi^{j}(\mathrm{D}_{\x}V^j),\x
\big\rangle
+
\boldsymbol{\Xi}^i(\mathrm{D}_{\x}\bV),
\end{align*}
and for all $(t,\x)\in \partial Q_T$ the boundary condition
$
  V^i(t,\x)
  = 
  g^i(t,\x)
$ holds. Here, $\boldsymbol{\Xi}^i$ corrects for replacing the optimal control $\kappa^i(t,\x,\mathrm{D}_{\x}\bV(t,\x))$, restricted to $A^i$, in the explicit terms of the equation, with the unrestricted control $\Phi^i(t,\x,\mathrm{D}_{\x}\bV(t,\x))$. This in particular reveals the structure of the unconstrained problem with $A=\R^{d\times N}$ or the main term when $A$ is compact but large.
\end{example}



\begin{example}[Linear--quadratic games]\label{ex:LQ}
We consider the setting of Example~\ref{ex:1}, with the modifications that $\D=\R^n$ and $A=(\R^d)^N$. Here, the non-compact set $A$ is not covered by our assumptions. We nevertheless include this example, since it is a central case, admits a semi-explicit solution, and is used in our numerical experiment. We further specify
\begin{align*}
    \bb_1(t,\x)  &= F(\boldsymbol{\xi}-\x),  \quad
    b_2^i(t,\x)   = B^i,        \quad
    \bSigma(t,\x) = \CC,        \quad
    f_1^i(t,\x)   = \big\| \stateCostPlayeri\x\big\|_{\R^n}^2, \quad 
    g^i(\x)       = \big\|G^i\x\big\|_{\R^n}^2,
\end{align*}
for $i \in \I$, where $F\in\R^{n\times n}$, $\boldsymbol{\xi}\in\R^n$, $\CC\in\R^{n\times n}$, $B^i\in\R^{n\times d}$, $\stateCostPlayeri\in\R^{n\times n}$, and $G^i\in\R^{n\times n}$. The Nash system reads: find $\bV=(V^1,\dots,V^N)\colon Q_T\to\R^N$ that, for all $(t,\x)\in Q_T^0$ and $i\in\I$, satisfies 
\begin{align}
\begin{split}
      -\frac{\partial V^i}
  {\partial t}
  &=
  \frac{1}{2} \mathrm{Tr}\big(\CC\, \mathrm{D}_{\x}^2 V^i\, \CC^\top\big) 
  +
  \big\langle 
    F(\boldsymbol{\xi}-\x),\mathrm{D}_{\x}V^i
  \big\rangle
+
\sumAllPlayersj
\big\langle
    \Phij,
    \big(
        \controlMatrixPlayerj
    \big)^\transpose
    \gradient
    \valueFunctionPlayeri
\big\rangle
\\
&\qquad
+
\big\|\stateCostPlayeri\x\big\|_{\R^n}^2
+
\frac12
\sum_{j\in\I}
\big\langle K_{a}^{i,j}
\Phij,
\Phij
\big\rangle
+
\sumAllPlayersj
\big\langle K_{a\x}^{i,j} 
\Phij,
\x
\big\rangle,
\end{split} \label{eq:HJB_equation_LQG_example}
\end{align}
and for all $(t,\x)\in\partial Q_T$ the boundary condition
$V^i(t,\x)
=
\|G^i\x\|_{\R^n}^2$. Here, $\Phii$ is the straightforward modification of $\Phi^i$ given by
\begin{align}
    \Phii
    =
    -   
        \big(
            K_{a}^{i,i}
        \big)^{-1}
        \Big(
            \big(
                \controlMatrixPlayeri
            \big)^\transpose
            \gradient
            \valueFunctionPlayeri
        +   
            \big(
                K_{a\x}^{i,i}
            \big)^\transpose
            \x
        \Big).
    \label{eq:Phii_Ex}
\end{align}
The system~\eqref{eq:HJB_equation_LQG_example} admits a semi-analytical solution of the form
\begin{align}\label{eq:LQ_Vi}
    \valueFunctionPlayeri(t,\x)
    =
        \frac{1}{2}
        \left\langle
            \x,
            P^i_t
            \x
        \right\rangle
    +
        \left\langle
            Q^i_t,
            \x
        \right\rangle
    +
        R^i_t,
\end{align}
where $P^i \colon [0,T]\to \R^{n\times n}$, $Q^i \colon[0,T]\to\R^n$ and $R^i\colon[0,T]\to\R$, $i\in\I$, solve the coupled Riccati system
\begin{align}
    \begin{split}\label{eq:Riccati_system}
        \dot{P}^i
    &=
        \attractionMatrix^\transpose
        P^i
    +   
        P^i
        \attractionMatrix
    -
        \big(
            \stateCostPlayeri
        \big)^\transpose
        \stateCostPlayeri
    -
        2
        \sumAllPlayersj
        \big(
            \PPshort^i
            \big(
                \ThetaOne
                \PPshort^j
            +
                \ThetaTwo       
            \big)
        +
            \PPshort^j
            \ThetaThree
            \PPshort^j
        +
            \ThetaFourFiveSeven
            \PPshort^j
        +   
            \ThetaSixEight
        \big),
    \\
    \dot{Q}^i
    &=
    -
        \PPshort^i
        \attractionMatrix
        \boldsymbol{\xi}
    +
        \attractionMatrix^\transpose
        Q^i
    -
        \sumAllPlayersj
        \big(
            \PPshort^i
            \ThetaOne
            Q^j
        +
            \PPshort^j
            \big(
                \ThetaOne
            \big)^\transpose
            Q^i
        +
            \big(
                \ThetaTwo
            \big)^\transpose
            Q^i
        +  
            \PPshort^j
            \ThetaThreeTilde
            Q^j
        +
            \ThetaFourFiveSeven
            Q^j
        \big),
    \\
    \dot{R}^i
    &=
    -
        \frac{1}{2} \Tr 
        \big( 
            \diffusionMatrix 
            \PPshort^i
            \diffusionMatrix^\transpose 
        \big)
    -
        \left\langle 
                \attractionMatrix 
                \boldsymbol{\xi},
                Q^i
        \right\rangle
    -
        \sumAllPlayersj
        \big(
            \big\langle
                Q^i,
                \ThetaOne
                Q^j
            \big\rangle
        +
            \big\langle
                Q^j,
                \ThetaThree
                Q^j
            \big\rangle
        \big).
    \end{split}
\end{align}
Here, the notation $\PPshort^i = \tfrac{1}{2} \big( P^i + (P^i)^\transpose\big)$ has been used for brevity, as well as the matrices
\begin{align*}
    \ThetaOne 
    &= 
    -
        \controlMatrixPlayerj
        \left(
            \controlCostPlayerjj
        \right)^{-1}
        \left(
            \controlMatrixPlayerj
        \right)^\transpose
    , \\
    \ThetaTwo
    &=
    -
        \controlMatrixPlayerj
        \left(
            \controlCostPlayerjj
        \right)^{-1}
        \left(
            \controlstateCostPlayerjj
        \right)^\transpose,
    \\
    \ThetaThree
    &= 
        \tfrac{1}{2}
        \controlMatrixPlayerj   
        \left(
            \controlCostPlayerjj
        \right)^{-\transpose}
        \controlCostPlayerij
        \left(
            \controlCostPlayerjj
        \right)^{-1}
        \left(
            \controlMatrixPlayerj
        \right)^\transpose,
    \\
    \ThetaThreeTilde
    &=
        \tfrac{1}{2}
        \controlMatrixPlayerj   
        \left(
            \controlCostPlayerjj
        \right)^{-\transpose}
        \Big(
            \left(
                \controlCostPlayerij
            \right)^\transpose
        +
            \controlCostPlayerij
        \Big)
        \left(
            \controlCostPlayerjj
        \right)^{-1}
        \left(
            \controlMatrixPlayerj
        \right)^\transpose,
    \\
    \ThetaFourFiveSeven
    &=
        \Big(
            \tfrac{1}{2}
            \controlstateCostPlayerjj
            \big(
                \controlCostPlayerjj
            \big)^{-\transpose}
            \Big(
                \big(
                    \controlCostPlayerij
                \big)^\transpose
            +
                \controlCostPlayerij
            \Big)
        -   
            \controlstateCostPlayerij
        \Big)
        \left(
            \controlCostPlayerjj
        \right)^{-1}
        \left(
            \controlMatrixPlayerj
        \right)^\transpose,
    \\
    \ThetaSixEight
    &=
        \Big(
            \tfrac{1}{2}
            \controlstateCostPlayerjj
            \left(
                \controlCostPlayerjj
            \right)^{-\transpose}
            \controlCostPlayerij
        -
            \controlstateCostPlayerij
        \Big)
        \left(
            \controlCostPlayerjj
        \right)^{-1}
        \left(
            \controlstateCostPlayerjj
        \right)^\transpose.
\end{align*}
\end{example}

\begin{remark}\label{rem:bdd_domain}
For FBSDEs with one-dimensional backward process, i.e., for stochastic control, \cite{delarue2006weak} 
develops a theory of (probabilistically) weak solutions to the corresponding FBSDE. Under general and for applications very natural assumptions, existence and uniqueness of weak solutions are guaranteed, and through the FBSDE also a solution to the HJB equation. It would be desirable with an analogous theory for the Nash FBSDE system, but, to the best of our knowledge, it has not been established.
\end{remark}

\section{Fictitious play and its convergence}\label{sec:FP}

System~\eqref{eq:FBSDE_full} is a coupled FBSDE with an $N$-dimensional BSDE component, leading to significant computational complexity. To address this, we adopt the concept of fictitious play, originally introduced for zero-sum games in~\cite{brown1949some} and further developed in~\cite{brown1951iterative,robinson1951iterative}. The approach has since been extended to stochastic differential games (see, e.g., \cite{elliott1975stochastic}), and has recently gained renewed attention in related computational approaches~\cite{hu2019deep,han2020deep}.

Fictitious play is an iterative best-response scheme: at each iteration, every player updates their strategy by solving an optimal control problem assuming the opponents’ strategies are fixed. In this stochastic differential game setting, fixing the strategies of the other players reduces each player’s problem to a single-agent stochastic control problem. The control problem admits an FBSDE representation, whose solution determines the player’s optimal feedback response. Repeating this procedure generates a sequence of strategies that, under suitable conditions, converges to a Nash equilibrium.

Throughout this section, we fix the initial value $\x_0\in \D$. Recalling the family $(\X^{t,\x},\Y^{t,\x},\Z^{t,\x})$ from~\eqref{eq:FBSDE_full}, our aim is to approximate $(\X,\Y,\Z):=(\X^{0,\x_0},\Y^{0,\x_0},\Z^{0,\x_0})$. The fictitious-play iteration starts from the zero Markov policy $\alpha^{i,0} = 0 \in \A^i$ for each player $i \in \I$. For $m \geq 1$, each player $i \in \I$ solves~\eqref{eq:FBSDE_single_player} with $\bbeta^{-i} = \bal^{-i,m-1}$. This yields a non-equilibrium Markov map $\zeta^{i,m}\colon [0,T]\times\R^n\to\R^n$, and the policy is updated by
\begin{align}
    \label{eq:FP_construction}
    \alpha^{i,m}(t, \x)
    :=
    \kappa^i 
    \big(
        t,\x,\bPhi(t,\x)\zeta^{i,m},\bal^{-i,m-1}(t,\x)
    \big).
\end{align}
The iterative scheme leads to FBSDEs whose solution triples are denoted by $\big(\X^{i,m},\Y^{i,m},\Z^{i,m}\big)$.

The remainder of this section is devoted to the convergence analysis of the fictitious-play scheme. To quantify the error at iteration $m$, we introduce
\begin{align}
    \delta \X_t^{i,m} 
  \defeq \X_t-\X_t^{i,m}, \qquad \delta Y_t^{i,m} 
  \defeq Y_t^i-Y_t^{i,m}. \label{eq:error_quantities}
\end{align}
In Section~\ref{sec:reformulation}, a reformulation of the fictitious-play FBSDE is presented. Then, under the assumptions stated in Section~\ref{sec:setting3}, we show in Section~\ref{sec:conv_FP} that the error quantities in~\eqref{eq:error_quantities}, together with the associated feedback controls, converge to zero at a geometric rate as $m \to \infty$.

\subsection{A reformulation of the fictitious play FBSDEs}\label{sec:reformulation}

The coefficients $\tilde\bb^i$ and $\tilde\ell^i$ in~\eqref{eq:FBSDE_single_player} depend on arbitrary policies $\bbeta^{-i}\in\A^{-i}$. For the convergence analysis, we reformulate the corresponding FBSDEs in terms of the Markov maps generated by the fictitious-play iteration. For each $i \in \I$ and $m \geq 1$, we introduce functions $\bb^{i,m}\colon Q_T\times \R^n\times\R^{(N-1)\times n}\to\R^n$ and $\ell^{i,m}\colon Q_T\times \R^n\times\R^{n\times (N-1)}\to\R$ such that, at iteration $m$, player~$i$ solves the FBSDE
\begin{align}
\label{eq:FP_FBSDE}
\begin{split}
    \X_t^{i,m}\! 
    &=\! \x_0
    + \int_0^{t\wedge\tau^{i,m}} 
    \bb^{i,m}\big(
    s,\X_s^{i,m},
    Z_s^{i,m},
    \bzeta^{-i,m-1}\big(s,\X_s^{i,m}\big)
    \big)\,\d s
    + \int_0^{t\wedge\tau^{i,m}} 
    \bSigma\big(s,\X_s^{i,m}\big)\,\d\W_s,
    \\
    Y_t^{i,m}\!
    &=\! g^{i}\big(\tau^{i,m},\X_{\tau^{i,m}}^{i,m}\big) 
    \!+\! \int_{t\wedge\tau^{i,m}}^{\tau^{i,m}}\! 
    \ell^{i,m}\big(
    s,\X_s^{i,m},
    Z_s^{i,m},
    \bzeta^{-i,m-1}(s,\X_s^{i,m})
    \big)\,\d s
    \!-\! \int_{t\wedge\tau^{i,m}}^{\tau^{i,m}}\! 
    \big(Z_s^{i,m}\big)^\top \d\W_s,
\end{split}
\end{align}
where $\zeta^{i,m} : [0,T] \times \R^n \to \R^n$ satisfies for almost all $t\in[0,T]$, $\P$-almost surely
\begin{align}
    Z_t^{i,m}=\zeta^{i,m}\big(t,\X_t^{i,m}\big) \label{eq:Z_FP}.
\end{align}
We now identify the coefficients in this reformulation. For the moment, assume that the Markov maps $\zeta^{i,m}$ exist. Sufficient assumptions for this are stated in Section~\ref{sec:setting3}. For $i \in \I$, set $\kappa^{i,0} \equiv 0$ and define for $m \geq 1$ the functions $\kappa^{i,m}\colon Q_T\times\R^n\to A^i$ by
\begin{align*}
    \kappa^{i,m}\big(t,\x,z\big)
    \defeq
    \kappa^i
    \big(
        t,
        \x,
        \bPhi(t,\x)z,
        \boldsymbol\kappa^{-i,m-1}\big(
        t,\x,\bzeta^{-i,m-1}(t,\x)
        \big)
    \big),
\end{align*}
where for \((t,\x,\boldsymbol z^{-i})\in Q_T\times\R^{(N-1)\times n}\),
\begin{align*}
    &\boldsymbol\kappa^{-i,m}\big(t,\x,\boldsymbol z^{-i}\big)
    \defeq
    \big(
        \kappa^{1,m}\big(t,\x,z^1\big),
        \ldots,
        \kappa^{i-1,m}\big(t,\x,z^{i-1}\big),
        \kappa^{i+1,m}\big(t,\x,z^{i+1}\big),
        \ldots,
        \kappa^{N,m}\big(t,\x,z^N\big)
    \big).
\end{align*}
When necessary, we interpret $\zeta^{i,0} \equiv 0$. We claim that the coefficients $\bb^{i,m}$ and $\ell^{i,m}$ are given by
\begin{align*}
    \bb^{i,m}\big(t,\x,z,\boldsymbol z^{-i}\big) 
    &\defeq
    \tilde\bb^i\big(
    t,\x,
    z,
    \boldsymbol\kappa^{-i,m-1}\big(
    t,\x,\boldsymbol z^{-i}
    \big)
    \big),
    \\
    \ell^{i,m}\big(t,\x,z,\boldsymbol z^{-i}\big)
    &\defeq
    \tilde\ell^i\big(
    t,\x,
    z,
    \boldsymbol\kappa^{-i,m-1}\big(
    t,\x,\boldsymbol z^{-i}
    \big)
    \big).
\end{align*}
To prove the claim, it suffices to identify the opponent policy appearing in $\tilde\bb^i$ and $\tilde\ell^i$ at iteration $m$. By construction, this policy is $\bal^{-i,m-1}$. Thus, the claim follows once we show for every $r\geq 0$ that
\begin{align}
    \label{eq:induction_proof_relation_first}
    \bal^{-i,r}(t,\x) = \boldsymbol\kappa^{-i,r}\big(t,\x,\bzeta^{-i,r}(t,\x)\big).
\end{align}
Equivalently, it is sufficient to prove for every $i \in \I$ that
\begin{align}
    \label{eq:induction_proof_relation}
    \alpha^{i,r}(t,\x) = \kappa^{i,r}\big(t,\x,\zeta^{i,r}(t,\x)\big).
\end{align}
We prove the claim by induction. For $r=0$, the relation~\eqref{eq:induction_proof_relation} holds trivially, since both sides are initialized to the zero Markov policy. Assume that~\eqref{eq:induction_proof_relation}, equivalently~\eqref{eq:induction_proof_relation_first}, holds for some fixed $r \geq 0$. We show that it then holds for $r+1$. Using the construction~\eqref{eq:FP_construction}, the induction hypothesis, and the definition of $\kappa^{i,r}$, we obtain
\begin{align*}
\alpha^{i,r+1}(t,\x)
&=
\kappa^i\bigl(
t,x,
\bPhi(t,x)\zeta^{i,r+1}(t,\x),
\bal^{-i,r}(t,\x)
\bigr)\\
&=
\kappa^i\bigl(
t,\x,
\bPhi(t,\x)\zeta^{i,r+1}(t,\x),
\boldsymbol\kappa^{-i,r}
\big(t,\x,\bzeta^{-i,r}(t,\x)\big)
\bigr)\\
&=
\kappa^{i,r+1}
\bigl(t,\x,\zeta^{i,r+1}(t,\x)\bigr).
\end{align*}
This completes the proof.

\color{black}

\subsection{Assumptions}
\label{sec:setting3}
Let the setting from Sections~\ref{sec:setting2}--\ref{sec:FBSDE_Nash_system} hold.
For technical reasons, we use the convention that $\bb,\bSigma,\f,\bzeta$ are killed at $\tau$, while $\bb^{i,m},\ell^{i,m},\bzeta^{i,m}$ are killed at $\tau^{i,m}$, by multiplication with the indicator functions $1_{\{t\leq\tau\}}$ and $1_{\{t\leq\tau^{i,m}\}}$, respectively. Thus, these coefficients vanish on the boundary $\partial Q_T$, so that the processes $\X,Y$ and $\X^{i,m},Y^{i,m}$ remain constant after $\tau$ and $\tau^{i,m}$, respectively. This convention does not redefine the coefficient functions and therefore does not affect their assumptions.

For all $i \in \I$ and $m \ge 1$, we assume that~\eqref{eq:FP_FBSDE} admits a unique solution
\begin{align*}
    \big(\X^{i,m}, Y^{i,m}, Z^{i,m}\big) \in \mathbb S^{2,n}_{T} \times \mathbb S^{2,1}_{T} \times \mathbb H^{2,n}_{T}.
\end{align*}
Moreover, for all $i \in \I$ and $m \geq 0$, there exists $L_{\zeta^{i,m}}^{\FP}$ such that for all $(t_1,\x_1),(t_2,\x_2)\in Q_T$ we have
\begin{align}
    \big\|
      \zeta^{i,m}(t_1,\x_1)
      -
      \zeta^{i,m}(t_2,\x_2)
    \big\|^2_{\R^{n}}
    &\leq
    L_{\zeta^{i,m}}^{\FP}
    \big(|t_1-t_2|
    +
    \|\x_1-\x_2\|_{\R^n}^2\big). \label{eq:lipschitz_zetaim}
\end{align}
In addition, we assume that
\begin{align}
\label{eq:LFP}
    \LFP
    :=
    \sup_{m\in\N}
    \max_{i\in\I}
    L_{\zeta^{i,m}}^{\FP}
    <\infty, 
\end{align}
which is an assumption that we do not verify theoretically. 
Moreover, we assume that there exists $L_{\bzeta}\geq0$ such that for all $(t_1,\x_1),(t_2,\x_2)\in Q_T$ we have
\begin{align*}
    \tnorm{\bzeta(t_1,\x_1)-\bzeta(t_2,\x_2)}^2_{\R^{n\times N}}
    &\leq
    L_{\bzeta}
    \big(|t_1-t_2|
    +
    \|\x_1-\x_2\|_{\R^n}^2\big).
\end{align*}
These are abstract assumptions whose validity must be verified in any concrete application. See Remark~\ref{remark:gradient_bound} for further discussion.

We introduce the following Lipschitz condition (L). There exist constants $L_{\bb}, L_{\f}, L_\kappa^{\FP}, L_{\bb}^{\FP}, L_{\ell}^{\FP}\geq 0$ such that, for all $i \in \I$, $m \geq 1$, and all admissible arguments,
\begin{align}
    \begin{aligned}
        \|\bb(t_1,\x_1,\z_1)-\bb(t_2,\x_2,\z_2)\|_{\R^{n}}^2
        &\leq
        L_{\bb}
        \big(
          |t_1-t_2|
          +
          \|\x_1-\x_2\|_{\R^n}^2
          +
          \tnorm{\z_1 - \z_2}^2_{\R^{n\times N}}
        \big),\\
        \|\f(t_1,\x_1,\z_1)-\f(t_2,\x_2,\z_2)\|_{\R^{N}}^2
        &\leq
        L_{\f}
        \big(
          |t_1-t_2|
          +
          \|\x_1-\x_2\|_{\R^n}^2
          +
          \tnorm{\z_1 - \z_2}^2_{\R^{n\times N}}
        \big),
            \\
        \big\|
            \kappa^{i,m}(t_1,\x_1,z_1)
            -
            \kappa^{i,m}(t_2,\x_2,z_2)
        \big\|_{\R^d}^2
        &\leq
        L_\kappa^{\FP}
        \big(
          |t_1-t_2|
          +
          \|\x_1-\x_2\|_{\R^n}^2
          +
          \| z_1 - z_2 \|^2_{\R^{n}}
        \big),\\    
        \big\|
          \bb^{i,m}(t_1,\x_1,z_1^i,\z_1^{-i})-\bb^{i,m}(t_2,\x_2,z^i_2,\z_2^{-i})
        \big\|_{\R^{n}}^2
        &\leq
        L_{b}^{\FP}
        \big(
          |t_1-t_2|
          +
          \|\x_1-\x_2\|_{\R^n}^2
          +
          \tnorm{\z_1 - \z_2}^2_{\R^{n\times N}}
        \big),\\    
        \big|
          \ell^{i,m}(t_1,\x_1,z_1^i,\z_1^{-i})
          -
          \ell^{i,m}(t_2,\x_2,z_2^i,\z_2^{-i})
        \big|^2
        &\leq
        L_{\ell}^{\FP}
        \big(
          |t_1-t_2|
          +
          \|\x_1-\x_2\|_{\R^n}^2
          +
          \tnorm{\z_1 - \z_2}^2_{\R^{n\times N}}
        \big).
    \end{aligned}
    \tag{L}
\end{align}
Here, admissible arguments means either all $\z_1, \z_2\in\mathbb R^{n\times N}$, in the global
case, or all $\z_1,\z_2$ in a bounded subset of $\mathbb R^{n\times N}$, in the local case. 

We introduce two alternative sets of assumptions and require that at least one of them holds.

\begin{enumerate}[label=(\Alph*)]

\item \label{assump:gradient}\textbf{Gradient bound condition:} The domain is either $\D = \R^n$ or $|\D| < \infty$ with boundary $\partial \D$ of class $C^2$. For the corresponding fictitious-play iterations, there exists a unique collection $\{V^{i,m}
\}_{i\in\I,\; m\ge1}
\subset C^{1,2}(\bar Q_T)$
such that, for all $i\in\I$, $m\ge1$, and $(t,\x)\in Q^0_T$, it holds
\begin{align}\label{eq:HJB-best-response}
\frac{\partial V^{i,m}}{\partial t}(t,\x)
+ \mathcal{L}_tV^{i,m}(t,\x)
+
H^{i,m}
\big(
t,\x,
\mathrm D_{\x}V^{i,m}(t,\x),
\bzeta^{-i,m-1}(t,\x)
\big)
=0,
\end{align}
and for all $(t,\x)\in\partial Q_T$,
\[
V^{i,m}(t,\x)=g^i(t,\x).
\]
Here, for $(t,\x,p,\boldsymbol z^{-i})\in Q_T\times\R^n\times\R^{n\times (N-1)}$, the Hamiltonian is given by
\[
H^{i,m}(t,\x,p,\boldsymbol z^{-i})
=
\big\langle
\bb^{i,m}
\big(
t,\x,
\bSigma^\top(t,\x)p,
\boldsymbol z^{-i}
\big),
p
\big\rangle
+
\ell^{i,m}
\big(
t,\x,
\bSigma^\top(t,\x)p,
\boldsymbol z^{-i}
\big).
\]
Moreover, for $m\ge1$, we set
\[
\zeta^{i,m}(t,\x)
=
\bSigma^\top(t,\x)\mathrm D_{\x}V^{i,m}(t,\x).
\]
Finally, there exists $\MFP < \infty$ such that
\begin{align}\label{eq:grad_bound}
\sup_{m\in\N}\max_{i\in\I}\sup_{(t,\x)\in Q_T}
\big\|
  \mathrm D_{\x} V^{i,m}(t,\x)
\big\|_{\R^n}^2
\leq
\MFP.
\end{align}

\item \label{assump:smalltime} \textbf{Smallness condition:} The domain is $\D=\R^n$ and there exist $\beta > 0$ and $\eta > 0$ such that
\begin{align*}
    2C_{\beta} < 1, 
    \qquad
    \frac{2C_\beta}
         {1 - 2C_\beta}(N-1) < 1,
    \qquad
    (1+\eta)L_\kappa^a < 1, 
\end{align*}
and
\begin{align*}
    \big(1 + \eta^{-1}\big) L_\kappa^a L_{\kappa}^{\FP} \frac{2C_\beta' N}{1-2C_\beta} < \Big( 1 - \frac{2C_\beta}
         {1 - 2C_\beta}(N-1) \Big) \big(1 - (1+\eta)L_\kappa^a \big),
\end{align*}
where
\begin{align*}
    C_\beta
    &\defeq\frac{4L_\ell^{\FP}}\beta
           +Ce^{\beta T}
           \bigg(L_g+\frac{2L_\ell^{\FP} T(1+2L_\zeta^{\FP} N)}\beta
           +
           L_\zeta^{\FP} T\bigg),
    \\
    C_\beta'
    &\defeq
    \frac{2 L_{\bar{\boldsymbol{f}}} (1+L_\kappa)}{\beta}
           +Ce^{\beta T}
           \bigg(L_g+\frac{2L_\ell^{\FP} T(1+2L_\zeta^{\FP} N)}\beta
           +
           L_\zeta^{\FP} T\bigg).
\end{align*}
The constant $C$ is derived in Lemma~\ref{lemma:delta_x_bound} and includes the non-explicit constant of the Burkholder--Davis--Gundy inequality. Furthermore, the Lipschitz condition (L) holds globally.

\end{enumerate}

Under Assumption~\ref{assump:gradient}, the gradient bound~\eqref{eq:grad_bound} yields a uniform bound on the relevant $\z$-argument, and hence the local version of the Lipschitz condition~(L) holds. Under Assumption~\ref{assump:smalltime}, no such bound is available, so the global version of~(L) is imposed directly. A notable case in which the global version can be verified from the preceding Lipschitz assumptions is additive time-independent noise. Appendix~\ref{sec:lipschitz} contains the detailed verification of the Lipschitz condition~(L).

\begin{remark}\label{remark:gradient_bound}
Assumptions \eqref{eq:lipschitz_zetaim}, \eqref{eq:LFP} and~\eqref{eq:grad_bound} are abstract assumptions that need to be verified in specific settings. For $\D=\R^n$, the gradient estimate~\eqref{eq:grad_bound} holds under strong regularity assumptions, for instance by~\cite[Theorem~4.3]{fleming2006controlled}, whereas~\cite{chaintron2023existence} obtains such estimates under substantially weaker assumptions, namely Lipschitz continuity and linear growth of the coefficients.

For $|\D|<\infty$, however, we are not aware of a directly applicable result in the literature for the semilinear fictitious-play HJB system considered here. The parabolic Hölder-space theory for linear equations (see, e.g.,~\cite[Theorem 10.3.3]{krylov1996lectures}) provides existence, uniqueness, and a priori estimates, which in particular yield bounds on spatial derivatives under appropriate assumptions. Extending such estimates to the present semilinear setting would require an additional argument, most naturally based on a fixed-point approach combined with Schauder estimates, for instance via Banach’s fixed point theorem, Schauder’s fixed point theorem, or the Leray--Schauder fixed point theorem.

\end{remark}

\subsection{Convergence analysis}\label{sec:conv_FP}

The goal of this subsection is to establish the main convergence result for the fictitious-play scheme. To quantify the error at iteration $m$, we define
\begin{align}
    \Psi^{m} \defeq  \max_{i\in\I}\big\|\delta\X^{i,m}\big\|^2_{\mathbb S^{2,n}_{T}}
            +&
            \sum_{i\in\I}\big\|\delta Y^{i,m}\big\|^2_{\mathbb S^{2,1}_{T}}
            +
            \sum_{i\in\I}\big\|\zeta^{i,[\tau]}(\cdot,\X_\cdot)-\zeta^{i,m,[\tau^{i,m}]}\big(\cdot,\X_\cdot^{i,m}\big)\big\|_{\mathbb H^{2,n}_{T}}^2. \label{eq:Psim}
\end{align}
Our main theorem shows that, under the assumptions of Section~\ref{sec:setting3}, this error decays geometrically as $m \to \infty$. More precisely, we prove an estimate of the form
\begin{align*}
    \Psi^{m} \lesssim CK^m, 
\end{align*}
for $C < \infty$ and $K \in (0,1)$. Under Assumption~\ref{assump:gradient}, in special cases where the policy-mismatch error vanishes, the estimate can be sharpened to a super-exponential rate. This refinement is discussed after the main theorem.

The proof is divided into three steps. We first derive a decomposition of the error $\Psi^{m}$ into stopping-time and Markov-map contributions. Then, we bound the stopping-time error in terms of the Markov-map error, and finally estimate the Markov-map error itself.  These results are combined in the last part, where we state and prove the main theorem. We begin by introducing the notation used throughout the convergence analysis.

\subsubsection{Notation}

For $i \in \I$, $m \geq 1$ and $t \in [0,T]$, define the coefficient differences by
\begin{align}
    \delta \bb^{i,m}_t&\defeq\bb\big(t,\X_t,\zeta^i(t,\X_t),\bzeta^{-i}(t,\X_t)\big)-\bb^{i,m}\big(t,\X_t^{i,m},\zeta^{i,m}\big(t,\X_t^{i,m}\big),\bzeta^{-i,m-1}\big(t,\X_t^{i,m}\big)\big),  \label{eq:definition_deltabi} \\
        \delta \bSigma_t^{i,m}
        &\defeq
        \bSigma(t,\X_t) - \bSigma\big(t,\X_t^{i,m}\big), \label{eq:definition_deltasigmai} \\
      \delta \ell_t^{i,m}
      &\defeq
      f^i\big(t,\X_t,\zeta^i(t,\X_t),\bzeta^{-i}(t,\X_t)\big)- \ell^{i,m}\big(t,\X_t^{i,m},\zeta^{i,m}\big(t,\X_t^{i,m}\big),\bzeta^{-i,m-1}\big(t,\X_t^{i,m}\big)\big),\\
\delta g^{i,m}
&\defeq
g^i(\tau,\X_\tau)- g^i\big(\tau^{i,m},\X_{\tau^{i,m}}^{i,m}\big).
    \end{align}   
The Markov maps require a more delicate analysis, for which we need to distinguish between the errors in the Markov maps themselves and errors caused by evaluating the same map along different state processes. For $i,j\in\I$, $m,k\in\N$ and \(t\in[0,T]\) we therefore define 
\begin{align*}
\delta \zeta_t^{j,i,m,k} 
&\defeq
\zeta^{j}(t,\X_t)-\zeta^{j,m}\big(t,\X_t^{i,k}\big),
&\qquad
\delta\bzeta_t^{-i,m,k} 
&\defeq \bzeta^{-i}(t,\X_t)-\bzeta^{-i,m}\big(t,\X_t^{i,k}\big),\\
\delta \bar{\zeta}_t^{i,m} 
&\defeq \zeta^{i}(t,\X_t)-\zeta^{i,m}(t,\X_t),
&\qquad
\delta \bar{\bzeta}_t^{-i,m} 
&\defeq \bzeta^{-i}(t,\X_t)-\bzeta^{-i,m}(t,\X_t),\\
\delta \tilde{\zeta}_t^{j,i,m,k} 
&\defeq \zeta^{j,m}(t,\X_t)-\zeta^{j,m}\big(t,\X_t^{i,k}\big),
&\qquad
\delta \tilde{\bzeta}_t^{-i,m,k} 
&\defeq \bzeta^{-i,m}(t,\X_t)-\bzeta^{-i,m}\big(t,\X_t^{i,k}\big).
\end{align*}
Here, the bar notation denotes the difference between the true Markov map and its current approximation evaluated at the same state process, whereas the tilde notation captures the effect of evaluating the same approximation at different state processes. In particular, we have the decomposition
\begin{align}
\label{eq:split}
\delta\zeta_t^{j,i,m,k} = \delta\bar{\zeta}_t^{j,m}+\delta\tilde{\zeta}_t^{j,i,m,k}.
\end{align}
We also use the shorthand notation
\begin{align*}
    \delta\zeta^{i,m,k}_t \defeq \delta\zeta^{i,i,m,k}_t, \quad
    \delta\zeta^{j,i,m}_t \defeq \delta\zeta^{j,i,m,m}_t, \quad
    \delta\zeta^{i,m}_t \defeq \delta\zeta^{i,i,m,m}_t,
\end{align*}
and analogously for the tilde notation. Furthermore, for the stopping times, we write
\begin{align*}
    \delta \tau^{i,m} \defeq \tau - \tau^{i,m}, \quad \tau_{\min}^{i,m} \defeq \tau\wedge\tau^{i,m}, \quad \tau_{\max}^{i,m} \defeq \tau\vee\tau^{i,m}.
\end{align*}
Finally, we define two error quantities that are central in the analysis. For $i \in \I$ and $m \geq 0$, denote the player-specific policy mismatch by
\begin{align*}
    \Delta^{i,m}(t,\x) 
    := 
    \alpha^{*,i}(t,\x) 
    - 
    \kappa^{i,m}
    \big(
        t, \x, \zeta^i(t,\x)
    \big),
\end{align*}
and let the policy error be quantified by
\begin{align*}
    E_\Delta^m := \sum_{i \in \I}
    \big\|
        \Delta^{i,m}(\cdot, \X_{\cdot})
    \big\|^2_{\mathbb{H}^{2,d}_T}.
\end{align*}
For the global Markov-map error, we introduce
\begin{align*}
    E_\zeta^m \defeq \sum_{i \in \I} \big\| \delta \bar\zeta^{i,m} \big\|^2_{\mathbb{H}^{2,n}_T}.
\end{align*}

\subsubsection{State error decomposition}

The purpose of this subsection is to estimate the different terms in the error quantity $\Psi^{m}$ defined in~\eqref{eq:Psim}. The main goal is to express the approximation errors in $\X$, $Y$ and $\zeta$ in terms of the stopping-time error $\delta \tau$, the Markov-map error $\delta \bar \zeta$ and the policy-mismatch error.

First, Lemma~\ref{lemma:delta_zeta_decomposition} provides a decomposition of $\delta \zeta$ that will be used repeatedly throughout the analysis. Second, Lemma~\ref{lem:delta_bi_bound} derives bounds for the coefficient differences $\delta \bb^{i,m}_t$ and $\delta \ell^{i,m}_t$. Lemma~\ref{lemma:delta_x_bound} then estimates the forward error $\delta \X$, while Lemma~\ref{lemma:delta_y_bound} controls the backward error $\delta Y$ together with the associated $\zeta$-term. Combining these estimates yields Proposition~\ref{prop:deltaXYZ}, which gives the desired decomposition of the overall error $\Psi^{m}$.

\begin{lemma}\label{lemma:delta_zeta_decomposition}
For any $t\in[0,T]$, $i,j\in\I$, \(\beta\geq 0\), and $m,l\in\N$, it holds
\[
\Big\|\delta\zeta^{j,i,l,m,\,[\tau_{\min}^{i,m}]}\Big\|_{\mathbb H^{2,n}_{\beta,t}}^{2}
\;\le\;
2\,\Big\|\delta\bar\zeta^{j,l,\,[\tau_{\min}^{i,m}]}\Big\|_{\mathbb H^{2,n}_{\beta,t}}^{2}
\;+\;
2\,L_\zeta^{\FP}\,
\Big\|\delta \X^{i,m,\,[\tau_{\min}^{i,m}]}\Big\|_{\mathbb H^{2,n}_{\beta,t}}^{2}.
\]
\end{lemma}
\begin{proof}
The result follows from the split~\eqref{eq:split}, the elementary inequality
\begin{align}
    \left\| a_1 + a_2 \right\|^2_{\mathcal{V}} 
    \leq 
        2 \left\| a_1 \right\|^2_{\mathcal{V}} 
    + 
        2 \left\| a_2 \right\|^2_{\mathcal{V}}, 
    \quad a_1, a_2 \in \mathcal{V}, \label{ineq:sumsq}
\end{align}
for any normed vector space $\mathcal{V}$, and the Lipschitz continuity of $\zeta^{j,l}$.
\end{proof}

\begin{lemma}\label{lem:delta_bi_bound}
    Fix $i \in \I$ and $m \geq 1$. Then, there exists $C < \infty$ such that
    \begin{align*}
        &\Big\|
            \delta \bb^{i,m,[\tau^{i,m}_{\min}]}
        \Big\|^2_{\mathbb{H}^{2,n}_T}
        +
        \Big\|
            \delta \ell^{i,m,[\tau^{i,m}_{\min}]}
        \Big\|^2_{\mathbb{H}^{2,1}_T}
        \\
        &\qquad \leq
        C
        \bigg(
            \Big\|
                \delta \X^{i,m,[\tau^{i,m}_{\min}]}
            \Big\|^2_{\mathbb{H}^{2,n}_T}
            +
            \Big\|
                \delta \bar{\zeta}^{i,m,[\tau^{i,m}_{\min}]}
            \Big\|^2_{\mathbb{H}^{2,n}_T}
            +
            \sum_{j\in \I\backslash\{i\}}
            \Big\|
                \delta \bar{\zeta}^{j,m-1,[\tau^{i,m}_{\min}]}
            \Big\|^2_{\mathbb{H}^{2,n}_T}
            +
            E_\Delta^{m-1}
        \bigg).
    \end{align*}
\end{lemma}
\begin{proof}
    Using~\eqref{ineq:sumsq}, we can split the error as
    \begin{align*}
        \Big\|
            \delta \bb^{i,m,[\tau^{i,m}_{\min}]}
        \Big\|^2_{\mathbb{H}^{2,n}_T}
        \leq 
        2\,
        \Big\|
            \delta\bar{\bb}^{i,m,[\tau^{i,m}_{\min}]}
        \Big\|^2_{\mathbb{H}^{2,n}_T}
        +
        2\,
        \Big\|
            \delta \tilde{\bb}^{i,m,[\tau^{i,m}_{\min}]}
        \Big\|^2_{\mathbb{H}^{2,n}_T}
    \end{align*}
    where
    \begin{align*}
        \delta \bar{\bb}^{i,m}_{t} 
        &:=
        \bb
        \big(
            t,\X_t,\zeta^i(t,\X_t),\bzeta^{-i}(t,\X_t)
        \big)
        -
        \bb^{i,m}
        \big(
        t,\X_t,\zeta^{i}\big(t,\X_t\big),\bzeta^{-i}\big(t,\X_t\big)
        \big), \\
        \delta \tilde{\bb}^{i,m}_{t} 
        &:=
        \bb^{i,m}
        \big(
        t,\X_t,\zeta^{i}\big(t,\X_t\big),\bzeta^{-i}\big(t,\X_t\big)
        \big)
        -
        \bb^{i,m}
        \big(
        t,\X_t^{i,m},\zeta^{i,m}\big(t,\X_t^{i,m}\big),\bzeta^{-i,m-1}\big(t,\X_t^{i,m}\big)
        \big).
    \end{align*}
    The bound for $\delta \tilde{\bb}^{i,m}_{t}$ follows immediately from the Lipschitz continuity of $\bb^{i,m}$. More precisely, we get
    \begin{align*}
        \big\|
            \delta \tilde{\bb}^{i,m}_{t}
        \big\|_{\R^n}^2
        \leq
        C
        \bigg(
            \big\|
                \delta \X^{i,m}_t
            \big\|_{\R^n}^2
            +
            \big\|
                \zeta^i(t,\X_t)
                -
                \zeta^{i,m}(t,\X_t^{i,m})
            \big\|_{\R^n}^2
            +
            \sum_{j\in \I\backslash\{i\}}
            \big\|
                \zeta^j(t,\X_t)
                -
                \zeta^{j,m-1}(t,\X_t^{i,m})
            \big\|_{\R^n}^2
        \bigg).
    \end{align*}
    After killing at $\tau^{i,m}_{\min}$, integrating, taking expectation and applying Lemma~\ref{lemma:delta_zeta_decomposition}, this gives
    \begin{align}\label{eq:B_b_bound}
        \Big\|
            \delta\tilde{\bb}^{i,m,[\tau^{i,m}_{\min}]}
        \Big\|^2_{\mathbb{H}^{2,n}_T}
        \leq
        C
        \bigg(
            \Big\|
                \delta \X^{i,m,[\tau^{i,m}_{\min}]}
            \Big\|^2_{\mathbb{H}^{2,n}_T}
            +
            \Big\|
                \delta \bar{\zeta}^{i,m,[\tau^{i,m}_{\min}]}
            \Big\|^2_{\mathbb{H}^{2,n}_T}
            +
            \sum_{j\in \I\backslash\{i\}}
            \Big\|
                \delta \bar{\zeta}^{j,m-1,[\tau^{i,m}_{\min}]}
            \Big\|^2_{\mathbb{H}^{2,n}_T}
        \bigg).
    \end{align}
    For $\delta\bar{\bb}^{i,m}_{t}$, it follows from definitions that
    \begin{align*}
        \delta\bar{\bb}^{i,m}_{t}
        = 
        \bar\bb(t,\X_t,\bal^{*}(t,\X_t))
        -
        \bar\bb(t,\X_t,\bal^{i,m}(t,\X_t)),
    \end{align*}
    where
    \begin{align*}
        \bal^{i,m}(t,\x)
        =
        \big[
            \kappa^i
            \big(
                t,
                \x,
                \bPhi(t,\x)\zeta^i(t,\x),
                \boldsymbol\kappa^{-i,m-1}
                (
                    t,
                    \x,
                    \boldsymbol{\zeta}^{-i}(t,\x)
                )
            \big)
            ,
            \boldsymbol\kappa^{-i,m-1}
            (
                t,
                \x,
                \boldsymbol{\zeta}^{-i}(t,\x)
            )
        \big]_i.
    \end{align*}
    By the Lipschitz continuity of $\bar\bb$, we have
    \begin{align*}
        \big\|
             \delta\bar{\bb}^{i,m}_{t}
        \big\|_{\R^n}^2
        &\leq
        L_{\bar\bb} \sum_{j \in \I}
        \big\|
            \alpha^{*,j}(t,\X_t)
            -
            \alpha^{i,m,j}(t,\X_t)
        \big\|_{\R^d}^2. 
    \end{align*}
    Note that for $j \neq i$ the summation term is $\Delta^{j,m-1}$. For the term $j = i$, we have
    \begin{align*}
        \big\|
            \alpha^{*,j}(t,\X_t)
            &-
            \alpha^{i,m,j}(t,\X_t)
        \big\|_{\R^d}^2
        \\
        &\qquad 
        =
        \big\|
            \alpha^{*,j}(t,\X_t)
            -
            \kappa^i
            \big(
                t,
                \X_t,
                \bPhi(t,\X_t)\zeta^i(t,\X_t),
                \boldsymbol\kappa^{-i,m-1}
                (
                    t,
                    \X_t,
                    \boldsymbol{\zeta}^{-i}(t,\X_t)
                )
            \big)   
        \big\|_{\R^d}^2
        \\
        &\qquad \leq
        L_{\kappa}
        \sum_{j \in \I \backslash\{i\}}
        \big\|
            \alpha^{*,j}(t,\X_t)
            -
            \kappa^{j,m-1}
            (
                t, \X_t, \zeta^{j}(t,\X_t)
            )
        \big\|_{\R^d}^2
        \\
        &\qquad =
        L_{\kappa}
        \sum_{j \in \I \backslash\{i\}}
        \big\|
            \Delta^{j,m-1}(t,\X_t)
        \big\|_{\R^d}^2.
    \end{align*}
    Collecting terms, we get
    \begin{align*}
        \big\|
            \delta\bar{\bb}^{i,m}_{t}
        \big\|_{\R^n}^2
        &\leq
        L_{\bar\bb}(1+L_\kappa) \sum_{j\in\I\backslash\{i\}} \big\| \Delta^{j,m-1} \big\|_{\R^d}^2.
    \end{align*}
    After killing at $\tau^{i,m}_{\min}$, integrating and taking expectation, we have
    \begin{align*}
        \Big\|
            \delta\bar{\bb}^{i,m,[\tau^{i,m}_{\min}]}
        \Big\|^2_{\mathbb{H}^{2,n}_T}
        \leq
        L_{\bar\bb} (1 + L_{\kappa})
        \sum_{j \in \I\backslash\{i\}} 
        \Big\|
            \Delta^{j,m-1}(\cdot,\X_{\cdot})^{[\tau^{i,m}_{\min}]}
        \Big\|^2_{\mathbb{H}^{2,d}_T}
        \leq
         L_{\bar\bb}(1 + L_{\kappa}) E_{\Delta}^{m-1}.
    \end{align*}
    Next, we estimate the $\delta \ell^{i,m}$ part of the lemma. The error can similarly be split as
    \begin{align*}
        \Big\|
            \delta \ell^{i,m,[\tau^{i,m}_{\min}]}
        \Big\|^2_{\mathbb{H}^{2,1}_T}
        \leq 
        2\,
        \Big\|
            \delta \bar{\ell}^{i,m,[\tau^{i,m}_{\min}]}
        \Big\|^2_{\mathbb{H}^{2,1}_T}
        +
        2\,
        \Big\|
           \delta \tilde{\ell}^{i,m,[\tau^{i,m}_{\min}]}
        \Big\|^2_{\mathbb{H}^{2,1}_T}
    \end{align*}
    where
    \begin{align*}
        \delta \bar{\ell}^{i,m}_{t} 
        &:=
        f^i
        \big(
            t,\X_t,\zeta^i(t,\X_t),\bzeta^{-i}(t,\X_t)
        \big)
        -
        \ell^{i,m}
        \big(
        t,\X_t,\zeta^{i}\big(t,\X_t\big),\bzeta^{-i}\big(t,\X_t\big)
        \big), \\
        \delta\tilde{\ell}^{i,m}_{t} 
        &:=
        \ell^{i,m}
        \big(
        t,\X_t,\zeta^{i}\big(t,\X_t\big),\bzeta^{-i}\big(t,\X_t\big)
        \big)
        -
        \ell^{i,m}
        \big(
        t,\X_t^{i,m},\zeta^{i,m}\big(t,\X_t^{i,m}\big),\bzeta^{-i,m-1}\big(t,\X_t^{i,m}\big)
        \big).
    \end{align*}
    For $\delta\tilde{\ell}^{i,m}_{t}$, we use the Lipschitz continuity of $\ell^{i,m}$, and repeat the computations used for $\delta\tilde{\bb}^{i,m}_{t} $ to get a bound on the same form as~\eqref{eq:B_b_bound}. For $\delta \bar{\ell}^{i,m}_{t}$, we have from definitions
    \begin{align*}
        \delta \bar{\ell}^{i,m}_{t}
        =
        \bar f^i
        \big(
            t,\X_t,\bal^{*}(t,\X_t)
        \big)
        -
        \bar f^i
        \big(
            t,\X_t,\bal^{i,m}(t,\X_t)
        \big).
    \end{align*}
    Thus, repeating the computations used for $\delta \bar{\bb}^{i,m}_{t}$, we arrive at the bound
    \begin{align*}
        \Big\|
            \delta \bar{\ell}^{i,m,[\tau^{i,m}_{\min}]}
        \Big\|^2_{\mathbb{H}^{2,1}_T}
        \leq
        L_{\bar{\boldsymbol{f}}}(1 + L_{\kappa})E_\Delta^{m-1},
    \end{align*}
    where $L_{\bar{\boldsymbol{f}}}$ comes from using the Lipschitz continuity of $\bar{\boldsymbol{f}}$ instead of $\bar\bb$. This completes the proof.
\end{proof}

\begin{lemma} \label{lemma:delta_x_bound}
There exists $C>0$ such that for all $i\in\I$ and \(m\geq1\) it holds
    \begin{align}
    \Big\|\delta \X^{i,m,[\tau_{\min}^{i,m}]}\Big\|_{\mathbb S^{2,n}_{T}}^2
\leq
C\bigg(
    \Big\| \delta \bar\zeta^{i,m,[\tau_{\min}^{i,m}]}\Big\|_{\mathbb{H}^{2,n}_{T}}^2
    +
    \sum_{j\in\I\setminus\{i\}}\Big\| \delta \bar\zeta^{j,m-1,[\tau_{\min}^{i,m}]}\Big\|_{\mathbb{H}^{2,n}_{T}}^2 + E_\Delta^{m-1}
\bigg). \label{eq:delta_x_bound_1}
    \end{align}
Moreover, for all \(\varepsilon\in(0,1)\) there exists $C>0$ such that for all $i\in\I$ and \(m\geq1\) it holds
    \begin{align}
\big\|\delta \X^{i,m}\big\|_{\mathbb S^{2,n}_{T}}^2
\leq
C\bigg(\big\|\delta \tau^{i,m}\big\|_{\mathbb{L}^{1,1}}^{1-\varepsilon}
+
    \Big\| \delta\bar \zeta^{i,m,[\tau_{\min}^{i,m}]}\Big\|_{\mathbb{H}^{2,n}_{T}}^2
    +
    \sum_{j\in\I\setminus\{i\}}\Big\| \delta \bar\zeta^{j,m-1,[\tau_{\min}^{i,m}]}\Big\|_{\mathbb{H}^{2,n}_{T}}^2 + E_\Delta^{m-1}
\bigg). \label{eq:delta_x_bound_2}
    \end{align}
    
\end{lemma}


\begin{proof}

Let $i \in \I$ and $m \geq 1$ be fixed for the remainder of the proof. For any $s \in [0,T]$, we have $\mathbb P$-almost surely
\begin{align*}
        \delta\X^{i,m,[\tau_{\min}^{i,m}]}_s
    &=
        \int_0^s \delta \bb^{i,m,[\tau_{\min}^{i,m}]}_r\,\d r
    +
        \int_0^s \delta \bSigma^{i,m,[\tau_{\min}^{i,m}]}_r\,\d \W_r.
\end{align*}
Taking the squared $\bbR^n$-norm of both sides, applying~\eqref{ineq:sumsq} followed by the Cauchy--Schwarz inequality for the first term, we obtain 
\begin{align*}
\Big\|\delta \X^{i,m,[\tau_{\min}^{i,m}]}_s\Big\|_{\R^n}^2
\leq\;&
2s\int_0^s\Big\|\delta\bb^{i,m,[\tau_{\min}^{i,m}]}_r
\Big\|_{\R^n}^2\,\d r
+
2\, \bigg\|\int_0^s\delta \bSigma^{i,m,[\tau_{\min}^{i,m}]}_r\,\d \W_r\bigg\|_{\R^n}^2
.
\end{align*}
Next, take the supremum over $s \in [0,t]$, first on the right-hand side and then on the left-hand side, and take expectations on both sides. Recalling the definitions of $\mathbb S^{2,n}_{t}$- and $\mathbb H^{2,n}_{t}$-norms, we have
\begin{align}\label{eq:2_terms}\begin{split}
\Big\|\delta \X^{i,m,[\tau_{\min}^{i,m}]}\Big\|_{\mathbb S^{2,n}_{t}}^2
\lesssim\;&
\Big\|\delta\bb^{i,m,[\tau_{\min}^{i,m}]}\Big\|_{\mathbb H^{2,n}_{t}}^2
+
\bigg\|\int_0^\cdot \delta \bSigma^{i,m,[\tau_{\min}^{i,m}]}_r\,\d \W_r\bigg\|_{\mathbb{S}_{t}^{2,n}}^2
.
\end{split}
\end{align} 
We bound the two terms on the right-hand side of \eqref{eq:2_terms} separately. For the first term, we use the bound from Lemma~\ref{lem:delta_bi_bound}. For the second term of~\eqref{eq:2_terms}, we apply the Burkholder–Davis–Gundy inequality \eqref{eq:BDG1}, use the definition of $\delta \bSigma^{i,m}$ in~\eqref{eq:definition_deltasigmai} together with the Lipschitz continuity of \(\bSigma\) from~\eqref{eq:lipschitz_sigma}, to get
\begin{align}
    \bigg\|\int_0^\cdot\delta \bSigma^{i,m,[\tau_{\min}^{i,m}]}_r\,\d \W_r\bigg\|_{\mathbb{S}_{t}^{2,n}}^2
    \lesssim
   \Big\|\delta \bSigma^{i,m,[\tau_{\min}^{i,m}]}\Big\|_{\mathbb H^{2,n,n}_{t}}^2
   \lesssim
    \Big\| \delta \X^{i,m,[\tau_{\min}^{i,m}]}\Big\|_{\mathbb{H}^{2,n}_{t}}^2
    \leq
    \int_0^t\Big\|\delta \X^{i,m,[\tau_{\min}^{i,m}]}\Big\|^2_{\mathbb S^{2,n}_{s}}\, \d s, \label{eq:second_bound}
\end{align}
where the last inequality follows from the definitions of the $\mathbb{H}^{2,n}_{t}$- and $\mathbb S^{2,n}_{s}$-norms. In total, we have
\begin{align*}
\Big\|\delta \X^{i,m,[\tau_{\min}^{i,m}]}\Big\|_{\mathbb S^{2,n}_{t}}^2
    &\lesssim
    \Big\|\delta \bar{\zeta}^{i,m,[\tau_{\min}^{i,m}]}\Big\|_{\mathbb{H}^{2,n}_{t}}^2
    +
    \sum_{j\in\I\setminus\{i\}}
\Big\|\delta \bar{\zeta}^{j,m-1,[\tau_{\min}^{i,m}]}\Big\|_{\mathbb{H}^{2,n}_{t}}^2
+
\int_0^t\Big\|\delta \X^{i,m,[\tau_{\min}^{i,m}]}\Big\|^2_{\mathbb S^{2,n}_{s}}\, \d s + E_\Delta^{m-1}.
\end{align*}
Since this inequality holds for all $t \in [0,T]$, Gronwall's lemma yields
\begin{align*}
\Big\|\delta \X^{i,m,[\tau_\text{min}^{i,m}]}\Big\|_{\mathbb S^{2,n}_{T}}^2
\lesssim
    \Big\|\delta \bar{\zeta}^{i,m,[\tau_\text{min}^{i,m}]}\Big\|_{\mathbb{H}^{2,n}_{T}}^2
    +
    \sum_{j\in\I\setminus\{i\}}\Big\|\delta \bar{\zeta}^{j,m-1,[\tau_\text{min}^{i,m}]}\Big\|_{\mathbb{H}^{2,n}_{T}}^2 + E_\Delta^{m-1}.
\end{align*}
This proves the first claim of the lemma. 

For the second claim, we first use~\eqref{ineq:sumsq} as
\begin{align}\label{eq:delta_S_S2}
\begin{split}
\big\|\delta \X^{i,m}\big\|_{\mathbb S^{2,n}_{T}}^2
&\le
2\Big\|\delta\X^{i,m,[\tau_{\min}^{i,m}]}\Big\|_{\mathbb S^{2,n}_{T}}^2
+2\Big\|\delta\tilde\X^{i,m,[\tau_{\min}^{i,m},\tau_{\max}^{i,m}]}\Big\|_{\mathbb S^{2,n}_{T}}^2,
\end{split}
\end{align}
where the tilde-notation denotes the increment of the state-error dynamics over $[\tau^{i,m}_{\min}, \tau^{i,m}_{\max}]$ initialized at zero at $\tau^{i,m}_{\min}$, that is
\begin{align*}
    \delta\tilde\X^{i,m,[\tau_{\min}^{i,m},\tau_{\max}^{i,m}]}_s
    \defeq
    \int_{0}^{s}
        \delta \bb^{i,m,[\tau_{\min}^{i,m},\tau_{\max}^{i,m}]}_r
    \, \d r
    +
    \int_{0}^{s}
        \delta \bSigma^{i,m,[\tau_{\min}^{i,m},\tau_{\max}^{i,m}]}_r
    \, \d \W_r.
\end{align*}
Note that the first term in~\eqref{eq:delta_S_S2} is already bounded from the first part of the lemma. For the second term, similarly to \eqref{eq:2_terms}, we have 
\begin{align}
\Big\|\delta \tilde\X^{i,m,[\tau_{\min}^{i,m},\tau_{\max}^{i,m}]}\Big\|_{\mathbb S^{2,n}_{T}}^2
\lesssim\;&
\Big\|\delta\bb^{i,m,[\tau_{\min}^{i,m},\tau_{\max}^{i,m}]}\Big\|_{\mathbb H^{2,n}_{T}}^2
+
\bigg\|\int_0^\cdot\delta \bSigma^{i,m,[\tau_{\min}^{i,m},\tau_{\max}^{i,m}]}_r\,\d \W_r\bigg\|_{\mathbb{S}_{T}^{2,n}}^2
. \label{eq:2_terms_again}
\end{align} 
Again, we bound the two terms on the right-hand side separately. For the first term, taking $\varepsilon\in(0,1)$, applying Hölder’s inequality with $p=(1-\varepsilon)^{-1}$ and $q=\varepsilon^{-1}$ leading to $p^{-1}+q^{-1}=1$ yields
\begin{align}
\begin{split}
    \Big\|\delta \bb^{i,m,[\tau_{\min}^{i,m},\tau_{\max}^{i,m}]}\Big\|_{\mathbb H^{2,n}_{T}}^2
&= \mathbb E\!\int_0^T \mathbf 1_{\{\tau_{\min}^{i,m}< s\le \tau_{\max}^{i,m}\}}
\Big\|\delta \bb^{i,m,[\tau_{\min}^{i,m},\tau_{\max}^{i,m}]}_s\Big\|_{\R^n}^2\,\d s \\
&\leq  \bigg(\mathbb E\!\int_0^T \mathbf 1_{\{\tau_{\min}^{i,m}< s\le \tau_{\max}^{i,m}\}}\,\d s\bigg)^{\frac1p}
\bigg(\E\int_0^T\Big\|\delta \bb^{i,m,[\tau_{\min}^{i,m},\tau_{\max}^{i,m}]}_s\Big\|_{\R^n}^{2q}\,\d s\bigg)^{\frac1q}\\
&\leq \|\delta\tau^{i,m}\|_{\mathbb L^{1,1}}^{1-\varepsilon}
\bigg(\E\int_0^T\Big\|\delta \bb_s^{i,m,[\tau_{\min}^{i,m},\tau_{\max}^{i,m}]}\Big\|_{\R^n}^{\frac2\varepsilon}\, \d s\bigg)^{\varepsilon}. \label{eq:second_counterpart}
\end{split}
\end{align}
For the second term, the Burkholder--Davis--Gundy inequality~\eqref{eq:BDG1} and the H\"older inequality gives
\begin{align}
\begin{split}
    \bigg\|\int_0^{\cdot} \delta \bSigma^{i,m,[\tau_{\min}^{i,m},\tau_{\max}^{i,m}]}_r\,\d \W_r\bigg\|_{\mathbb{S}_{T}^{2,n}}^2
&\lesssim\; 
\Big\|\delta \bSigma^{i,m,[\tau_{\min}^{i,m},\tau_{\max}^{i,m}]}\Big\|_{\mathbb H^{2,n,n}_{T}}^2\\
&\leq\big\|\delta \tau^{i,m}\big\|_{\mathbb{L}^{1,1}}^{1-\varepsilon}\bigg(\E\int_0^T\Big\|\delta \bSigma_t^{i,m,[\tau_{\min}^{i,m},\tau_{\max}^{i,m}]}\Big\|_{\R^n}^{\frac2\varepsilon}\, \d t\bigg)^{\varepsilon}.
\end{split} \label{eq:tmptmptmp}
\end{align}
Insert~\eqref{eq:second_counterpart} and~\eqref{eq:tmptmptmp} into~\eqref{eq:2_terms_again}, and note that by the Lipschitz continuity of $\bb, \bb^{i,m}$, $\bSigma$, and the Markov maps, the compactness of $A$, and the finite moments of $\X$ and $\X^{i,m}$, it follows that
\begin{align*}
\sup_{m\geq 1}\,
\E\bigg[\int_0^T\Big\|\delta \bb_t^{i,m,[\tau_{\min}^{i,m},\tau_{\max}^{i,m}]}\Big\|_{\R^n}^{\frac2\varepsilon}\, \d t\bigg]<\infty
\quad
\textrm{and}
\quad
\sup_{m\geq 1}\,
\E\bigg[\int_0^T\Big\|\delta \bSigma_t^{i,m,[\tau_{\min}^{i,m},\tau_{\max}^{i,m}]}\Big\|_{\R^n}^{\frac2\varepsilon}\, \d t\bigg]<\infty.
\end{align*}
This concludes the proof.
\end{proof}

\begin{lemma}\label{lemma:delta_y_bound}
For all \(\varepsilon\in(0,1)\) there exists $C>0$ such that for all $i\in\I$ and \(m\geq1\) we have
\begin{align*}      
    \sum_{i\in \I}\big\|\delta Y^{i,m}\big\|_{\mathbb S^{2,1}_{T}}^2
    &\lesssim
    E_\zeta^m
    +
    E_\zeta^{m-1}
    +
    E_\Delta^{m-1}
    +
    \sum_{i\in \I}
   \big\|\delta \tau^{i,m}\big\|_{\mathbb{L}^{1,1}} 
   + 
   \big\|\delta \tau^{i,m}\big\|_{\mathbb{L}^{1,1}}^{1-\varepsilon},
\end{align*}
and
\begin{align*}
    \sum_{i\in \I}
    \Big\|
    \zeta^{i,[\tau]}(\cdot,\X_\cdot)
    -
    \zeta^{i,m,[\tau^{i,m}]}\big(\cdot,\X_\cdot^{i,m}\big)
    \Big\|_{\mathbb H^{2,n}_{T}}^2
    \lesssim
    E_\zeta^m
    +
    E_\zeta^{m-1}
    +
    E_\Delta^{m-1}
    +
    \sum_{i\in \I}
    \big\|\delta \tau^{i,m}\big\|_{\mathbb{L}^{1,1}}^{1-\varepsilon} .
\end{align*}
\end{lemma}

\begin{proof}
Let $i \in \I$ and $m \geq 1$.  For all $t \in [0,T]$, we have $\mathbb P$-almost surely
\begin{align*}
\delta Y_t^{i,m}
=\;&
\delta g^{i,m}
+
\int_t^T \delta \ell^{i,m}_s
\, \d s
-
\int_t^T \big(
\delta \zeta^{i,m}_s\big)^\top\, \d \W_s
.
\end{align*}
Note that each backward process is frozen after its respective stopping time. When the equations are extended to the fixed horizon $[0,T]$ via killed coefficients, the processes therefore remain constant beyond their stopping times. In particular,
\begin{align}
    \delta Y_T^{i,m} 
= Y_{\tau}^{i} - Y_{\tau^{i,m}}^{i,m}
= g^i(\tau,\X_\tau) - g^i\big(\tau^{i,m},\X_{\tau^{i,m}}^{i,m}\big)
= \delta g^{i,m}. \label{eq:Y_T_gi}
\end{align}
Applying Itô’s lemma to the map \(y\mapsto|y|^2\) along $\delta Y_t^{i,m}$, together with~\eqref{eq:Y_T_gi}, we have
\begin{align*}
            \big|\delta Y_t^{i,m}\big|^2
            =&
            \big|\delta g^{i,m}\big|^2
            +
            2\int_t^T
            \delta Y_s^{i,m}\delta \ell^{i,m}_s\, \d s
-
    \int_t^T\big\|
\delta \zeta^{i,m}_s\big\|_{\R^n}^2
\, \d s
-
            2\int_t^T \delta Y_s^{i,m}\big(
\delta \zeta^{i,m}_s
\big)^\top\, \d \W_s.
        \end{align*}
Dropping the third term and taking absolute values, we obtain
\begin{align*}
        \big|\delta Y_t^{i,m}\big|^2
        &\leq
        \big|\delta g^{i,m}\big|^2
        +
        2\, \bigg|\int_t^T
        \delta Y_s^{i,m}
\delta \ell^{i,m}_s
\, \d s\bigg|
        +2\, \bigg|
        \int_{t}^T \delta Y_s^{i,m}\big(
\delta \zeta^{i,m}_s
\big)^\top\, \d \W_s\,\bigg|.
    \end{align*}
Taking the supremum over $t\in[0,T]$, first on the right-hand side and then on the left-hand side, and then expectations yields
\begin{align}
        \big\|\delta Y^{i,m}\big\|^2_{\mathbb S^{2,1}_{T}}
        &\leq
        \big\|\delta g^{i,m}\big\|^2_{\mathbb L^{2,1}}
        +
        2\, \bigg\|\int_{\cdot}^T
        \delta Y_s^{i,m}
        \delta \ell^{i,m}_s
        \, \d s\bigg\|_{\mathbb S^{1,1}_T}
        +
        2\, \bigg\|
         \int_\cdot^T\delta Y_s^{i,m}\big(\delta \zeta_s^{i,m}
         \big)^\top \, \d\W_s\bigg\|_{\mathbb S^{1,1}_{T}}. \label{eq:deltay_first_bound}
\end{align}
We bound the two integral terms separately. For the first one, we use~\eqref{eq:integral_in_S11_to_Hnorms} together with~\eqref{eq:HS2}, followed by Young's inequality with parameter $\gamma > 0$, such that
\begin{align}
       2\, \bigg\|\int_{\cdot}^T
        \delta Y_s^{i,m}
        \delta \ell^{i,m}_s
        \, \d s\bigg\|_{\mathbb S^{1,1}_T}
    &\leq 
        2T \big\| \delta Y^{i,m} \big\|_{\mathbb S^{2,1}_{T}}
        \big\| \delta \ell^{i,m} \big\|_{\mathbb H^{2,1}_{T}}
    \leq
        \gamma^{-1} T^2\big\| \delta Y^{i,m} \big\|^2_{\mathbb S^{2,1}_{T}}
    +
        \gamma  \big\| \delta \ell^{i,m} \big\|_{\mathbb H^{2,1}_{T}}^2. \label{eq:young_first}
\end{align}
For the second one, we use the Burkholder--Davis--Gundy inequality~\eqref{eq:BDG1} with $p=1$, then~\eqref{eq:HS3} and finally Young's inequality, to get
\begin{align}
\begin{split}
        2\, \bigg\|
         \int_\cdot^T\delta Y^{i,m}\big(\delta \zeta_s^{i,m}
         \big)^\top \, \d\W_s\bigg\|_{\mathbb S^{1,1}_{T}}
    &\leq 
        2K_1\big\|
             \delta Y^{i,m}\delta \zeta^{i,m}
              \big\|_{\mathbb H^{1,n}_{T}} \\
    &\leq 
        2K_1\big\|
             \delta Y^{i,m}
             \big\|_{\mathbb S^{2,1}_T}
             \big\|
             \delta \zeta^{i,m}
              \big\|_{\mathbb H^{2,n}_{T}} \\
    &\leq 
        K_1 \lambda^{-1}
        \big\|\delta Y^{i,m}\big\|_{\mathbb S^{2,1}_{T}}^2
        +
        K_1\lambda
        \big\|\delta \zeta^{i,m}\big\|_{\mathbb H^{2,n}_{T}}^2. 
\end{split} \label{eq:young_second}
\end{align}
Substituting~\eqref{eq:young_first} and~\eqref{eq:young_second} into~\eqref{eq:deltay_first_bound}, choosing $\gamma = 4T^2$ and $\lambda = 4K_1$, and moving the terms involving $\delta Y^{i,m}$ to the left-hand side, we obtain
\begin{align}
    \big\|\delta Y^{i,m}\big\|^2_{\mathbb S^{2,1}_{T}}
        &\lesssim
        \big\|\delta g^{i,m}\big\|^2_{\mathbb L^{2,1}}
        +
        \big\| \delta \ell^{i,m} \big\|_{\mathbb H^{2,1}_{T}}^2
        +
        \big\|\delta \zeta^{i,m}\big\|_{\mathbb H^{2,n}_{T}}^2. \label{eq:Ybound_second}
\end{align}
We now estimate each of the three terms on the right-hand side of~\eqref{eq:Ybound_second} separately. For the first term, the Lipschitz continuity of $g^i$ and~\eqref{ineq:sumsq} imply
\begin{align*}
    \big\|\delta g^{i,m}\big\|_{\mathbb L^{2,1}}^2
    &\lesssim
    \big\|\delta \tau^{i,m}\big\|_{\mathbb{L}^{1,1}}
    +
    \Big\|\X_\tau-\X_{\tau^{i,m}}^{i,m}\Big\|_{\mathbb{L}^{2,n}}^2\\
    &=
    \big\|\delta \tau^{i,m}\big\|_{\mathbb{L}^{1,1}}
    +
    \Big\|\X_\tau-\X_{\tau^{i,m}} +\delta\X_{\tau^{i,m}}^{i,m}\Big\|_{\mathbb{L}^{2,n}}^2\\
    &\leq
    \big\|\delta \tau^{i,m}\big\|_{\mathbb{L}^{1,1}}
    +
    2\big\|\X_\tau-\X_{\tau^{i,m}}\big\|_{\mathbb{L}^{2,n}}^2
  + 2\Big\|\delta\X_{\tau^{i,m}}^{i,m}\Big\|_{\mathbb{L}^{2,n}}^2\\
  &\lesssim
  \big\|\delta \tau^{i,m}\big\|_{\mathbb{L}^{1,1}}
+\big\|\delta\X^{i,m}\big\|_{\mathbb S^{2,n}_{T}}^2 \\
&\lesssim
\big\|\delta \tau^{i,m}\big\|_{\mathbb{L}^{1,1}}
+
\big\|\delta \tau^{i,m}\big\|_{\mathbb{L}^{1,1}}^{1-\varepsilon}
+
    \Big\| \delta\bar \zeta^{i,m,[\tau_{\min}^{i,m}]}\Big\|_{\mathbb{H}^{2,n}_{T}}^2
    +
    \sum_{j\in\I\setminus\{i\}}\Big\| \delta \bar\zeta^{j,m-1,[\tau_{\min}^{i,m}]}\Big\|_{\mathbb{H}^{2,n}_{T}}^2 + E_\Delta^{m-1},
\end{align*}
where the fourth step uses the estimates
\begin{align*}
    \big\|\X_\tau-\X_{\tau^{i,m}}\big\|_{\mathbb{L}^{2,n}}^2 &\lesssim \big\|\delta \tau^{i,m}\big\|_{\mathbb{L}^{1,1}}, \\
    \Big\|\delta\X_{\tau^{i,m}}^{i,m}\Big\|_{\mathbb L^{2,n}}^2 &\leq \big\|\delta\X^{i,m}\big\|_{\mathbb S^{2,n}_{T}}^2,
\end{align*}
while the final inequality follows from Lemma~\ref{lemma:delta_x_bound}. For the second term in~\eqref{eq:Ybound_second}, we use~\eqref{eq:H2_DS} to split the process over two disjoint intervals. The first part is then estimated using Lemma~\ref{lem:delta_bi_bound}, and the second part is treated as in~\eqref{eq:second_counterpart}. This gives
\begin{align*}
    \big\|\delta\ell^{i,m}\big\|_{\mathbb H^{2,1}_{T}}^2
    &=
    \Big\|\delta\ell^{i,m,[\tau_{\min}^{i,m}]}\Big\|_{\mathbb H^{2,1}_{T}}^2
    +
    \Big\|\delta\ell^{i,m,[\tau_{\min}^{i,m},\tau_{\max}^{i,m}]}\Big\|_{\mathbb H^{2,1}_{T}}^2 \\
    &\lesssim
    \Big\|\delta \X^{i,m,[\tau_{\min}^{i,m}]}\Big\|_{\mathbb S^{2,n}_{T}}^2
+
\Big\|\delta\bar\zeta^{i,m,[\tau_{\min}^{i,m}]}\Big\|^2_{\mathbb H^{2,n}_{T}} 
+
\sum_{j\in\I\setminus\{i\}}\Big\|\delta\bar\zeta^{j,m-1,[\tau_{\min}^{i,m}]}\Big\|^2_{\mathbb H^{2,n}_{T}} + E_\Delta^{m-1} + \big\|\delta \tau^{i,m}\big\|_{\mathbb{L}^{1,1}}^{1-\varepsilon} \\
&\lesssim
\Big\|\delta\bar\zeta^{i,m,[\tau_{\min}^{i,m}]}\Big\|^2_{\mathbb H^{2,n}_{T}} 
+
\sum_{j\in\I\setminus\{i\}}\Big\|\delta\bar\zeta^{j,m-1,[\tau_{\min}^{i,m}]}\Big\|^2_{\mathbb H^{2,n}_{T}} + E_\Delta^{m-1} + \big\|\delta \tau^{i,m}\big\|_{\mathbb{L}^{1,1}}^{1-\varepsilon},
\end{align*}
where the final inequality follows from Lemma~\ref{lemma:delta_x_bound}. The third term in~\eqref{eq:Ybound_second} is handled analogously, as
\begin{align*}
    \big\|\delta\zeta^{i,m}\big\|_{\mathbb H^{2,n}_{T}}^2
    &=
    \Big\|\delta\zeta^{i,m,[\tau_{\min}^{i,m}]}\Big\|_{\mathbb H^{2,n}_{T}}^2
    +
    \Big\|\delta\zeta^{i,m,[\tau_{\min}^{i,m},\tau_{\max}^{i,m}]}\Big\|_{\mathbb H^{2,n}_{T}}^2 \\
    &\lesssim 
    \Big\|\delta \X^{i,m,[\tau_\text{min}^{i,m}]}\Big\|_{\mathbb{S}^{2,n}_{T}}^2
    +
    \Big\|\delta \bar{\zeta}^{i,m,[\tau_\text{min}^{i,m}]}\Big\|_{\mathbb{H}^{2,n}_{T}}^2
    +
    \big\|\delta \tau^{i,m}\big\|_{\mathbb{L}^{1,1}}^{1-\varepsilon}\\
&\lesssim
\Big\|\delta\bar\zeta^{i,m,[\tau_{\min}^{i,m}]}\Big\|^2_{\mathbb H^{2,n}_{T}} 
+
\sum_{j\in\I\setminus\{i\}}\Big\|\delta\bar\zeta^{j,m-1,[\tau_{\min}^{i,m}]}\Big\|^2_{\mathbb H^{2,n}_{T}} + \big\|\delta \tau^{i,m}\big\|_{\mathbb{L}^{1,1}}^{1-\varepsilon}.
\end{align*}
Combining the bounds for the three terms in~\eqref{eq:Ybound_second} and summing over $i\in\I$, we obtain
\begin{align*}
    \sum_{i\in\I}
    \big\|
    \delta Y^{i,m}
    \big\|^2_{\mathbb S^{2,1}_{T}}
    \lesssim
    E_\zeta^{m}
    +
    E_\zeta^{m-1}
    +
    E_\Delta^{m-1}
    +
    \sum_{i\in\I}
    \big\|\delta \tau^{i,m}\big\|_{\mathbb{L}^{1,1}}
    + 
    \big\|\delta \tau^{i,m}\big\|_{\mathbb{L}^{1,1}}^{1-\varepsilon}.
\end{align*}
This concludes the first claim of the lemma. For the second claim, we have the decomposition
\begin{equation*}
\Big\|\zeta^{i,[\tau]}(\cdot,\X_\cdot)-\zeta^{i,m,[\tau^{i,m}]}\big(\cdot,\X_\cdot^{i,m}\big)\Big\|_{\mathbb H^{2,n}_{T}}^2
=
\Big\|\delta\zeta^{i,m,[\tau_{\min}^{i,m}]}\Big\|_{\mathbb H^{2,n}_{T}}^2
+
\Big\|\delta\zeta^{i,m,[\tau_{\min}^{i,m},\tau_{\max}^{i,m}]}\Big\|_{\mathbb H^{2,n}_{T}}^2.
\end{equation*}
Applying Lemma~\ref{lemma:delta_zeta_decomposition}, \eqref{eq:HS2}, Lemma~\ref{lemma:delta_x_bound} and summing over $i\in\I$ yields the desired estimate.
\end{proof}

\begin{proposition}\label{prop:deltaXYZ}
Recall the error quantity $\Psi^{m}$ from~\eqref{eq:Psim}. For all \(\varepsilon\in(0,1)\), \(m\geq1\), it holds that 
\begin{align*}
    \Psi^{m} 
    \lesssim 
    E_\zeta^{m}
    +
    E_\zeta^{m-1}
    +
    E_\Delta^{m-1}
    +
    \sum_{i\in\I}
    \big\|\delta\tau^{i,m}\big\|_{\mathbb{L}^{1,1}}
    +
    \big\|\delta\tau^{i,m}\big\|_{\mathbb{L}^{1,1}}^{1-\varepsilon}.
\end{align*}    
\end{proposition}

\begin{proof}
    The result follows immediately from Lemma~\ref{lemma:delta_x_bound} and Lemma~\ref{lemma:delta_y_bound}.
\end{proof}

\subsubsection{Stopping-time error}

We now estimate the stopping-time error $\delta \tau^{i,m}$. The following proposition shows that this error is controlled by the Markov-map error, and disappears entirely in the case $\D = \R^n$.

\begin{proposition}\label{prop:stoppin_time_L1}
For $|\D|<\infty$, there exists $C<\infty$ such that for all $i\in\I$ and $m\geq 1$, it holds
\begin{align*}
    \big\|\delta\tau^{i,m}\big\|_{\mathbb{L}^{1,1}}
    \leq
    C
    \bigg(
        \Big\|
        \delta\bar \zeta^{i,m,[\tau_{\min}^{i,m}]}
        \Big\|_{\mathbb H^{2,n}_{T}}^2
        +
        \sum_{j\in\I\setminus\{i\}}
        \Big\|
        \delta \bar{\zeta}^{j,m-1,[\tau_{\min}^{i,m}]}
        \Big\|_{\mathbb{H}^{2,n}_{T}}^2 
        +
        E_\Delta^{m-1}
    \bigg)^{1/2}.
\end{align*}
For $\D=\R^n$, it holds $\P$-almost surely that $\delta\tau^{i,m}=0$.
\end{proposition}

\begin{proof}
    In the case $|\D|<\infty$, by \cite[Theorem 2.3]{dokuchaev2004estimates}, the Cauchy--Schwarz inequality and taking the supremum over $[0,T]$ we have for $i\in\I$ and $m\geq1$ that
\begin{align*}
  \big\|\delta \tau^{i,m}\big\|_{\mathbb{L}^{1,1}}
    \lesssim
    \Big\|\delta\X_{\tau_{\min}^{i,m}}^{i,m}\Big\|_{\mathbb{L}^{1,n}}
    \leq
    \Big\|\delta\X_{\tau_{\min}^{i,m}}^{i,m}\Big\|_{\mathbb{L}^{2,n}}
    \leq
    \Big\|\delta\X^{i,m,[\tau_{\min}^{i,m}]}\Big\|_{\mathbb S^{2,n}_{T}}.
\end{align*}
The claim now follows by applying Lemma~\ref{lemma:delta_x_bound}. For $\D = \R^n$, we have $\partial\D=\emptyset$ and thus $\tau=\tau^{i,m}$ $\mathbb P$-almost surely. In particular, $\|\delta \tau^{i,m}\|_{\mathbb{L}^{1,1}}=0$.
\end{proof}

\subsubsection{Markov-map error}

We next turn to the Markov-map error. The analysis also involves the policy-mismatch error $E_\Delta^m$, such that $E_\zeta^m$ and $E_\Delta^m$ must be estimated jointly. Lemmas~\ref{lem:functional_error_bounded} and~\ref{lemma:functional_error_unbounded} provide the Markov-map estimates under Assumptions~\ref{assump:gradient} and~\ref{assump:smalltime}, respectively. Lemma~\ref{lem:policy_error} gives the corresponding bound for the policy error. Lemma~\ref{lem:spectral_radius} then shows that the resulting coupled estimate is contractive, and Proposition~\ref{prop:functional_error} summarizes the resulting convergence estimate.

\begin{lemma}\label{lem:functional_error_bounded}
   Consider the setting in Assumption~\ref{assump:gradient}. There exists $\beta_0 \in (0,\infty)$ such that for every $\beta > \beta_0$, there exist constants $C^{\Asup}_\Delta(\beta) < \infty$ and $\tilde K_{\Asup}(\beta) \in (0,1)$ for which, for all $m \geq 1$, 
    \begin{align*}
        E^m_{\zeta, \beta}
        \leq
         \tilde{K}_{\Asup}(\beta) E_{\zeta,\beta}^{m-1}
         +
         C^{\Asup}_\Delta(\beta) E_{\Delta,\beta}^{m-1}.
    \end{align*}
\end{lemma}

\begin{proof}

    Let $i \in \I$ and $m \geq 1$ be fixed. To facilitate the analysis, we introduce an auxiliary FBSDE whose stochastic integral has integrand $\zeta^{i,m}_t = \bSigma^\top \mathrm{D}_{\x} V^{i,m}(t,\X_t)$. 
For notational convenience, we define
\(
V^{i,0} \equiv 0,\) and
\(\zeta^{i,0} \equiv 0.
\)
We stress that $V^{i,0}$ is not obtained from a fictitious-play PDE. Since the initialization of fictitious play is given by the zero policy, the coefficients $b^{i,1}$ and $\ell^{i,1}$ are independent of $V^{i,0}$. The above definition is introduced solely to simplify the notation in the convergence analysis.

We introduce the notation
\begin{align*}
    \delta\hat\bb^{i,m}_t
    \defeq
    \bb\big(t,\X_t,\zeta^{i}_t,\bzeta^{-i}_t\big)
    -
    \bb^{i,m}\big(t,\X_t,\zeta^{i,m}_t,\bzeta^{-i,m-1}_t\big),
    \quad
    t\in[0,T],
\end{align*}
and let $\Gamma_t^{i,m}\defeq\mathrm{D}_x V^{i,m}(t,\X_t)$. Itô's lemma applied to $\hat Y^{i,m}_t \defeq V^{i,m}(t,\X_t)$ yields $\mathbb P$-almost surely
\begin{align*}  
    \hat Y_t^{i,m} 
    &=
    g^i(\tau,\X_\tau) 
    +
    \int_{t\wedge \tau}^\tau \ell^{i,m}\big(s,\X_s,\zeta^{i,m}_s,\bzeta^{-i,m-1}_s\big) \ds
    -
    \int_{t\wedge \tau}^\tau\big\langle \delta\hat\bb_s^{i,m},\Gamma_s^{i,m}\big\rangle \ds
    - 
    \int_{t \wedge \tau}^\tau(\zeta^{i,m}_s)^\top \d\W_s. 
\end{align*} 
For the corresponding difference process $\delta\hat Y^{i,m}_t:=Y^i_t-\hat Y^{i,m}_t$, we introduce
\begin{align*}
    \delta\hat\ell^{i,m}_t
    \defeq
    f^{i}\big(t,\X_t,\zeta^{i}_t,\bzeta^{-i}_t\big)
    -
    \ell^{i,m}\big(t,\X_t,\zeta^{i,m}_t,\bzeta^{-i,m-1}_t\big),
    \quad
    t\in[0,T].
\end{align*}
Then, for all $t \in [0,T]$, we have $\mathbb P$-almost surely
\begin{align*}
    \delta \hat Y_t^{i,m} 
    &=
    \int_{t\wedge \tau}^\tau \delta\hat\ell^{i,m}_s \ds
    +
    \int_{t\wedge \tau}^\tau\big\langle \delta\hat\bb_s^{i,m},\Gamma_s^{i,m}\big\rangle \ds 
    - 
    \int_{t \wedge \tau}^\tau(\delta\bar\zeta^{i,m}_s)^\top\d\W_s.
\end{align*} 
For $\beta>0$ we apply Itô's lemma to $(t,y) \mapsto e^{\beta t}|y|^2$ along $\delta\bar Y^{i,m}_t$ to obtain $\mathbb P$-almost surely
\begin{align*}
    \int_{0}^\tau\e^{\beta s}\big\|\delta\bar\zeta_s^{i,m}\big\|^2_{\R^n} \ds
    &=
    2\int_{0}^\tau \e^{\beta s}
            \delta \hat Y_s^{i,m}\delta \hat\ell^{i,m}_s \ds
            +
             2\int_{0}^\tau \e^{\beta s}
            \delta \hat Y_s^{i,m}\big\langle \delta\hat\bb_s^{i,m},\Gamma_s^{i,m}\big\rangle \ds \\
            &\quad
            -
            \beta\int_{0}^\tau\e^{\beta s}\big|\delta \hat Y_s^{i,m}\big|^2 \ds
            -
            2\int_{0}^\tau \e^{\beta s}\delta \hat Y_s^{i,m}\big(\delta\bar \zeta^{i,m}_s
            \big)^\top\d \W_s
            -
            |\delta\hat Y_{0}^{i,m}|^2 .
\end{align*}
Taking expectations, inserting absolute values, and dropping the non-positive $\delta\hat Y_{0}^{i,m}$-term, we get
\begin{align*}
     \Big\|\delta\bar\zeta^{i,m,[\tau]}\Big\|^2_{\mathbb H^{2,n}_{\beta,T}}\!
     &\leq 
    2\,\Big|
        \big\langle
            \delta \hat Y^{i,m},\delta \hat\ell^{i,m,[\tau]}
        \big\rangle_{\mathbb H^{2,1}_{\beta,T}}
    \Big|
        +
    2\,\Big|
        \big\langle
            \delta \hat Y^{i,m},
            \big\langle \delta\hat\bb^{i,m,[\tau]},\Gamma^{i,m}\big\rangle
        \big\rangle_{\mathbb H^{2,1}_{\beta,T}}
    \Big|
            -
    \beta\big\|\delta \hat Y^{i,m}\big\|^2_{\mathbb H^{2,1}_{\beta,T}}.
\end{align*}
For the first term, by Cauchy--Schwarz and Young's inequality with parameter $\beta/2$, we have
\begin{align*}
    2
    \Big|
        \big\langle
            \delta \hat Y^{i,m},\delta \hat\ell^{i,m,[\tau]}
        \big\rangle_{\mathbb H^{2,1}_{\beta,T}}
    \Big|
    \leq 
    2
    \big\| \delta \hat Y^{i,m} \big\|_{\mathbb H^{2,1}_{\beta,T}}
    \big\|\delta \hat\ell^{i,m,[\tau]} \big\|_{\mathbb H^{2,1}_{\beta,T}}
    \leq 
    \frac{\beta}{2}
    \big\| \delta \hat Y^{i,m} \big\|_{\mathbb H^{2,1}_{\beta,T}}^2
    +
    \frac{2}{\beta}
    \big\|\delta \hat\ell^{i,m,[\tau]} \big\|_{\mathbb H^{2,1}_{\beta,T}}^2.
\end{align*}
The second term is treated analogously. Combining the estimates and canceling the $\delta \hat Y^{i,m}$-terms yields
\begin{align}\label{eq:thre_terms_functional}
    \big\|\delta\bar\zeta^{i,m,[\tau]}\big\|^2_{\mathbb H^{2,n}_{\beta,T}}
    \leq
    \frac2\beta\Big(\big\|\delta \hat\ell^{i,m,[\tau]}\big\|^2_{\mathbb H^{2,1}_{\beta,T}}
            +
             \big\|\big\langle \delta\hat\bb^{i,m,[\tau]},\Gamma^{i,m}\big\rangle\big\|^2_{\mathbb H^{2,1}_{\beta,T}}\Big).
\end{align}
For the first term, note that we can split the coefficient difference as
\begin{align*}
    \big\|\delta \hat\ell^{i,m,[\tau]}\big\|^2_{\mathbb H^{2,1}_{\beta,T}} 
    \leq 
    2\big\|
        \delta \bar{\ell}^{i,m}
    \big\|^2_{\mathbb H^{2,1}_{\beta,T}}
    +
    2\big\|
        \ell^{i,m}\big(\cdot,\X_{\cdot},\zeta^{i},\bzeta^{-i}\big)
        -
        \ell^{i,m}\big(\cdot,\X_{\cdot},\zeta^{i,m},\bzeta^{-i,m-1}\big)
    \big\|^2_{\mathbb H^{2,1}_{\beta,T}}.
\end{align*}
The first part has already been bounded in the proof of Lemma~\ref{lem:delta_bi_bound}, such that
\begin{align*}
    \big\|
        \delta \bar{\ell}^{i,m}
    \big\|^2_{\mathbb H^{2,1}_{\beta,T}}
    \leq
    L_{\bar{\boldsymbol{f}}} (1 + L_\kappa)
    E_{\Delta,\beta}^{m-1}.
\end{align*}
For the second part, we use the Lipschitz continuity of $\ell^{i,m}$, which gives
\begin{align*}
    \big\|
        \ell^{i,m}\big(\cdot,\X_{\cdot},\zeta^{i},\bzeta^{-i}\big)
        -
        \ell^{i,m}\big(\cdot,\X_{\cdot},\zeta^{i,m},\bzeta^{-i,m-1}\big)
    \big\|^2_{\mathbb H^{2,1}_{\beta,T}}
      &\leq
      L_\ell^{\FP}\bigg(  \big\|\delta \bar\zeta^{i,m,[\tau]}\big\|^2_{\mathbb H^{2,n}_{\beta,T}}
      +
      \sum_{j\in\I\setminus\{i\}}\big\|\delta \bar\zeta^{j,m-1,[\tau]}\big\|^2_{\mathbb H^{2,n}_{\beta,T}}\bigg).
\end{align*}
In total, this gives the estimate
\begin{align}\label{eq:lemma36_first_term}
    \big\|\delta \hat\ell^{i,m,[\tau]}\big\|^2_{\mathbb H^{2,1}_{\beta,T}} 
    \leq
     2e^{\beta T}L_{\bar{\boldsymbol{f}}} (1 + L_\kappa) E_\Delta^{m-1}
     +
     2 L_\ell^{\FP}\bigg(  \big\|\delta \bar\zeta^{i,m,[\tau]}\big\|^2_{\mathbb H^{2,n}_{\beta,T}}
      +
      \sum_{j\in\I\setminus\{i\}}\big\|\delta \bar\zeta^{j,m-1,[\tau]}\big\|^2_{\mathbb H^{2,n}_{\beta,T}}\bigg)
\end{align}
For the second term in~\eqref{eq:thre_terms_functional}, we use Cauchy--Schwarz and the uniform boundedness~\eqref{eq:grad_bound}, which yields
\begin{align}
\begin{split}
    \big\|\big\langle \delta\hat \bb^{i,m,[\tau]}, \Gamma^{i,m}\big\rangle\big\|^2_{\mathbb H^{2,1}_{\beta,T}}
&=
\E\bigg[\int_0^T e^{\beta t}\big |\big\langle \delta\hat \bb^{i,m,[\tau]}_t, \Gamma_t^{i,m}\big\rangle\big|^2\dt\bigg]\\
&\leq
\E\bigg[\int_0^T e^{\beta t}\big\| \delta\hat \bb^{i,m,[\tau]}_t\big\|^2_{\R^n} \big\|\Gamma_t^{i,m} \big\|^2_{\R^n}\dt\bigg]\\
&\leq M^{\FP}
\big\|\delta\hat \bb^{i,m,[\tau]}\big\|^2_{\mathbb H^{2,n}_{\beta,T}}.
\end{split}
\label{eq:lemma36_second_term}
\end{align}
Using a similar split as for $\delta\hat{\ell}^{i,m}$ and using the same arguments, we have the bound
\begin{align*}
    \big\|\delta\hat \bb^{i,m,[\tau]}\big\|^2_{\mathbb H^{2,n}_{\beta,T}}
    \leq 
    2 e^{\beta T} L_{\bar{\bb}} (1+L_\kappa) E_\Delta^{m-1}
    +
    2L_b^{\FP}\bigg(  \big\|\delta \bar\zeta^{i,m,[\tau]}\big\|^2_{\mathbb H^{2,n}_{\beta,T}}
    +
    \sum_{j\in\I\setminus\{i\}}\big\|\delta \bar\zeta^{j,m-1,[\tau]}\big\|^2_{\mathbb H^{2,n}_{\beta,T}}\bigg).
\end{align*}
Inserting~\eqref{eq:lemma36_first_term} and~\eqref{eq:lemma36_second_term} into~\eqref{eq:thre_terms_functional}, summing over $i\in\I$ and rearranging terms, we have
\begin{align*}
    E^m_{\zeta, \beta}
    \leq
     C^{\Asup}_\Delta(\beta) E_{\Delta,\beta}^{m-1}
    +
     \tilde{K}_{\Asup}(\beta) E_{\zeta,\beta}^{m-1},
\end{align*}
where
\begin{align*}
    \begin{split}
        C^{\Asup}_\Delta(\beta) 
        &\defeq 
        \frac{4N(1+L_\kappa)(L_{\bar{\boldsymbol{f}}} + M^{\FP}L_{\bar{\bb}})}{\beta}, \\
        \tilde{K}_{\Asup}(\beta)
        &\defeq
        \frac{2(N-1)(L_\ell^{\FP} + M^{\FP} L_b^{\FP})}{\beta - 2(L_\ell^{\FP} +\MFP L_b^{\FP})}.
    \end{split}
\end{align*}
This completes the proof.
\end{proof}

\begin{lemma}\label{lemma:functional_error_unbounded}
    Consider the setting in Assumption~\ref{assump:smalltime}. For all $m \geq 1$, there exist $\tilde K_{\Bsup}(\beta) \in (0,1)$ and $C^{\Bsup}_{\Delta}(\beta) < \infty$ such that
    \begin{align*}
        E_{\zeta,\beta}^{m} \leq \tilde{K}_{\Bsup}(\beta) E_{\zeta,\beta}^{m-1} + C^{\Bsup}_{\Delta}(\beta) E_{\Delta,\beta}^{m-1}.
    \end{align*}
\end{lemma}

\begin{proof}
    We apply a classical FBSDE analysis based on the smallness condition in Assumption~\ref{assump:smalltime}. Let $i\in \I$ and $m \geq 1$ be fixed. By the decomposition~\eqref{eq:split} and the inequality~\eqref{ineq:sumsq}, we have
\begin{align}
    \big\|\delta\bar\zeta^{i,m}\big\|^2_{\mathbb H^{2,n}_{\beta,T}}
    &\leq
    2\,\big\|\delta\zeta^{i,m}\big\|^2_{\mathbb H^{2,n}_{\beta,T}}
    +
    2\,
     \big\|\delta\tilde\zeta^{i,m}\big\|^2_{\mathbb H^{2,n}_{\beta,T}}. \label{eq:unbounded_lemma_two_terms}
\end{align}
We estimate the two terms on the right-hand side of~\eqref{eq:unbounded_lemma_two_terms} separately. For the first term, recall that $\delta Y^{i,m} = Y^i-Y^{i,m}$ satisfies for all $t\in[0,T]$, $\P$-almost surely
\begin{align*}
    \delta Y_t^{i,m}
    =
    \delta g^{i,m}
    + \int_t^T \delta \ell_s^{i,m} \ds
    - \int_t^T \big(\delta \zeta_s^{i,m}\big)^\top \,\d\W_s .
\end{align*}
For $\beta > 0$, we apply It\^o's lemma to $(t,y)\mapsto e^{\beta t}\lvert y\rvert^2$ along $\delta Y_t^{i,m}$, which gives $\mathbb P$-almost surely
\begin{align*}
    \int_0^T\e^{\beta s}\big\| \delta \zeta_s^{i,m}\big\|_{\R^n}^2\ds
    &=
    \e^{\beta T}\big| \delta g^{i,m}\big|^2
    + \int_0^T \e^{\beta s}\big(
        2\,\delta Y_s^{i,m}\,\delta \ell_s^{i,m}
        - \beta \big| \delta Y_s^{i,m}\big|^2
    \big)\,\d s \\
    &\quad
    - 2\int_0^T \e^{\beta s}\,\delta Y_s^{i,m}\,\big(\delta \zeta_s^{i,m}\big)^\top \,\d\W_s
    -
    \big| \delta Y_0^{i,m}\big|^2.
\end{align*}
Applying Young's inequality to $2\,\delta Y_s^{i,m}\,\delta \ell_s^{i,m}$ with parameter $\beta$, and discarding the final (non-positive) term, we get
\begin{align*}
    \int_0^T\e^{\beta s}\big\| \delta \zeta_s^{i,m}\big\|_{\R^n}^2 \ds
    \leq
    \e^{\beta T}\big| \delta g^{i,m}\big|^2
    + \frac1\beta\int_0^T \e^{\beta s}\big| \delta \ell_s^{i,m}\big|^2 \ds
    - 2\int_0^T \e^{\beta s}\,\delta Y_s^{i,m}\,
    \big(\delta \zeta_s^{i,m}\big)^\top \,\d\W_s .
\end{align*}
Since $\delta Y^{i,m} \in \mathbb S^{2,1}_{\beta,T}$ and $\delta \zeta^{i,m} \in \mathbb H^{2,n}_{\beta,T}$, we have $\delta Y^{i,m}\delta \zeta^{i,m} \in \mathbb H^{1,n}_{\beta,T}$ and by~\eqref{eq:BDG2} the Itô integral has zero expectation. Taking expectations on both sides, we obtain
\begin{align}\label{eq:delta_zeta_1_new}
    \big\| \delta \zeta^{i,m}\big\|_{\mathbb H^{2,n}_{\beta,T}}^2
    &\leq
    \E\Big[\e^{\beta T}\big| \delta g^{i,m}\big|^2\Big]
    + \frac1\beta\big\| \delta \ell^{i,m}\big\|_{\mathbb H^{2,1}_{\beta,T}}^2 .
\end{align}
By the Lipschitz continuity of $g$, we have
\begin{align*}
    \E\Big[\e^{\beta T}\big| \delta g^{i,m}\big|^2\Big]
    &\leq 
    L_g \, 
    \E\Big[e^{\beta T}\big\|\delta\X^{i,m}_T\big\|_{\R^n}^2\Big]
    \leq
    L_g \E\bigg[\sup_{t\in[0,T]}
    e^{\beta t}\big\|\delta\X^{i,m}_t\big\|_{\R^n}^2\bigg]
    =
    L_g \, \big\|\delta\X^{i,m}\big\|_{\mathbb S^{2,n}_{\beta,T}}^2.
\end{align*}
For the second term of~\eqref{eq:delta_zeta_1_new}, we again split the error $\delta \ell^{i,m}_t = \delta\bar{\ell}^{i,m}_{t} + \delta\tilde{\ell}^{i,m}_{t}$ and get
\begin{align*}
    \big\|
        \delta \bar{\ell}^{i,m}
    \big\|^2_{\mathbb H^{2,1}_{\beta,T}}
    \leq
    L_{\bar{\boldsymbol{f}}} (1 + L_\kappa)
    E_{\Delta,\beta}^{m-1},
\end{align*}
by repeating calculations from Lemma~\ref{lem:delta_bi_bound}. Moreover, by the Lipschitz continuity of $\ell^{i,m}$, the norm inequality~\eqref{eq:HS2}, and Lemma~\ref{lemma:delta_zeta_decomposition}, we have
\begin{align*}
    \big\| \delta\tilde{\ell}^{i,m} \big\|^2_{\mathbb{H}^{2,1}_{\beta,T}}
    &\leq
    L_\ell^{\FP}\bigg(
        T\big\|\delta\X^{i,m}\big\|_{\mathbb S^{2,n}_{\beta,T}}^2
        +
        \big\|\delta \zeta^{i,m}\big\|_{\mathbb H^{2,n}_{\beta,T}}^2
        +
        \sum_{j\in\I\setminus\{i\}}
        \big\|\delta \zeta^{j,i,m-1,m}\big\|_{\mathbb H^{2,n}_{\beta,T}}^2
    \bigg)\\
    &\leq 
    L_\ell^{\FP} T(1 + 2NL_\zeta^{\FP} ) 
        \big\|
          \delta \X^{i,m}
        \big\|_{\mathbb S^{2,n}_{\beta,T}}^2
    +
        2L_\ell^{\FP}
        \big\|
          \delta \bar\zeta^{i,m}
        \big\|_{\mathbb H^{2,n}_{\beta,T}}^2
        +
        2L_\ell^{\FP}
        \sum_{j\in\I\setminus\{i\}}
        \big\|
          \delta \bar \zeta^{j,m-1}
        \big\|_{\mathbb H^{2,n}_{\beta,T}}^2.
\end{align*}
Hence, using~\eqref{ineq:sumsq}, we get
\begin{align*}
    \big\| \delta \ell^{i,m}\big\|_{\mathbb H^{2,1}_{\beta,T}}^2
    &\leq
    2L_\ell^{\FP} T(1 + 2NL_\zeta^{\FP} ) 
    \big\|
      \delta \X^{i,m}
    \big\|_{\mathbb S^{2,n}_{\beta,T}}^2
    +
    4L_\ell^{\FP}
    \big\|
      \delta \bar\zeta^{i,m}
    \big\|_{\mathbb H^{2,n}_{\beta,T}}^2
    \\
    &\qquad+
    4L_\ell^{\FP}
    \sum_{j\in\I\setminus\{i\}}
    \big\|
      \delta \bar \zeta^{j,m-1}
    \big\|_{\mathbb H^{2,n}_{\beta,T}}^2
    +
    2 L_{\bar{\boldsymbol{f}}} (1+L_\kappa) E_{\Delta,\beta}^{m-1}.
\end{align*}
For the second term in~\eqref{eq:unbounded_lemma_two_terms}, the Lipschitz continuity of $\zeta^{i,m}$ and \eqref{eq:HS2} yield
\begin{align}\label{eq:delta_zeta_2_new}
     \big\|\delta\tilde\zeta^{i,m}\big\|^2_{\mathbb H^{2,n}_{\beta,T}}
     \leq
    L_\zeta^{\FP} T
     \big\|\delta\X^{i,m}\big\|^2_{\mathbb S^{2,n}_{\beta,T}}.
\end{align}
Inserting~\eqref{eq:delta_zeta_1_new} and~\eqref{eq:delta_zeta_2_new} into~\eqref{eq:unbounded_lemma_two_terms} gives
\begin{equation}
\label{eq:smallness_ineq}
\begin{split}
\bigg(\frac{1}{2}-\frac{4L_\ell^{\FP}}\beta \bigg)
    \big\| \delta \bar \zeta^{i,m}\big\|_{\mathbb H^{2,n}_{\beta,T}}^2
    &\leq
     \bigg(L_g+\frac{2L_\ell^{\FP} T(1+2L_\zeta^{\FP}  N)}\beta +L_\zeta^{\FP}  T\bigg)
        \big\|\delta\X^{i,m}\big\|_{\mathbb S^{2,n}_{\beta,T}}^2
    \\
    &\qquad
    +
   \frac{4L_\ell^{\FP} }{\beta}\sum_{j\in\I\setminus\{i\}}
        \big\|\delta \bar \zeta^{j,m-1}\big\|_{\mathbb H^{2,n}_{\beta,T}}^2
    +
    \frac{2L_{\bar{\boldsymbol{f}}} (1+L_\kappa)}{\beta} E^{m-1}_{\Delta,\beta}.
\end{split}
\end{equation}
Since $\delta\tau^{i,m}=0$ $\P$-almost surely in the unbounded domain setting, the norm equivalence~\eqref{eq:HHH} followed by Lemma~\ref{lemma:delta_x_bound} gives
\begin{equation}\label{eq:dzetacalc4}
\begin{split}
&\bigg( \frac{1}{2} - C_\beta \bigg)
    \big\| \delta \bar \zeta^{i,m}\big\|_{\mathbb H^{2,n}_{\beta,T}}^2
    \leq
       C_\beta
       \sum_{j\in\I\setminus\{i\}}
        \big\|\delta \bar \zeta^{j,m-1}\big\|_{\mathbb H^{2,n}_{\beta,T}}^2
        +
        C_{\beta}' E_{\Delta,\beta}^{m-1},
\end{split}
\end{equation}
where
\begin{align*}
C_\beta
&\defeq\frac{4L_\ell^{\FP}}\beta
       +Ce^{\beta T}
       \bigg(L_g+\frac{2L_\ell^{\FP} T(1+2L_\zeta^{\FP} N)}\beta
       +
       L_\zeta^{\FP} T\bigg),
\\
C_\beta'
&\defeq
\frac{2 L_{\bar{\boldsymbol{f}}} (1+L_\kappa)}{\beta}
       +Ce^{\beta T}
       \bigg(L_g+\frac{2L_\ell^{\FP} T(1+2L_\zeta^{\FP} N)}\beta
       +
       L_\zeta^{\FP} T\bigg).
\end{align*}
Now define $\tilde{K}_{\Bsup}(\beta) \defeq \frac{2C_\beta}{1-2C_\beta}(N-1)$ and $C^{\Bsup}_{\Delta}(\beta) \defeq \frac{2C_\beta' N}{1-2C_\beta}$. By Assumption~\ref{assump:smalltime}, it holds that $\tilde{K}_{\Bsup}(\beta) < 1$. Following rearrangement of terms and summation over $i \in \I$, we have
\begin{align*}
    \begin{split}
        E_{\zeta,\beta}^{m} \leq \tilde{K}_{\Bsup}(\beta) E_{\zeta,\beta}^{m-1} + C^{\Bsup}_{\Delta}(\beta) E_{\Delta,\beta}^{m-1}.
    \end{split}
    \qedhere
\end{align*}
\end{proof}

\begin{lemma}\label{lem:policy_error}
    For every $m \geq 1$, there exist $C_\zeta(\eta) < \infty$ and $\rho(\eta) < 1$ such that the policy error satisfies
    \begin{align}
        E_{\Delta,\beta}^{m}
        \leq
        \rho(\eta) E_{\Delta,\beta}^{m-1}
        +
        C_\zeta(\eta)
        E_{\zeta,\beta}^{m-1}.
    \end{align}
\end{lemma}
\begin{proof}
    Fix $m \geq 1$. For $(t,\x) \in Q_T$, we have by the Nash fixed-point relation,
    \begin{align*}
        \alpha^{*,i}(t,\x)
        = 
        \kappa^i
        \big(
            t,
            \x,
            \bPhi(t,\x) \zeta^i(t,\x),
            \bal^{*,-i}(t,\x)
        \big).
    \end{align*}
    By the definition of $\kappa^{i,m}$ and the contraction property~\eqref{eq:kappai_lipschitz}, we have
    \begin{align*}
        \sum_{i \in \I}
        \big\|
            \Delta^{i,m}(t,\x)
        \big\|_{\R^d}^2
        \leq
        L_{\kappa}^a
        \sum_{i \in \I}
        \big\|
            \alpha^{*,i}(t,\x)
            -
            \kappa^{i,m-1}(t,\x,\zeta^{i,m-1}(t,\x))
        \big\|_{\R^d}^2.
    \end{align*}
    For each $i \in \I$, add and subtract $\kappa^{i,m-1}(t,\x,\zeta^{i}(t,\x))$. By Young's inequality, for any $\eta > 0$, we have
    \begin{align*}
        \sum_{i \in \I}
        \big\|
            \Delta^{i,m}(t,\x)
        \big\|_{\R^d}^2
        &\leq
        (1+\eta)L_{\kappa}^a
        \sum_{i \in \I}
        \big\|
            \Delta^{i,m-1}(t,\x)
        \big\|_{\R^d}^2
        \\
        &\qquad
        +
        \big( 1 + \eta^{-1} \big) L_{\kappa}^a
        \sum_{i \in \I}
        \big\|
            \kappa^{i,m-1}(t,\x,\zeta^i(t,\x))
            -
            \kappa^{i,m-1}(t,\x,\zeta^{i,m-1}(t,\x))
        \big\|_{\R^d}^2.
    \end{align*}
    For the final sum, we apply the uniform Lipschitz continuity of $\kappa^{i,m-1}$, such that
    \begin{align*}
        \sum_{i \in \I}
        \big\|
            \Delta^{i,m}(t,\x)
        \big\|_{\R^d}^2
        &\leq
        (1+\eta)L_{\kappa}^a
        \sum_{i \in \I}
        \big\|
            \Delta^{i,m-1}(t,\x)
        \big\|_{\R^d}^2
        +
        \big( 1 + \eta^{-1} \big) L_{\kappa}^aL_{\kappa}^{\FP}
        \sum_{i \in \I}
        \big\|
            \zeta^i(t,\x)
            -
            \zeta^{i,m-1}(t,\x)
        \big\|_{\R^n}^2.
    \end{align*}
    Evaluating at $\x = \X_t$, multiplying by $e^{\beta t}$, integrating over $[0,T]$, and taking expectation gives
    \begin{align*}
        E^m_{\Delta,\beta}
        \leq
        (1+\eta)L_{\kappa}^a E_{\Delta,\beta}^{m-1}
        +
        \big( 1 +\eta^{-1} \big) L^a_{\kappa}L_{\kappa}^{\FP}
        E_{\zeta,\beta}^{m-1}.
    \end{align*}
    Since $L^a_\kappa < 1$, we can choose $\eta > 0$ such that
    \begin{align*}
        \rho(\eta) := (1 + \eta) L_\kappa^a < 1.
    \end{align*}
    With $C_\zeta(\eta) := (1+\eta^{-1})L_\kappa^a L_{\kappa}^{\FP}$, we obtain
    \begin{align*}
        \begin{split}
            E_{\Delta,\beta}^{m}
            \leq
            \rho(\eta) E_{\Delta,\beta}^{m-1}
            +
            C_\zeta(\eta)
            E_{\zeta,\beta}^{m-1}.
        \end{split}
         \qedhere
    \end{align*}
\end{proof}

\begin{lemma}\label{lem:spectral_radius}
    The matrices
    \begin{align*} 
        M_{\Asup}
        \defeq
        \begin{bmatrix}
            \tilde{K}_{\Asup}(\beta)  & C_\Delta^{\Asup}(\beta)
            \\
            C_\zeta(\eta) & \rho(\eta)
        \end{bmatrix},
        \quad
        \text{ and }
        \quad
        M_{\Bsup}
        \defeq
        \begin{bmatrix}
            \tilde{K}_{\Bsup}(\beta)  & C_\Delta^{\Bsup}(\beta)
            \\
            C_\zeta(\eta) & \rho(\eta)
        \end{bmatrix},
    \end{align*}
    under Assumption~\ref{assump:gradient} and Assumption~\ref{assump:smalltime}, respectively, have spectral radius $r(M_{\Asup}), r(M_{\Bsup}) < 1$. 
\end{lemma}
\begin{proof}
    We begin by proving the result for $M_{\Asup}$. Since all entries of $M_{\Asup}$ are non-negative, its spectral radius is given by its Perron root. For a $2 \times 2$ matrix with non-negative entries, this root is given by
    \begin{align*}
        r(M_{\Asup}) = \frac{\tilde{K}_{\Asup}(\beta) + \rho(\eta) + \sqrt{(\tilde{K}_{\Asup}(\beta) - \rho(\eta))^2 + 4 C_\Delta^{\Asup}(\beta) C_\zeta(\eta)}}{2}.
    \end{align*}
    The statement $r(M_{\Asup}) < 1$ thus holds if we can show that
    \begin{align*}
        \sqrt{(\tilde{K}_{\Asup}(\beta) - \rho(\eta))^2 + 4 C_\Delta^{\Asup}(\beta) C_\zeta(\eta)} < 2 - \tilde{K}_{\Asup}(\beta) - \rho(\eta).
    \end{align*}
    Note that, since both $\tilde{K}_{\Asup}(\beta) < 1$ and $\rho(\eta) < 1$, the right-hand side is strictly positive. Squaring both sides and rearranging terms give the condition
    \begin{align*}
         C_\Delta^{\Asup}(\beta) C_\zeta(\eta) < (1-\tilde{K}_{\Asup}(\beta))(1 - \rho(\eta)).
    \end{align*}
    We note that, given an $\eta$ such that $\rho(\eta) < 1$, we can consider $C_\zeta(\eta)$ and $\rho(\eta)$ fixed. Since $C^{\Asup}_\Delta(\beta)$ and $\tilde{K}_{\Asup}(\beta)$ both vanish as $\beta \rightarrow \infty$, it is always possible to choose $\beta$ sufficiently large such that this condition holds. Thus, it holds that $r(M_{\Asup}) < 1$. By Assumption~\ref{assump:smalltime}, the same conditions are satisfied for $M_{\Bsup}$, and it follows that $r(M_{\Bsup}) < 1$.
\end{proof}

\begin{proposition}\label{prop:functional_error}
Under either Assumption~\ref{assump:gradient} or~\ref{assump:smalltime}, there exist $\tilde K \in (0,1)$ and $C < \infty$ such that for all $m \geq 1$, it holds
\begin{align}
    E_{\zeta}^{m} + E_{\Delta}^{m} \leq C \tilde{K}^m.\label{eq:prop_functional_error}
\end{align}
\end{proposition}

\begin{proof}
    Under Assumption~\ref{assump:gradient}, Lemma~\ref{lem:functional_error_bounded} combined with Lemma~\ref{lem:policy_error} gives the coupled bound
    \begin{align*}
        E_{\zeta,\beta}^{m}  &\leq \tilde{K}_{\Asup} E^{m-1}_{\zeta,\beta} + C_{\Delta}^{\Asup} E^{m-1}_{\Delta,\beta}, \\
        E_{\Delta,\beta}^{m} &\leq C_\zeta(\eta) E^{m-1}_{\zeta,\beta} + \rho(\eta) E^{m-1}_{\Delta,\beta}.
    \end{align*}
    On vector form, this can be written as
    \begin{align*}
        U_m \leq M_{\Asup} U_{m-1} \leq M_{\Asup}^m U_0,
    \end{align*}
    where $U_m \defeq [E_{\zeta,\beta}^{m},E_{\Delta,\beta}^{m}]^\top$, and $M_{\Asup}$ is the matrix from Lemma~\ref{lem:spectral_radius}. Taking the induced $\ell^1$-norm on both sides, we have
    \begin{align}
        \| U_m \|_1
        \leq
        \| M_{\Asup}^m \|_1
        \| U_0 \|_1.\label{eq:Um_estimate}
    \end{align}
    By Lemma~\ref{lem:spectral_radius}, it holds under Assumption~\ref{assump:gradient} that $r(M_{\Asup}) < 1$. Thus, from the Gelfand formula, for every $\tilde{K} \in (r(M_{\Asup}),1)$ there exists $C_K \geq 1$ such that
    \begin{align*}
        \|M_{\Asup}^m\|_1 \leq C_K \tilde{K}^m.
    \end{align*}
    Inserting this into~\eqref{eq:Um_estimate} gives
    \begin{align*}
        E_{\zeta,\beta}^{m} + E_{\Delta,\beta}^{m} \leq C_K \tilde{K}^m \big( E_{\zeta,\beta}^{0} + E_{\Delta,\beta}^{0} \big).
    \end{align*}
    The norm equivalence~\eqref{eq:HHH} finishes the proof. The proof under Assumption~\ref{assump:smalltime} follows analogously.
\end{proof}

\subsubsection{Main theorem}

We now conclude the convergence analysis by combining the results from Propositions~\ref{prop:deltaXYZ}--\ref{prop:functional_error}. The following theorem states the main contribution of our paper.

\begin{theorem}
    \label{theorem:conv_FP_FBSDE}
Let either Assumption~\ref{assump:gradient} or~\ref{assump:smalltime} hold. Then there exist $K \in (0,1)$ and $C < \infty$ such that for all $m \geq 1$, it holds
\begin{align}\label{eq:XYZ}
            \max_{i\in\I}\big\|\delta\X^{i,m}\big\|^2_{\mathbb S^{2,n}_{T}}
            +&
            \sum_{i\in\I}\big\|\delta Y^{i,m}\big\|^2_{\mathbb S^{2,1}_{T}}
            +
            \sum_{i\in\I}\Big\|\zeta^{i,[\tau]}(\cdot,\X_\cdot)-\zeta^{i,m,[\tau^{i,m}]}\big(\cdot,\X_\cdot^{i,m}\big)\Big\|_{\mathbb H^{2,n}_{T}}^2
            \leq
CK^{m}.
\end{align}
\end{theorem}

\begin{proof}

Recall that $\Psi^{m}$ from~\eqref{eq:Psim} is the left-hand side of~\eqref{eq:XYZ}. From Proposition~\ref{prop:deltaXYZ}, it holds that
\begin{align}
    \Psi^{m} 
    \lesssim 
    E_\zeta^{m}
    +
    E_\zeta^{m-1}
    +
    E_\Delta^{m-1}
    +
    \sum_{i\in\I}
    \big\|\delta\tau^{i,m}\big\|_{\mathbb{L}^{1,1}}
    +
    \big\|\delta\tau^{i,m}\big\|_{\mathbb{L}^{1,1}}^{1-\varepsilon}.\label{eq:mainthm_bound}
\end{align}  
The last two terms can be bounded using Proposition~\ref{prop:stoppin_time_L1}. More precisely, for $\alpha \in \{1, 1 - \varepsilon\}$, we have
\begin{align*}
    \big\|
        \delta\tau^{i,m}
    \big\|_{\mathbb{L}^{1,1}}^{\alpha}
    &\lesssim 
    \bigg(
        \Big\|
            \delta\bar \zeta^{i,m,[\tau_{\min}^{i,m}]}
        \Big\|_{\mathbb H^{2,n}_{T}}^2
        +
        \sum_{j\in\I\setminus\{i\}}
        \Big\|
            \delta \bar{\zeta}^{j,m-1,[\tau_\text{min}^{i,m}]}
        \Big\|_{\mathbb{H}^{2,n}_{T}}^2 
        + 
        E_{\Delta}^{m-1}
    \bigg)^{\frac\alpha2}\\
    &\leq
    \bigg(
        \Big\|
            \delta\bar \zeta^{i,m,[\tau_{\min}^{i,m}]}
        \Big\|_{\mathbb H^{2,n}_{T}}^2
    \bigg)^{\frac\alpha2}
    +
    \sum_{j\in\I\setminus\{i\}}
    \bigg(
        \Big\|
            \delta \bar{\zeta}^{j,m-1,[\tau_\text{min}^{i,m}]}
        \Big\|_{\mathbb{H}^{2,n}_{T}}^2
    \bigg)^{\frac\alpha2} 
    + 
    \big( 
        E_{\Delta}^{m-1}
    \big)^{\frac{\alpha}{2}} \\
    &\leq
    \bigg(
        \Big\|
            \delta\bar \zeta^{i,m,[\tau]}
        \Big\|_{\mathbb H^{2,n}_{T}}^2
    \bigg)^{\frac\alpha2}
    +
    \sum_{j\in\I\setminus\{i\}}
    \bigg(
        \Big\|
            \delta \bar{\zeta}^{j,m-1,[\tau]}
        \Big\|_{\mathbb{H}^{2,n}_{T}}^2
    \bigg)^{\frac\alpha2} 
    +
    \big( 
        E_{\Delta}^{m-1}
    \big)^{\frac{\alpha}{2}},
\end{align*}
where the second inequality follows from subadditivity of the function $\varphi(y) = y^{\alpha/2}$, and the last follows from the fact that $\tau^{i,m}_{\min} \leq \tau$. By summing over all players $i \in \I$, the above inequality becomes
\begin{align}
\sum_{i\in\I}
\big\|\delta\tau^{i,m}\big\|_{\mathbb{L}^{1,1}}^{\alpha}
&\lesssim 
  \sum_{i\in\I}
  \bigg(
            \Big\|\delta\bar \zeta^{i,m,[\tau]}
            \Big\|_{\mathbb H^{2,n}_{T}}^2\bigg)^{\frac\alpha2}
            +
            \sum_{i\in \I}
            \sum_{j\in\I\setminus\{i\}}\bigg(
\Big\|\delta \bar{\zeta}^{j,m-1,[\tau]}\Big\|_{\mathbb{H}^{2,n}_{T}}^2\bigg)^{\frac\alpha2} + N\big( E_\Delta^{m-1} \big)^{\frac{\alpha}{2}}. \label{eq:mainthm_before_jensen}
\end{align}
Next, recall that from Jensen's inequality for concave functions it holds that
\begin{align*}
   \sum_{i=1}^N \varphi(y_i) \leq N \varphi\bigg( \frac{1}{N} \sum_{i=1}^N y_i  \bigg).
\end{align*}
Applying this to~\eqref{eq:mainthm_before_jensen} yields
\begin{align*}
\sum_{i\in\I}
\big\|\delta\tau^{i,m}\big\|_{\mathbb{L}^{1,1}}^{\alpha}
&\lesssim 
  N^{1-\frac\alpha2}
  \bigg(\sum_{i\in\I}
            \Big\|\delta\bar \zeta^{i,m,[\tau]}
            \Big\|_{\mathbb H^{2,n}_{T}}^2\bigg)^{\frac\alpha2}
            +
            N(N-1)^{1-\frac\alpha2}
            \bigg(\sum_{j\in\I\setminus\{i\}}
\Big\|\delta \bar{\zeta}^{j,m-1,[\tau]}\Big\|_{\mathbb{H}^{2,n}_{T}}^2\bigg)^{\frac\alpha2} + N\big( E_\Delta^{m-1} \big)^{\frac{\alpha}{2}}\\
&\lesssim
    \big(
    E_\zeta^m
    \big)^{\frac\alpha2}
    +
    \big(
        E_{\zeta}^{m-1}
    \big)^{\frac\alpha2} + \big( E_\Delta^{m-1}\big)^{\frac{\alpha}{2}}.
\end{align*}
Inserting this bound into~\eqref{eq:mainthm_bound} and applying Proposition~\ref{prop:functional_error} gives for the case $|\D|<\infty$ the bound
\begin{align*}
    \Psi^{m}&
    \lesssim
    \sum_{\alpha\in\{1-\varepsilon,1,2\}}
    \big(  E_\zeta^m + E_\zeta^{m-1} +  E_\Delta^{m-1} \big)^\frac\alpha2
    \lesssim
    \sum_{\alpha\in\{1-\varepsilon,1,2\}}
    \big( \tilde{K}^{m} + \tilde{K}^{m-1} \big)^{\frac{\alpha}{2}}
    \lesssim
    \tilde K^{\frac{(1-\varepsilon)m}2}
    =
    K^m,
\end{align*}
where $K=\tilde K^{\frac{1-\varepsilon}2}$. For the case $\D=\R^n$, where $\|\delta\tau^{i,m}\|_{\mathbb{L}^{1,1}}=0$, we instead have the bound
\begin{align*}
    \begin{split}
        \Psi^{m}
        \lesssim
         E_\zeta^m + E_\zeta^{m-1} +  E_\Delta^{m-1}
        \lesssim
        \tilde{K}^m + \tilde{K}^{m-1}
        \lesssim K^m.
    \end{split}
    \qedhere
\end{align*}
\end{proof}

\begin{remark}
    Theorem~\ref{theorem:conv_FP_FBSDE} establishes a geometric convergence rate of the fictitious-play scheme under both Assumption~\ref{assump:gradient} and Assumption~\ref{assump:smalltime}. Under Assumption~\ref{assump:gradient}, this rate can be sharpened to a super-exponential rate in the special case where the best-response map does not depend on the opponents' actions. This is the case, for instance, in Example~\ref{ex:1}.

    Suppose that $\kappa^i(t,\x,p,\boldsymbol{a}^{-i})$ is independent of $\boldsymbol{a}^{-i}$ for each player $i \in \I$. Then, by definition, $\Delta^i(t,\x) \equiv 0$, and in particular $E_\Delta^m \equiv 0$. Consequently, Lemma~\ref{lem:functional_error_bounded}, which measures the Markov-map error under Assumption~\ref{assump:gradient}, reduces to
    \begin{align*}
        E_\zeta^m 
        \leq E_{\zeta,\beta}^m 
        \leq \tilde{K}_{\Asup}(\beta) E_{\zeta,\beta}^{m-1}
        \leq \big(\tilde{K}_{\Asup}(\beta)\big)^m  E_{\zeta,\beta}^{0} 
        \leq e^{\beta T} \big(\tilde{K}_{\Asup}(\beta)\big)^m E_{\zeta}^0.
    \end{align*}
    As in Lemma~\ref{lem:functional_error_bounded}, $\tilde{K}_{\Asup}(\beta)$ is of order $\mathcal{O}(1/\beta)$. Choosing $\beta$ of order $m/T$ therefore gives
    \begin{align*}
        E_\zeta^m \lesssim \exp(-cm\log(m))
    \end{align*}
    for some constant $c > 0$. Inserting this estimate into the final error bound used in the proof of Theorem~\ref{theorem:conv_FP_FBSDE}, and using $E_\Delta^m \equiv 0$, gives the same super-exponential decay for the fictitious-play error.
\end{remark}

\color{black}

\section{Numerical example} \label{sec:NE}

We illustrate the convergence properties of fictitious play using the linear--quadratic interbank borrowing and lending model introduced in~\cite{fouque2015mean}. This example provides a convenient benchmark for the numerical method, since both the Nash equilibrium and the fictitious-play iterates admit explicit representations. In particular, this allows for a direct comparison between the numerical approximations and the corresponding analytical solutions. 

We consider an $N$-player game, where the setting is a special case of the linear--quadratic game from Example~\ref{ex:LQ}. The state space is $\D = \R^N$, and each player has action space $A^1 = \cdots = A^N = \R$. Denoting by $\mathbf{1} \in \R^N$ the vector of ones, we define the empirical mean $\mu \colon \R^N \to \R$ by
\begin{align*}
    \mu(\x) \defeq \frac{1}{N} \mathbf{1}^\top \x,
\end{align*}
and the corresponding vectorized function $\bmu(\x) \defeq \mu(\x) \mathbf{1}$. Let $\beta, \rho, \sigma, \delta, \varepsilon, \gamma \in \R$ be scalar parameters, whose values are specified below. For $(t,\x) \in Q_T$ and $\ba \in \R^N$, the drift coefficient is given by
\begin{align*}
    \bar{\bb}(t,\x,\ba)
    =
    \beta
    (
        \bmu(\x)
    -
        \x
    )
    +
    \ba,
\end{align*}
and the diffusion coefficient by
\begin{align*}
\bSigma(t,\x)
 =
\sigma
\begin{pmatrix}
\rho & \sqrt{1-\rho^2} & 0 & \cdots & 0\\
\rho & 0 & \sqrt{1-\rho^2} & \cdots & 0\\
\vdots & \vdots & \vdots & \ddots & \vdots\\
\rho & 0 & 0 & \cdots & \sqrt{1-\rho^2}
\end{pmatrix}.
\end{align*}
For $(t,\x) \in Q_T$ and $\ba \in \R^N$, the running and terminal costs of player $i \in \I$ are given by
\begin{align*}
\bar{f}^i(t,\x,\ba)
&=
\frac12
\big(
  a^i
\big)^2
-\delta\,a^i
\big(
  \mu(\x)-x^i
\big)
+\frac{\varepsilon}{2}
\big(
  \mu(\x)-x^i
\big)^2,\\
g^i(\x)
&=
\frac{\gamma}{2}
\big(\mu(\x)-x^i\big)^2.
\end{align*}
Recall that the HJB equation admits a semi-analytical solution of the form~\eqref{eq:LQ_Vi}. Using the ansatz
\begin{align*}
    P^i_t 
    = 
    \eta_t 
    \Bigl( \boldsymbol{e}_i - \frac{1}{N}\mathbf{1} \Bigr)
    \Bigl( \boldsymbol{e}_i - \frac{1}{N}\mathbf{1} \Bigr)^\top,
\end{align*}
as a solution to the Riccati system, one can show that the value function of player $i$ takes the form
\begin{align*}
    V^i(t,\x) 
    = 
    \frac{1}{2} 
    \eta_t
    \left( \mu(\x)-x^i \right)^2
    +
    \nu_t,
\end{align*}
where the scalar functions $\eta_t$ and $\nu_t$ satisfy
\begin{align*}
    \dot\eta_t
    &=
    \Bigl(1-\frac1{N^2}\Bigr)\eta_t^2
    +
    2(\beta+\delta)\eta_t
    -
    \varepsilon+\delta^2, \\
    \nu_t
    &=
    \frac12\,\sigma^2\big(1-\rho^2\big)\Bigl(1-\frac1N\Bigr)
    \int_t^T \eta_s\,\d s, 
\end{align*}
with the terminal condition $\eta_T = \gamma$. From this we obtain the semi-analytical reference solution for the processes $\Y$ and $\Z$. The forward process $\X$ can then be solved for explicitly.

We initialize the fictitious-play scheme with $\bzeta^{0} \equiv 0$. For $m \geq 1$, we consider the fictitious-play HJB equations~\eqref{eq:HJB-best-response} with an ansatz analogous to the equilibrium case. This leads to a value function on the similar form
\begin{align*}
    V^{i,m}(t,\x) 
    = 
    \frac{1}{2} 
    \eta^{m}_t
    \left( \mu(\x)-x^i \right)^2
    +
    \nu^{m}_t,
\end{align*}
where the scalar function $\eta^{m}_t$ satisfies
\begin{align*}
    \dot\eta^{m}_t
    &=
    \Bigl(1-\frac1N\Bigr)^2\big(\eta^{m}_t\big)^2
    +
    2\bigg(\beta+\frac{K^{m-1}(t)}{N}
    +
    \Bigl(1-\frac1N\Bigr)\delta\bigg)\eta^{m}_t
    -
    \varepsilon+\delta^2,
\end{align*}
with $K^{m}(t) \defeq \delta+(1-\frac1N)\eta^{m}(t)$, and where
\begin{align*}
    \nu^{m}_t
    =
    \frac12\,\sigma^2\big(1-\rho^2\big)\Bigl(1-\frac1N\Bigr)
    \int_t^T \eta^{m}_s\,\d s.
\end{align*}
This again yields a semi-analytical solution for the fictitious-play iterates. Since the derivation is standard in the linear--quadratic setting, we omit the details.

In the numerical example, we fix
\begin{align*}
    \beta = 0.1, \quad 
    \delta = 0.2, \quad
    \varepsilon = 0.5, \quad
    \gamma = 0.5, \quad
    \rho = 0.2, \quad
    \sigma = 1, \quad
    T = 10, 
\end{align*}
and consider the cases $N = 2$ and $N = 20$. We define the error metrics by
\begin{align*}
  \mathrm{Err}_{\scalebox{0.65}{$X$}} \defeq \max_{i\in\mathcal I}\big\|\delta \X^{i,m}\big\|_{\mathbb S^{2,N}_{T}}^2,
  \quad
  \mathrm{Err}_{\scalebox{0.65}{$Y$}} \defeq \sum_{i\in\mathcal I}\big\|\delta Y^{i,m}\big\|_{\mathbb S^{2,1}_{T}}^2 \quad\textrm{and}\quad
  \mathrm{Err}_{\scalebox{0.65}{$Z$}} \defeq \sum_{i\in\mathcal I}\big\|\delta Z^{i,m}\big\|_{\mathbb H^{2,N+1}_{T}}^2.
\end{align*}
Since these norms involve expectations, the reported values are their Monte Carlo estimators $\widehat{\mathrm{Err}}_{\scalebox{0.65}{$X$}}$, $\widehat{\mathrm{Err}}_{\scalebox{0.65}{$Y$}}$, and $\widehat{\mathrm{Err}}_{\scalebox{0.65}{$Z$}}$.

The Riccati equation for $\eta$ admits a closed-form solution, which we evaluate explicitly. For $\eta^{m}$, we solve the Riccati equation on each time interval by using its explicit solution with the previous iterate frozen on that interval. The quantities $\nu$ and $\nu^{m}$ are approximated on the same uniform grid using the trapezoidal rule. For the Monte Carlo approximation of the Brownian motion, we use $2^{6}$ samples, and we keep the same set of trajectories at every fictitious-play iteration so that the reported errors reflect only the dependence on the iteration index $m$. We use $5 \cdot 10^6$ time steps. The initial conditions are $\x_0 = (-4.95, 1.84)^\top$ and
\begin{align*}
    \x_0 = 
    \big(
        -&4.95,\, -1.84,\,  6.44,\,  0.97,\,   4.60,\,
        2.89,\,   -3.18,\,  2.71,\, -1.58,\,  -1.61, \\
        &0.49,\,  -7.63,\,  5.96,\, -3.36,\,   5.00,\, 
        0.68,\,    7.66,\,  -3.30,\, -1.56,\,  1.69
    \big)^\top
\end{align*}
for $N = 2$ and $N = 20$, respectively.

\begin{figure}
    \centering
    \input{numplot_v2}
    \caption{Convergence of the fictitious-play scheme for the interbank model in the cases $N=2$ (left) and $N=20$ (right).}
    \label{fig:placeholder}
\end{figure}
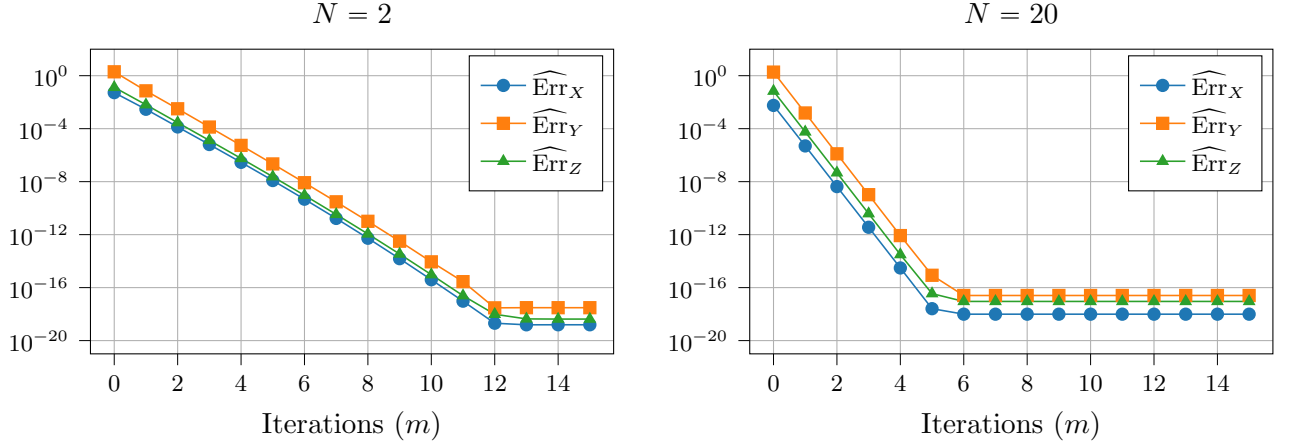

Figure~\ref{fig:placeholder} shows the convergence results for the numerical example. For both population sizes and all error metrics, the fictitious-play iterates exhibit clear exponential convergence until the errors reach machine precision, or values close to it. For $N = 2$ and $N = 20$, the errors decrease by factors of approximately $23$ and $1000$ per iteration, respectively. The faster convergence for $N=20$ is consistent with the weaker coupling between players as the population size increases. Indeed, in the limit $N \to \infty$, the game converges to a mean-field control problem. These results are consistent with our theoretical findings, even though this example does not satisfy the assumptions under which the convergence theory is established. This suggests that fictitious play may converge beyond the range of our theoretical assumptions.

\appendix

\section{Notes on the Lipschitz condition}\label{sec:lipschitz}

In Section~\ref{sec:setting3}, we introduced the Lipschitz condition~(L) for the functions $\bb$, $\f$, $\bb^{i,m}$, $\ell^{i,m}$ and $\kappa^{i,m}$. Under Assumption~\ref{assump:smalltime}, this condition is imposed directly. The purpose of this appendix is to verify that the local version of (L) follows under Assumption~\ref{assump:gradient}, and to identify a simple setting in which the global version follows from the standard Lipschitz assumptions.

The obstruction to global Lipschitz continuity is the variable $\bPhi(t,\x)\z$. The estimates below show that condition~(L) holds in either of the following two cases:
\begin{enumerate}
    \item[(I)] The relevant $\z$ arguments are uniformly bounded. That is, $\z$ belongs to the set of matrices with norm bounded by some $R_{\z} < \infty$. Under Assumption~\ref{assump:gradient}, this follows from the gradient bound.
    \item[(II)] The noise is additive and time-independent, so that $\bPhi$ is independent of $(t,\x)$.
\end{enumerate}
Throughout this section, we use the compact notation
\begin{align*}
    D_{t\x} = |t_1 - t_2| + \|\x_1 - \x_2\|^2_{\R^n}, \qquad D_{\z} = \tnorm{\z_1 - \z_2}^2_{\R^{n \times N}}, \qquad D_{\boldsymbol{a}^{-i}} = \tnorm{\boldsymbol{a}^{-i}_1 - \boldsymbol{a}^{-i}_2}^2_{\R^{d \times (N-1)}},
\end{align*}
and $D_{z}$ the adaptation of $D_{\z}$ onto $\R^n$. We show that the functions satisfy the bounds
\begin{align}
    \begin{split}
        \|\bb(t_1,\x_1,\z_1)-\bb(t_2,\x_2,\z_2)\|_{\R^{n}}^2
        &\leq
        L_{\bb}
        \big(
          D_{t\x}
          +
           D_{\z}
        \big),\\
        \|\f(t_1,\x_1,\z_1)-\f(t_2,\x_2,\z_2)\|_{\R^{N}}^2
        &\leq
        L_{\f}
        \big(
          D_{t\x}
          +
           D_{\z}
        \big),\\
        \big\|
            \kappa^{i,m}(t_1,\x_1,z_1)
            -
            \kappa^{i,m}(t_2,\x_2,z_2)
        \big\|_{\R^d}^2
        &\leq
        L_\kappa^{\FP}
        \big(
            D_{t\x}
            +
            D_{z}
        \big),\\   
        \big\|
          \bb^{i,m}(t_1,\x_1,z_1^i,\z_1^{-i})-\bb^{i,m}(t_2,\x_2,z^i_2,\z_2^{-i})
        \big\|_{\R^{n}}^2
        &\leq
        L_{b}^{\text{FP}}
        \big(
          D_{t\x}
          +
           D_{\z}
        \big),\\    
        \big|
          \ell^{i,m}(t_1,\x_1,z_1^i,\z_1^{-i})
          -
          \ell^{i,m}(t_2,\x_2,z_2^i,\z_2^{-i})
        \big|^2
        &\leq
        L_{\ell}^{\text{FP}}
        \big(
          D_{t\x}
          +
        D_{\z}
        \big).
    \end{split}\label{eq:lipschtiz_bounds}
\end{align}
The following lemma shows the Lipschitz continuity for the functions $\bb$ and $\f$.
\begin{lemma}\label{lem:lipschitz_fb}
    If case (I) holds, the functions $\bb$ and $\f$ satisfy the bounds in~\eqref{eq:lipschtiz_bounds} with constants
    \begin{align*}
        L_{\bb} 
        &= 
        L_{\bar{\bb}} 
        \max
        \big\{
            1 + L_{\bc} + 2L_{\bc}L_{\bPhi}R_{\z}^2,
            2L_{\bc}C_{\bPhi}^2
        \big\},
        \\
        L_{\f} 
        &= 
        L_{\bar{\f}} 
        \max
        \big\{
            1 + L_{\bc} + 2L_{\bc}L_{\bPhi}R_{\z}^2,
            2L_{\bc}C_{\bPhi}^2
        \big\}.
    \end{align*}
    Alternatively, if case (II) holds, the functions $\bb$ and $\f$ satisfy the bounds with constants
    \begin{align*}
        L_{\bb} 
        &= 
        L_{\bar{\bb}} 
        \max
        \big\{
            1 + L_{\bc},
            2L_{\bc}C_{\bPhi}^2
        \big\},
        \\
        L_{\f} 
        &= 
        L_{\bar{\f}} 
        \max
        \big\{
            1 + L_{\bc},
            2L_{\bc}C_{\bPhi}^2
        \big\}.
    \end{align*}
\end{lemma}
\begin{proof}
    Using the definition of $\bb$, and the Lipschitz continuity of $\bar{\bb}$ and $\bc$, we obtain
    \begin{align*}
        &\big\|
            \bb(t_1, \x_1, \z_1)
            -
            \bb(t_2, \x_2, \z_2)
        \big\|_{\R^n}^2 
        \\
        &\qquad=
        \big\|
            \bar{\bb}(t_1, \x_1, \bc(t_1,\x_1,\bPhi(t_1,\x_1)\z_1))
            -
            \bar{\bb}(t_2, \x_2, \bc(t_2,\x_2,\bPhi(t_2,\x_2)\z_2))
        \big\|^2_{\R^n}
        \\
        &\qquad\leq
        L_{\bar{\bb}}
        \big(
            D_{t\x}
            +
            \tnorm{
                \bc(t_1,\x_1,\bPhi(t_1,\x_1)\z_1))
                -
                \bc(t_2,\x_2,\bPhi(t_2,\x_2)\z_2))}_{\R^{d\times N}}^2
        \big)
        \\
        &\qquad\leq
        L_{\bar{\bb}}
        \Big(
            D_{t\x}
            +
            L_{\bc}
            \big(
                D_{t\x}
                +
                \tnorm{
                    \bPhi(t_1,\x_1)\z_1
                    -
                    \bPhi(t_2,\x_2)\z_2}^2_{\R^{n \times N}}
            \big)
        \Big).
    \end{align*}
    For the final term, adding and subtracting $\bPhi(t_1,\x_1)\z_2$ gives
    \begin{align*}
        \tnorm{
            \bPhi(t_1,\x_1)\z_1
            -
            \bPhi(t_2,\x_2)\z_2}^2_{\R^{n \times N}}
        &\leq
        2
        \tnorm{
            \bPhi(t_1,\x_1)(\z_1 - \z_2)}^2_{\R^{n \times N}}
        +
        2
        \tnorm{
            \big(
                \bPhi(t_1, \x_1)
                -
                \bPhi(t_2, \x_2)
            \big)
            \z_2}^2_{\R^{n \times N}}
        \\
        &\leq
        2C_{\bPhi}^2 D_{\z}
        +
        2L_{\bPhi}D_{t\x}
        \tnorm{
            \z_2}_{\R^{n\times N}}^2.
    \end{align*}
    Consequently, we have
    \begin{align*}
        \big\|
            \bb(t_1, \x_1, \z_1)
            -
            \bb(t_2, \x_2, \z_2)
        \big\|_{\R^n}^2
        &\leq
        L_{\bar{\bb}}
        \big(
            D_{t\x}
            +
            L_{\bc}
            \big(
                D_{t\x}
                +
                2C_{\bPhi}^2 D_{\z}
                +
                2L_{\bPhi}D_{t\x}\tnorm{
            \z_2}_{\R^{n\times N}}^2
            \big)
        \big)
        \\
        &\leq
        L_{\bar{\bb}}
        \big(
            (1 + L_{\bc} + 2L_{\bc}L_{\bPhi}\tnorm{
            \z_2}_{\R^{n\times N}}^2)
            D_{t\x}
            +
            2L_{\bc}C_{\bPhi}^2 
            D_{\z}
        \big).
    \end{align*}
    Thus, if case (I) holds, then $\tnorm{\z_2}_{\R^{n \times N}} \leq R_{\z}$, and $\bb$ is Lipschitz with constant 
    \begin{align*}
        L_{\bb}
        = 
        L_{\bar{\bb}} 
        \max
        \big\{
            1 + L_{\bc} + 2L_{\bc}L_{\bPhi}R_{\z}^2,
            2L_{\bc}C_{\bPhi}^2
        \big\}.
    \end{align*}
    Alternatively, if case (II) holds, then $\bPhi$ is independent of $(t,\x)$ and $L_{\bPhi}=0$. The bound then becomes global with constant
    \begin{align*}
        L_{\bb}
        = 
        L_{\bar{\bb}} 
        \max
        \big\{
            1 + L_{\bc},
            2L_{\bc}C_{\bPhi}^2
        \big\}.
    \end{align*}
    The same arguments apply to $\f$, with $L_{\bar{\bb}}$ replaced by $L_{\bar{\f}}$.
\end{proof}
Next, we focus on the Lipschitz continuity of $\kappa^{i,m}$. The following lemma first shows Lipschitz continuity of $\boldsymbol{\alpha}^m$, from which the result for $\kappa^{i,m}$ follows.

\begin{lemma}
    If case (I) or (II) holds, then there exists $L_\alpha > 0$, independent of $m$, such that
    \begin{align*}
        \tnorm{
            \boldsymbol{\alpha}^m(t_1, \x_1) - \boldsymbol{\alpha}^m(t_2, \x_2) }^2_{\R^{d \times N}}
        \leq 
        L_\alpha D_{t\x},
    \end{align*}
    and, for all $i\in \I$, $m\geq 1$, the functions $\kappa^{i,m}$ satisfy~\eqref{eq:lipschtiz_bounds}.
\end{lemma}
\begin{proof}
     We introduce the compact notation
\begin{align*}
    p^i_k = \bPhi(t_k, \x_k) \zeta^{i,m}(t_k, \x_k), \qquad \alpha^{i,m}_k = \alpha^{i,m}(t_k, \x_k), 
\end{align*}
for $k = 1,2$. We want to show that
\begin{align*}
    A_m \defeq \sum_{i \in \I} \| \alpha^{i,m}_1 - \alpha^{i,m}_2 \|^2_{\R^d} \leq L_\alpha D_{t\x},
\end{align*}
for all $m\geq 1$. Using the definition of $\alpha^{i,m}$, we have
\begin{align*}
    A_m
    &=
    \sum_{i \in \I}
    \big\| 
        \kappa^i
        \big(
            t_1, 
            \x_1, 
             p^i_1, 
            \boldsymbol{\alpha}^{-i,m-1}_1
        \big)
        -
        \kappa^i
        \big(
            t_2, 
            \x_2, 
             p^i_2, 
            \boldsymbol{\alpha}^{-i,m-1}_2
        \big)
    \big\|^2_{\R^d}
    \\
    &\leq 
    (1 + \eta)
    \sum_{i \in \I}
    \big\| 
        \kappa^i
        \big(
            t_1, 
            \x_1, 
             p^i_1, 
            \boldsymbol{\alpha}^{-i,m-1}_1
        \big)
        -
        \kappa^i
        \big(
            t_1, 
            \x_1, 
             p^i_1, 
            \boldsymbol{\alpha}^{-i,m-1}_2
        \big)
    \big\|^2_{\R^d}
    \\
    &\qquad
    +
    (1 + \eta^{-1})
    \sum_{i \in \I}
    \big\| 
        \kappa^i
        \big(
            t_1, 
            \x_1, 
             p^i_1, 
            \boldsymbol{\alpha}^{-i,m-1}_2
        \big)
        -
        \kappa^i
        \big(
            t_2, 
            \x_2, 
             p^i_2, 
            \boldsymbol{\alpha}^{-i,m-1}_2
        \big)
    \big\|^2_{\R^d},
\end{align*}
where the second step is Young's inequality. The first term is controlled by the joint contraction
\begin{align*}
    \sum_{i \in \I}
    \big\| 
        \kappa^i
        \big(
            t_1, 
            \x_1, 
             p^i_1, 
            \boldsymbol{\alpha}^{-i,m-1}_1
        \big)
        -
        \kappa^i
        \big(
            t_1, 
            \x_1, 
             p^i_1, 
            \boldsymbol{\alpha}^{-i,m-1}_2
        \big)
    \big\|^2_{\R^d}
    \leq
    L_\kappa^a 
    \sum_{i \in \I}
    \big\|
        \alpha^{i,m-1}_1
        -
        \alpha^{i,m-1}_2
    \big\|^2_{\R^d}
    =
    L_\kappa^a A_{m-1}.
\end{align*}
For the second term, the Lipschitz continuity of $\kappa^i$ gives
\begin{align*}
    \sum_{i \in \I}
    \big\| 
        \kappa^i
        \big(
            t_1, 
            \x_1, 
             p^i_1, 
            \boldsymbol{\alpha}^{-i,m-1}_2
        \big)
        -
        \kappa^i
        \big(
            t_2, 
            \x_2, 
             p^i_2, 
            \boldsymbol{\alpha}^{-i,m-1}_2
        \big)
    \big\|^2_{\R^d}
    \leq
    \sum_{i \in \I}
    L_{\kappa}\big(D_{t\x} + \big\| p^i_1 - p^i_2 \big\|_{\R^n}^2 \big).
\end{align*}
As in the proof of Lemma~\ref{lem:lipschitz_fb}, we get
\begin{align*}
    \big\| p^i_1 - p^i_2 \big\|_{\R^n}^2 
    &= 
    \big\| 
        \bPhi(t_1, \x_1) \zeta^{i,m}(t_1, \x_1) 
        -  
        \bPhi(t_2, \x_2) \zeta^{i,m}(t_2, \x_2) 
    \big\|_{\R^n}^2
    \\
    &\leq
    2C^2_{\bPhi}
    \big\|
         \zeta^{i,m}(t_1, \x_1)
         -
          \zeta^{i,m}(t_2, \x_2)
    \big\|^2_{\R^n}
    +
    2 L_{\bPhi} D_{t\x}
    \big\|  \zeta^{i,m}(t_2, \x_2) \big\|^2_{\R^n}.
\end{align*}
For the first term we use the uniform Lipschitz bound for $\zeta^{i,m}$. Under case (I), we have a uniform bound on $\zeta^{i,m}$ for the second term. Alternatively, under case (II), we have $L_{\bPhi} = 0$ and the second term vanishes completely. This gives the bound
\begin{align}
    \| p_1^i - p_2^i \|^2_{\R^n} \leq L_p D_{t\x}. \label{eq:phiz_bound}
\end{align}
Thus, the bound for $A_m$ becomes
\begin{align*}
    A_m \leq (1+\eta)L_\kappa^a A_{m-1} + (1+\eta^{-1})L_\kappa \sum_{i\in\I} \big( D_{t\x} + L_p D_{t\x} \big) = q A_{m-1} + BD_{t\x},
\end{align*}
where we have defined
\begin{align*}
    q \defeq  (1+\eta)L_\kappa^a, \qquad B \defeq (1+\eta^{-1})L_\kappa N(1+L_p),
\end{align*}
and choose $\eta$ such that $q < 1$. Iterating the bound, we get
\begin{align*}
    A_m &\leq q^m A_0 + \Bigg( \sum_{r=0}^{m-1} q^r \Bigg) B D_{t\x} \leq \frac{B}{1-q} D_{t\x},
\end{align*}
where in the last inequality we use $A_0=0$ by zero-policy initialization, and the geometric sum formula. We thus have Lipschitz continuity of $\boldsymbol{\alpha}^{m}$ for all $m$, with 
\begin{align*}
    L_\alpha \defeq  \frac{B}{1-q}.
\end{align*}
This concludes the $\boldsymbol{\alpha}^m$-part of the lemma. The Lipschitz continuity of $\kappa^{i,m}$ now follows by using the Lipschitz continuity of $\kappa^i$, the same type of estimate for the $\bPhi(t_k,\x_k)z_k$ terms as above, and the Lipschitz continuity of $\boldsymbol{\alpha}^m$,
\begin{align*}
    \big\|
        \kappa^{i,m}(t_1, \x_1, z_1)
        -
        \kappa^{i,m}(t_2, \x_2, z_2)
    \big\|^2_{\R^d} 
    &\leq
    \big\|
        \kappa^{i}(t_1, \x_1, \bPhi_1z_1, \boldsymbol{\alpha}^{-i,m-1}_1)
        -
        \kappa^{i}(t_2, \x_2, \bPhi_2z_2, \boldsymbol{\alpha}^{-i,m-1}_2)
    \big\|^2_{\R^d}
    \\
    &\leq
    L_{\kappa^i}
    \big(
        D_{t\x}
        +
        \big\|
            \bPhi_1 z_1
            -
            \bPhi_2 z_2
        \big\|_{\R^n}^2
    \big)
    \!+\!
    L_\kappa^{a}
    \btnorm{
        \boldsymbol{\alpha}^{-i,m-1}_1
        -
        \boldsymbol{\alpha}^{-i,m-1}_2}_{\R^{d\times (N-1)}}^2
    \\
    &\leq
    L_{\kappa}^{\FP}
    \big(
        D_{t\x}
        +
        D_{z}
    \big).
\end{align*}
\end{proof}
It remains to show the Lipschitz continuity of $\bb^{i,m}$ and $\ell^{i,m}$. This requires Lipschitz continuity of the coefficient functions $\tilde{\bb}^i$ and $\tilde{\ell}^i$, which follows under similar calculations and assumptions as for $\bb^i$ and $\ell^i$ in Lemma~\ref{lem:lipschitz_fb}. The following lemma concludes the appendix.
\begin{lemma}
    If case (I) or (II) holds, then the functions $\bb^{i,m}$ and $\ell^{i,m}$ satisfy~\eqref{eq:lipschtiz_bounds}.
\end{lemma}
\begin{proof}
Using the definition of $\bb^{i,m}$ and the Lipschitz continuity of $\kappa^{i,m}$, we obtain
\begin{align*}
    &\big\| 
        \bb^{i,m}(t_1, \x_1, z^i_1, \z_1^{-i})
        -
        \bb^{i,m}(t_2, \x_2, z^i_2, \z_2^{-i})
    \big\|^2
    \\
    &\qquad
    =
    \big\| 
        \tilde{\bb}^{i}(t_1, \x_1, z^i_1, \boldsymbol{\kappa}^{-i,m-1}(t_1, \x_1, \boldsymbol{z}_1^{-i}))
        -
        \tilde{\bb}^{i}(t_2, \x_2, z^i_2,\boldsymbol{\kappa}^{-i,m-1}(t_2, \x_2, \boldsymbol{z}_2^{-i}))
    \big\|^2
    \\
    &\qquad
    \leq 
    L_{\tilde{\bb}^i}
    \bigg(
        D_{t\x}
        +
        \| z^i_1 - z_2^i \|^2_{\R^n}
        +
        L_\kappa^{\FP}
        \bigg(
            (N-1) D_{t\x} 
            + 
            \sum_{i\in\I\backslash\{i\}} \| z^j_1 - z^j_2 \|^2_{\R^n} 
        \bigg)
    \bigg)
    \\
    &\qquad
    \leq
    L^{\FP}_{\bb}
    \left(
        D_{t\x}
        +
        D_{\z}
    \right).
\end{align*}
The same argument gives Lipschitz continuity for $\ell^{i,m}$.
\end{proof}

\section{Synthesis of the LQ-problem}\label{App:LQ_riccati}

This section contains a detailed derivation of the analytical solution to the LQG problem presented in Example~\ref{ex:LQ}. To lighten the notation, the explicit dependence on $(t,\x)$ is suppressed. The value function $\valueFunctionPlayeri$ satisfies the HJB equation
\begin{align}
\begin{split}
        \partial_t \valueFunctionPlayeri 
    &+  
        \frac{1}{2} \Tr 
        \left( 
            \diffusionMatrix 
            \hessian 
            \valueFunctionPlayeri 
            \diffusionMatrix^\transpose 
        \right)
    +   
        \left\langle 
            \attractionMatrix 
            \left( 
                \boldsymbol{\xi} - \x 
            \right),
            \gradient
            \valueFunctionPlayeri
        \right\rangle
    +   
        \big\| 
            \stateCostPlayeri
            \x
        \big\|_{\mathbb R^n}^2
    \\
    &\quad
    +  
        \sumAllPlayersj 
        \left\langle
            \Phij,
            \left(
                \controlMatrixPlayerj
            \right)^\transpose
            \gradient
            \valueFunctionPlayeri
        \right\rangle
    +
        \frac{1}{2}
        \sumAllPlayersj
        \left\langle
            \controlCostPlayerij
            \Phij,
            \Phij
        \right\rangle
    +   
        \sumAllPlayersj
        \left\langle
            \controlstateCostPlayerij
            \Phij,
            \x
        \right\rangle
    = 0,
\end{split}\label{eq:HJB_LQG_Example_new}
\end{align}
where we have used the convenient notation
\begin{align}
    \Phii
    =
    -   
        \left(
            \controlCostPlayerii
        \right)^{-1}
        \left(
            \left(
                \controlMatrixPlayeri
            \right)^\transpose
            \gradient
            \valueFunctionPlayeri
        +   
            \left(
                \controlstateCostPlayerii
            \right)^\transpose
            \x
        \right).
    \label{eq:Phii_new}
\end{align}
We make the quadratic ansatz
\begin{align}
    \valueFunctionPlayeri(t,\x)
    =
        \frac{1}{2}
        \left\langle
            \x,
            P^i
            \x
        \right\rangle
    +
        \left\langle
            Q^i,
            \x
        \right\rangle
    +
        R^i.
    \label{eq:ex_ansatz_new}
\end{align}
Differentiating $\valueFunctionPlayeri$ in time and space, respectively, we get
\begin{align}
    \partial_t \valueFunctionPlayeri 
    &= 
        \frac{1}{2}
        \big\langle
            \x,
            \dot{P}^i
            \x
        \big\rangle
    +
        \big\langle
            \dot{Q}^i,
            \x
        \big\rangle
    +
        \dot{R}^i. 
    \\
    \gradient \valueFunctionPlayeri
    &=
         \PPshort^i \x 
    +
        Q^i, \label{eq:gradient_V}
    \\
    \hessian \valueFunctionPlayeri
    &=
        \PPshort^i,
\end{align}
where the notation $\PPshort^i = \tfrac{1}{2} \big( P^i + (P^i)^\transpose\big)$ is used for brevity. Next, we evaluate each term in~\eqref{eq:HJB_LQG_Example_new} by explicitly inserting~\eqref{eq:Phii_new} followed by the ansatz~\eqref{eq:ex_ansatz_new}. The first term immediately becomes
\begin{align*}
    \frac{1}{2} \Tr 
        \left( 
            \diffusionMatrix 
            \hessian 
            \valueFunctionPlayeri 
            \diffusionMatrix^\transpose 
        \right)
    =
        \frac{1}{2} \Tr 
        \big( 
            \diffusionMatrix 
            \PPshort^i
            \diffusionMatrix^\transpose 
        \big).
\end{align*}
Inserting~\eqref{eq:gradient_V} into the second term and rearranging terms, we get 
\begin{align*}
    \left\langle 
            \attractionMatrix 
            \left( 
                \boldsymbol{\xi} - \x 
            \right),
            \gradient
            \valueFunctionPlayeri
    \right\rangle
    &=
    -
        \left\langle 
                \x,
                \attractionMatrix^\transpose
                 \PPshort^i 
                 \x 
        \right\rangle
    +
        \left\langle 
                 \x,
                 \PPshort^i 
                \attractionMatrix 
                \boldsymbol{\xi}
            -
                \attractionMatrix^\transpose
                Q^i
        \right\rangle
    +
        \left\langle 
                \attractionMatrix 
                \boldsymbol{\xi},
                Q^i
        \right\rangle.
\end{align*}
For the first sum in~\eqref{eq:HJB_LQG_Example_new}, note that each term can be written as
\begin{align*}
    \left\langle
        \Phij,
        \left(
            \controlMatrixPlayerj
        \right)^\transpose\!
        \gradient
        \valueFunctionPlayeri
    \right\rangle
    &=
    -   
        \left\langle
            \gradient
            \valueFunctionPlayeri,
            \controlMatrixPlayerj
            \left(
                \controlCostPlayerjj
            \right)^{-1}
            \left(
                \left(
                    \controlMatrixPlayerj
                \right)^\transpose
                \gradient
                \valueFunctionPlayerj
            +   
                \left(
                    \controlstateCostPlayerjj
                \right)^\transpose
                \x
            \right)
        \right\rangle
    \\
    &=
        \left\langle
            \gradient
            \valueFunctionPlayeri,
            \ThetaOne
            \gradient
            \valueFunctionPlayerj
        +
            \ThetaTwo
            \x
        \right\rangle
    \\
    &=
        \left\langle\!
            \x,\!
            \big(
                \PPshort^i
                \ThetaOne
                \PPshort^j
            \!
            +
            \!
                \PPshort^i
                \ThetaTwo
            \big)
            \x\!
        \right\rangle
    \!
    +
    \!
        \left\langle\!
            \x,\!
            \PPshort^i
            \ThetaOne
            Q^j
        \!
        +
        \!
            \PPshort^j
            \big(
                \ThetaOne
            \big)^\transpose\!
            Q^i
        \!
        +
        \!
            \big(
                \ThetaTwo
            \big)^\transpose\!
            Q^i\!
        \right\rangle
    \!
     +
     \!
        \left\langle\!
            Q^i,
            \ThetaOne
            Q^j\!
        \right\rangle,
\end{align*}
where we introduced the notation
\begin{align*}
    \ThetaOne 
    &\defeq 
    -
        \controlMatrixPlayerj
        \left(
            \controlCostPlayerjj
        \right)^{-1}
        \left(
            \controlMatrixPlayerj
        \right)^\transpose,
    \quad
    \ThetaTwo
    \defeq
    -
        \controlMatrixPlayerj
        \left(
            \controlCostPlayerjj
        \right)^{-1}
        \left(
            \controlstateCostPlayerjj
        \right)^\transpose.
\end{align*}
Next, we define
\begin{align*}
    \rho^{i,j}
    \defeq
    \tfrac{1}{2}
    \big(
        \controlCostPlayerjj
    \big)^{-\transpose}
    \controlCostPlayerij
    \big(
        \controlCostPlayerjj
    \big)^{-1}
\end{align*}
For each term in the second sum of~\eqref{eq:HJB_LQG_Example_new}, we insert the expressions for $\Phij$ and $\gradient \valueFunctionPlayerj$, and rearrange terms to get
\begin{align*}
    \frac{1}{2}
    \left\langle
        \controlCostPlayerij
        \Phij,
        \Phij
    \right\rangle
    &=
        \left\langle
            \gradient
            \valueFunctionPlayerj,
            \controlMatrixPlayerj
            \rho^{i,j}
            \left(
                \controlMatrixPlayerj
            \right)^\transpose
            \gradient
            \valueFunctionPlayerj
        \right\rangle
    +
        \left\langle
            \x,
            \controlstateCostPlayerjj
            \rho^{i,j}
            \left(
                \controlMatrixPlayerj
            \right)^\transpose
            \gradient
            \valueFunctionPlayerj
        \right\rangle
    \\
    &\hspace{0.5cm}
    +
        \left\langle
            \gradient
            \valueFunctionPlayerj,
            \controlMatrixPlayerj
            \rho^{i,j}
            \left(
                \controlstateCostPlayerjj
            \right)^\transpose
            \x
        \right\rangle
    +
        \left\langle
            \x,
            \controlstateCostPlayerjj
            \rho^{i,j}
            \left(
                \controlstateCostPlayerjj
            \right)^\transpose
            \x
        \right\rangle
    \\
    &=
        \left\langle
            \gradient
            \valueFunctionPlayerj,
            \ThetaThree
            \gradient
            \valueFunctionPlayerj
        \right\rangle
    +
        \left\langle
            \x,
            \ThetaFour
            \gradient
            \valueFunctionPlayerj
        \right\rangle
    +
        \left\langle
            \gradient
            \valueFunctionPlayerj,
            \ThetaFive
            \x
        \right\rangle
    +
        \left\langle
            \x,
            \ThetaSix
            \x
        \right\rangle
    \\
    &=
        \left\langle
            \x,
            \big(
                \PPshort^j
                \ThetaThree
                \PPshort^j
            +
                \ThetaFour
                \PPshort^j
            +
                \PPshort^j
                \ThetaFive
            +
                \ThetaSix
            \big)
            \x
        \right\rangle
    \\
    &\hspace{0.5cm}
    +
        \left\langle
            \x,
            \PPshort^j
            \big(
                \ThetaThree
            \big)^\transpose
            Q^j
        +
            \PPshort^j
            \ThetaThree
            Q^j
        +
            \ThetaFour
            Q^j
        +
            \big(
                \ThetaFive
            \big)^\transpose
            Q^j
        \right\rangle
    +
        \left\langle
            Q^j,
            \ThetaThree
            Q^j
        \right\rangle.
\end{align*}
The matrices introduced in the second equality are defined as
\begin{align*}
    \ThetaThree
    \defeq
        \controlMatrixPlayerj
        \rho^{i,j}
        \left(
            \controlMatrixPlayerj
        \right)^\transpose,
    \quad
    \ThetaFour
    \defeq
        \controlstateCostPlayerjj
        \rho^{i,j}
        \left(
            \controlMatrixPlayerj
        \right)^\transpose,
    \quad
    \ThetaFive
    \defeq
        \controlMatrixPlayerj
        \rho^{i,j}
        \left(
            \controlstateCostPlayerjj
        \right)^\transpose,
    \quad
    \ThetaSix
    \defeq
        \controlstateCostPlayerjj
        \rho^{i,j}
        \left(
            \controlstateCostPlayerjj
        \right)^\transpose.
\end{align*}
Each term in the final sum of~\eqref{eq:HJB_LQG_Example_new} can be written as
\begin{align*}
    \left\langle
        \controlstateCostPlayerij
        \Phij,
        \x
    \right\rangle
    &=
    -
        \left\langle
            \x,
            \controlstateCostPlayerij
            \left(
                \controlCostPlayerjj
            \right)^{-1}
            \left(
                \left(
                    \controlMatrixPlayerj
                \right)^\transpose
                \gradient
                \valueFunctionPlayerj
            +   
                \left(
                    \controlstateCostPlayerjj
                \right)^\transpose
                \x
            \right)
        \right\rangle
    =
        \left\langle
            \x,
            \ThetaSeven
            \gradient
            \valueFunctionPlayerj
        +
            \ThetaEight
            \x
        \right\rangle
    \\
    &=
        \left\langle
            \x,
            \big(
                \ThetaSeven
                \PPshort^j
            +
                \ThetaEight
            \big)
            \x
        \right\rangle
    +
        \left\langle
            \x,
            \ThetaSeven
            Q^j
        \right\rangle,
\end{align*}
where we have introduced the matrices
\begin{align*}
    \ThetaSeven
    &\defeq
    -
        \controlstateCostPlayerij
        \left(
            \controlCostPlayerjj
        \right)^{-1}
        \left(
            \controlMatrixPlayerj
        \right)^\transpose,
    \quad
    \ThetaEight
    \defeq
    -
        \controlstateCostPlayerij
        \left(
            \controlCostPlayerjj
        \right)^{-1}
        \left(
            \controlstateCostPlayerjj
        \right)^\transpose.
\end{align*}
Now, we combine the earlier defined matrices according to
\begin{align*}
    \ThetaThreeTilde \defeq \ThetaThree + \big( \ThetaThree \big)^\top, \quad \ThetaFourFiveSeven \defeq \ThetaFour + \big( \ThetaFive \big)^\top + \ThetaSeven, \quad \ThetaSixEight \defeq \ThetaSix + \ThetaEight.
\end{align*}
Inserting the derived expressions into~\eqref{eq:HJB_LQG_Example_new}, we then have
\begin{align*}
    0 &= 
        \Big\langle
            \x,
            \Big(
                \tfrac{1}{2}
                \dot{P}^i
            -
                \attractionMatrix^\transpose
                \PPshort^i
            +
                (\stateCostPlayeri)^\transpose\stateCostPlayeri
            +
                \sumAllPlayersj
                \big(
                    \PPshort^i
                    (
                        \ThetaOne
                        \PPshort^j
                    +
                        \ThetaTwo
                    )
                +
                    \PPshort^j
                    \ThetaThree
                    \PPshort^j
                +
                    \ThetaFourFiveSeven
                    \PPshort^j
                +   
                    \ThetaSixEight
                \big)
            \Big)
            \x
        \Big\rangle
    \\
    &
    \hspace{1cm}
    +
        \Big\langle
            \x,
            \dot{Q}^i
        +
            \PPshort^i
            \attractionMatrix
            \boldsymbol{\xi}
        -
            \attractionMatrix^\transpose
            Q^i
        +
        \sumAllPlayersj
        \Big(
            \PPshort^i
            \ThetaOne
            Q^j
        +
            \PPshort^j
            \big(
                \ThetaOne
            \big)^\transpose
            Q^i
        +
            \big(
                \ThetaTwo
            \big)^\transpose
            Q^i
        +  
            \PPshort^j
            \ThetaThreeTilde
            Q^j
        +
            \ThetaFourFiveSeven
            Q^j
        \Big)
        \Big\rangle
    \\
    &
    \hspace{1cm}
    +
        \dot{R}^i
    +
        \frac{1}{2} \Tr 
        \left( 
            \diffusionMatrix 
            \PPshort^i
            \diffusionMatrix^\transpose 
        \right)
    +
        \left\langle 
                \attractionMatrix 
                \boldsymbol{\xi},
                Q^i
        \right\rangle
    +
        \sumAllPlayersj
        \big(
            \big\langle
                Q^i,
                \ThetaOne
                Q^j
            \big\rangle
        +
            \big\langle
                Q^j,
                \ThetaThree
                Q^j
            \big\rangle
        \big).
\end{align*}
Setting each inner product to zero, and using the fact that
\begin{align*}
    \big\langle
        \x,
        2
        \attractionMatrix^\transpose
        \PPshort^i
        \x
    \big\rangle
    &=
    \big\langle
        \x,
        \big(
            \attractionMatrix^\transpose
            P^i
        +   
            P^i
            \attractionMatrix
        \big)
        \x
    \big\rangle,
\end{align*}
yields the Riccati system~\eqref{eq:Riccati_system}, that is
\begin{align}
\begin{split}\label{eq:Riccati_system_derived}
        \dot{P}^i
    &=
        \attractionMatrix^\transpose
        P^i
    +   
        P^i
        \attractionMatrix
    -
        (\stateCostPlayeri)^\transpose\stateCostPlayeri
    -
        2
        \sumAllPlayersj
        \Big(
            \PPshort^i
            \big(
                \ThetaOne
                \PPshort^j
            +
                \ThetaTwo       
            \big)
        +
            \PPshort^j
            \ThetaThree
            \PPshort^j
        +
            \ThetaFourFiveSeven
            \PPshort^j
        +   
            \ThetaSixEight
        \Big)
    \\
    \dot{Q}^i
    &=
    -
        \PPshort^i
        \attractionMatrix
        \boldsymbol{\xi}
    +
        \attractionMatrix^\transpose
        Q^i
    -
        \sumAllPlayersj
        \Big(
            \PPshort^i
            \ThetaOne
            Q^j
        +
            \PPshort^j
            \big(
                \ThetaOne
            \big)^\transpose
            Q^i
        +
            \big(
                \ThetaTwo
            \big)^\transpose
            Q^i
        +  
            \PPshort^j
            \ThetaThreeTilde
            Q^j
        +
            \ThetaFourFiveSeven
            Q^j
        \Big)
    \\
    \dot{R}^i
    &=
    -
        \frac{1}{2} \Tr 
        \left( 
            \diffusionMatrix 
            \PPshort^i
            \diffusionMatrix^\transpose 
        \right)
    -
        \left\langle 
                \attractionMatrix 
                \boldsymbol{\xi},
                Q^i
        \right\rangle
    -
        \sumAllPlayersj
        \Big(
            \big\langle
                Q^i,
                \ThetaOne
                Q^j
            \big\rangle
        +
            \big\langle
                Q^j,
                \ThetaThree
                Q^j
            \big\rangle
        \Big)
\end{split}
\end{align}
where the $\Theta$-matrices are
\begin{align*}
    \ThetaOne 
    &= 
    -
        \controlMatrixPlayerj
        \left(
            \controlCostPlayerjj
        \right)^{-1}
        \left(
            \controlMatrixPlayerj
        \right)^\transpose
    , \\
    \ThetaTwo
    &=
    -
        \controlMatrixPlayerj
        \left(
            \controlCostPlayerjj
        \right)^{-1}
        \left(
            \controlstateCostPlayerjj
        \right)^\transpose,
    \\
    \ThetaThree
    &= 
        \tfrac{1}{2}
        \controlMatrixPlayerj   
        \left(
            \controlCostPlayerjj
        \right)^{-\transpose}
        \controlCostPlayerij
        \left(
            \controlCostPlayerjj
        \right)^{-1}
        \left(
            \controlMatrixPlayerj
        \right)^\transpose,
    \\
    \ThetaThreeTilde
    &=
        \tfrac{1}{2}
        \controlMatrixPlayerj   
        \left(
            \controlCostPlayerjj
        \right)^{-\transpose}
        \Big(
            \left(
                \controlCostPlayerij
            \right)^\transpose
        +
            \controlCostPlayerij
        \Big)
        \left(
            \controlCostPlayerjj
        \right)^{-1}
        \left(
            \controlMatrixPlayerj
        \right)^\transpose,
    \\
    \ThetaFourFiveSeven
    &=
        \Big(
            \tfrac{1}{2}
            \controlstateCostPlayerjj
            \big(
                \controlCostPlayerjj
            \big)^{-\transpose}
            \Big(
                \big(
                    \controlCostPlayerij
                \big)^\transpose
            +
                \controlCostPlayerij
            \Big)
        -   
            \controlstateCostPlayerij
        \Big)
        \left(
            \controlCostPlayerjj
        \right)^{-1}
        \left(
            \controlMatrixPlayerj
        \right)^\transpose,
    \\
    \ThetaSixEight
    &=
        \Big(
            \tfrac{1}{2}
            \controlstateCostPlayerjj
            \left(
                \controlCostPlayerjj
            \right)^{-\transpose}
            \controlCostPlayerij
        -
            \controlstateCostPlayerij
        \Big)
        \left(
            \controlCostPlayerjj
        \right)^{-1}
        \left(
            \controlstateCostPlayerjj
        \right)^\transpose.
\end{align*}

\begin{remark}
We consider the single-player case, i.e., $\I = \{1\}$, and show that the derived system corresponds to the generalized Riccati differential equation (see, e.g.,~\cite[Remark~2.1]{FerN13}). Since there is only one player, we can remove the superscript indices of all variables. Assuming $P$ and $K_a$ symmetric, we have
\begin{align} 
    \dot{P}
    &=
        \attractionMatrix^\transpose
        P
    +   
        P
        \attractionMatrix
    -
         K_{\x}^\transpose K_{\x}
    -
        2
        \Big(
            P
            \big(
                \Theta_1
            +
                \Theta_3
            \big)
            P
        +
            P
            \Theta_2
        +
            \Theta_{4,5,7}
            P
        +   
            \Theta_{6,8}
        \Big). \label{eq:P_equation_single_agent}
\end{align}
In this setting, the matrices can be simplified according to
\begin{align*}
    \Theta_1
+
    \Theta_3
    &=
    -
        B
        K_{a}^{-1}
        B^\transpose
    +
        \tfrac{1}{2}
        B   
        K_{a}^{-\transpose}
        K_{a}
        K_{a}^{-1}
        B^\transpose=
    -
        \tfrac{1}{2}
        B
        K_{a}^{-1}
        B^\transpose, \\
    \Theta_2 
    &=
    -
        B
        K_a^{-1}
        K_{a\x}^\top, \\
    \Theta_{4,5,7}
    &=
        \Big(
            \tfrac{1}{2}
            K_{a\x}
            K_{a}^{-\transpose}
            \Big(
                K_{a}^\transpose
            +
                K_{a}
            \Big)
        -   
            K_{a\x}
        \Big)
        K_{a}^{-1}
        B^\transpose
    =
    0, \\
    \Theta_{6,8}
    &=
        \Big(
            \frac{1}{2}
            K_{a\x}
            K_{a}^{-\transpose}
            K_{a}
        -
            K_{a\x}
        \Big)
        K_{a}^{-1}
        K_{a\x}^\transpose
    =
        -
        \frac{1}{2}
        K_{a\x}
        K_{a}^{-1}
        K_{a\x}^\transpose
\end{align*}
Inserting these expressions into~\eqref{eq:P_equation_single_agent}, we obtain the single agent Riccati equation
\begin{align}
    \dot{P}
    &=
        \attractionMatrix^\transpose
        P
    +   
        P
        \attractionMatrix
    -
        K_{\x}^\transpose
        K_{\x}
    +
        P B K_{a}^{-1} B^\transpose P
    +
        2 P B K_{a}^{-1} K_{a\x}^\transpose
    +
        K_{a\x} K_{a}^{-1} K_{a\x}^\transpose.
\end{align}
Compare this to the generalized Riccati differential equation (see, e.g.,~\cite[Remark~2.1]{FerN13}), which can be written on the form
\begin{align}
    \dot{P}
    =
    -
    P A
    -
    A^\transpose P
    +
    P B R^{-1} B^\transpose P
    +
    S R^{-1} S^\transpose
    +
    2 S R^{-1} B^\transpose P
    -
    Q.
\end{align}
Note that, with $Q = K_{\x}^\transpose K_{\x}$, $S = K_{a\x}$ and $R=K_{a}$, the generalized Riccati equation coincides with our single-player Riccati equation up to the sign convention for the drift term.
\end{remark}

\bibliographystyle{plain}
\bibliography{references}

\newpage
\appendix

\end{document}

%% file: numplot_v2.tex

\begin{tikzpicture}

\definecolor{darkgray176}{RGB}{176,176,176}
\definecolor{darkorange25512714}{RGB}{255,127,14}
\definecolor{forestgreen4416044}{RGB}{44,160,44}
\definecolor{lightgray204}{RGB}{204,204,204}
\definecolor{steelblue31119180}{RGB}{31,119,180}

\begin{groupplot}[
group style={group size=2 by 1, horizontal sep=1.8cm},
width=0.50\textwidth,
height=0.33\textwidth,
log basis y={10},
tick align=outside,
tick pos=left,
tick label style={font=\footnotesize},
x grid style={darkgray176},
xlabel={Iterations (\(\displaystyle m\))},
xmajorgrids,
xmin=-0.75, xmax=15.75,
xminorgrids,
xtick style={color=black},
y grid style={darkgray176},
ymajorgrids,
ymode=log,
ytick style={color=black}
]

\nextgroupplot[
title={$N=2$},
ymin=1e-21, ymax=1e2,
ytick={1e0,1e-4,1e-8,1e-12,1e-16,1e-20},
yminorgrids=false,
legend cell align={left},
legend style={
    font=\footnotesize,
    fill=white,
    draw=black,
    inner xsep=4pt,
    inner ysep=4pt,
    nodes={inner ysep=1.8pt},
    at={(0.98,0.98)},
    anchor=north east
}
]
\addplot [semithick, steelblue31119180, mark=*, mark size=2.2, mark options={solid}]
table {%
0 5.26518998e-02
1 2.91803156e-03
2 1.35161139e-04
3 6.36637624e-06
4 2.87064030e-07
5 1.21582812e-08
6 4.73964004e-10
7 1.67816136e-11
8 5.35374750e-13
9 1.53312608e-14
10 3.95362776e-16
11 9.48388226e-18
12 2.05702268e-19
13 1.55745364e-19
14 1.56891021e-19
15 1.56733102e-19
};
\addlegendentry{$\widehat{\mathrm{Err}}_{X}$}

\addplot [semithick, darkorange25512714, mark=square*, mark size=2.2, mark options={solid}]
table {%
0 1.96616004e+00
1 7.19235153e-02
2 3.17518175e-03
3 1.32575672e-04
4 5.46377468e-06
5 2.17913544e-07
6 8.32900483e-09
7 3.01116867e-10
8 1.01577315e-11
9 3.18748156e-13
10 8.83600965e-15
11 2.81613432e-16
12 2.97812829e-18
13 3.05740415e-18
14 3.04723935e-18
15 3.04847339e-18
};
\addlegendentry{$\widehat{\mathrm{Err}}_{Y}$}

\addplot [semithick, forestgreen4416044, mark=triangle*, mark size=2.2, mark options={solid}]
table {%
0 1.32360995e-01
1 6.41875065e-03
2 2.87054502e-04
3 1.29873939e-05
4 5.68571069e-07
5 2.37801194e-08
6 9.32308812e-10
7 3.38312427e-11
8 1.12733649e-12
9 3.43588314e-14
10 9.55438167e-16
11 2.47009805e-17
12 9.74256058e-19
13 4.34122684e-19
14 4.21595327e-19
15 4.21488890e-19
};
\addlegendentry{$\widehat{\mathrm{Err}}_{Z}$}

\nextgroupplot[
title={$N=20$},
ymin=1e-21, ymax=1e2,
ytick={1e0,1e-4,1e-8,1e-12,1e-16,1e-20},
yminorgrids=false,
legend cell align={left},
legend style={
    font=\footnotesize,
    fill=white,
    draw=black,
    inner xsep=4pt,
    inner ysep=4pt,
    nodes={inner ysep=1.8pt},
    at={(0.98,0.98)},
    anchor=north east
}
]
\addplot [semithick, steelblue31119180, mark=*, mark size=2.2, mark options={solid}]
table {%
0 5.69941981e-03
1 4.98405524e-06
2 4.23443470e-09
3 3.59773474e-12
4 3.04882144e-15
5 2.56829633e-18
6 9.88178572e-19
7 9.87746654e-19
8 9.87765633e-19
9 9.87799618e-19
10 9.87806680e-19
11 9.87807563e-19
12 9.87830956e-19
13 9.87823452e-19
14 9.87792998e-19
15 9.87789025e-19
};
\addlegendentry{$\widehat{\mathrm{Err}}_{X}$}

\addplot [semithick, darkorange25512714, mark=square*, mark size=2.2, mark options={solid}]
table {%
0 1.86579279e+00
1 1.53548787e-03
2 1.26345971e-06
3 1.03882820e-09
4 8.47524442e-13
5 8.55239484e-16
6 2.54191695e-17
7 2.54363850e-17
8 2.54364105e-17
9 2.54373186e-17
10 2.54371567e-17
11 2.54374328e-17
12 2.54380101e-17
13 2.54377881e-17
14 2.54375210e-17
15 2.54370028e-17
};
\addlegendentry{$\widehat{\mathrm{Err}}_{Y}$}

\addplot [semithick, forestgreen4416044, mark=triangle*, mark size=2.2, mark options={solid}]
table {%
0 6.86075718e-02
1 5.71117306e-05
2 4.67930274e-08
3 3.82928461e-11
4 3.12338336e-14
5 3.51553933e-17
6 9.06380349e-18
7 9.05345593e-18
8 9.05331226e-18
9 9.05358229e-18
10 9.05358141e-18
11 9.05364694e-18
12 9.05384416e-18
13 9.05378129e-18
14 9.05354495e-18
15 9.05350262e-18
};
\addlegendentry{$\widehat{\mathrm{Err}}_{Z}$}

\end{groupplot}

\end{tikzpicture}